\documentclass[reqno]{amsart}
\usepackage{amsmath}
\usepackage{comment}
\usepackage{amssymb}
\usepackage{hyperref}
\usepackage{tipa}
\hypersetup{
    colorlinks,
    linkcolor={red!70!black},
    citecolor={blue!70!black},
    urlcolor={blue!90!black}
}
\usepackage{esint}
\usepackage{mathtools}
\usepackage[utf8]{inputenc}
\usepackage{fourier}
\usepackage{braket}

\newcommand{\ud}{\,\mathrm{d}}
\newcommand{\paramY}{\langle Y\rangle}
\newcommand{\from}{\colon}
\newcommand{\symdiff}{\mathbin{\triangle}}
\newcommand{\WR}{W_\mathsf{L}}

\newcommand{\XR}{X_\mathsf{R}}
\newcommand{\XL}{X_\mathsf{L}}
\newcommand{\YR}{Y_\mathsf{R}}
\newcommand{\YL}{Y_\mathsf{L}}
\newcommand{\Rz}{\mathsf{Rsz}}
\newcommand{\R}{\mathbb{R}}
\newcommand{\N}{\mathbb{N}}
\newcommand{\Z}{\mathbb{Z}}
\newcommand{\HH}{\mathbb{H}}
\newcommand{\cB}{\mathcal{B}}
\newcommand{\cC}{\mathcal{C}}
\newcommand{\cD}{\mathcal{D}}
\newcommand{\cM}{\mathcal{M}}
\newcommand{\cQ}{\mathcal{Q}}
\newcommand{\cH}{\mathcal{H}}
\newcommand{\cL}{\mathcal{L}}
\newcommand{\cS}{\mathcal{S}}
\newcommand{\cT}{\mathcal{T}}
\newcommand{\C}{\mathbb{C}}
\newcommand{\zero}{\mathbf{0}}
\newcommand{\one}{\mathbf{1}}
\newcommand{\Gsm}{G^{\mathrm{sm}}}
\newcommand{\Glg}{G^{\mathrm{lg}}}

\DeclareMathOperator{\Cone}{Cone}
\DeclareMathOperator{\supp}{supp}
\DeclareMathOperator{\spt}{spt}
\DeclareMathOperator{\diam}{diam}
\newcommand{\Kor}{\mathrm{Kor}}
\DeclareMathOperator{\Lip}{Lip}
\DeclareMathOperator{\id}{id}
\DeclareMathOperator{\slope}{slope}
\DeclareMathOperator{\interior}{int}

\DeclareMathOperator{\pv}{p.v.}
\newcommand{\sE}{\mathsf{e}}
\newcommand{\sO}{\mathsf{o}}
\newcommand{\sml}{\mathrm{sm}}
\newcommand{\lrg}{\mathrm{lg}}
\newcommand{\Aff}{\mathsf{Aff}}

\newcommand{\rI}{\textup{I}}
\newcommand{\rII}{\textup{II}}

\newtheorem{thm}{Theorem}[section]
\newtheorem{prop}[thm]{Proposition}
\newtheorem{lemma}[thm]{Lemma}
\newtheorem{cor}[thm]{Corollary}
\theoremstyle{remark}

\newtheorem{question}[thm]{Question}

\usepackage{pgfornament}

\title{The Riesz transform on intrinsic Lipschitz graphs in the Heisenberg group}

\author{Vasileios Chousionis}
\address{Department of Mathematics, University of Connecticut}
\email{vasileios.chousionis@uconn.edu}

\author{Sean Li}
\address{Department of Mathematics, University of Connecticut}
\email{sean.li@uconn.edu}

\author{Robert Young}
\address{Courant Institute of Mathematical Sciences, New York University}
\email{ryoung@cims.nyu.edu}

\thanks{V.~C.\ was supported by Simons Foundation Collaboration grant 521845. R.~Y.\ was supported by NSF grant 2005609}
\date{\today}

\begin{document}

\maketitle
\begin{abstract}
We prove that the Heisenberg Riesz transform is $L_2$--unbounded on a family of intrinsic Lipschitz graphs in the first Heisenberg group $\HH$. We construct this family by combining a method from \cite{NY2} with a stopping time argument, and we establish the $L_2$--unboundedness of the Riesz transform by 
introducing several new techniques to analyze singular integrals on intrinsic Lipschitz graphs.
These include a formula for the Riesz transform in terms of a singular integral on a vertical plane and bounds on the flow of singular integrals that arises from a perturbation of a graph. On the way, we use our construction to show that the strong geometric lemma fails in $\HH$ for all exponents in $[2,4)$. 

Our results are in stark contrast to two fundamental results in Euclidean harmonic analysis and geometric measure theory: Lipschitz graphs in $\R^n$ satisfy the strong geometric lemma, and the $m$--Riesz transform is $L_2$--bounded on $m$--dimensional Lipschitz graphs in $\mathbb{R}^n$ for $m\in (0,n)$.
\end{abstract}

\tableofcontents

\section{Introduction}

Given a Radon measure $\nu$ in $\R^n$, the $m$--dimensional Riesz transform is formally defined by
$$T^{m}\nu(x)=\int R_m(x-y) \ud \nu (y),$$
where $R_m(x)=x |x|^{-m-1}$ is the $m$--dimensional Riesz kernel. 
If $\Gamma \subset \R^n$ is an $m$--dimensional Lipschitz graph and $\nu_\Gamma=\mathcal{H}^m|_{\Gamma}$ is the restriction of the $m$--dimensional Hausdorff measure on $\Gamma$, then $$f \mapsto T^m[f \ud \nu_\Gamma]$$
defines a bounded operator in $L_2(\Gamma):=L_2(\mathcal{H}^m
|_{\Gamma})$. 
This fundamental result was first obtained by Calderon in \cite{cal77} for $1$--dimensional Lipschitz graphs in the complex plane with sufficiently small Lipschitz constant. (In this case the $1$--dimensional Riesz kernel $R_1$ essentially coincides with the Cauchy kernel $k(z)=z^{-1}, z \in \C$.) The restriction on the Lipschitz constant was removed a few years later by Coifman, McIntosh and Meyer \cite{cmm}. Finally, Coifman, David and Meyer \cite{cdm} proved that $T^m$ is bounded in $L_2(\Gamma)$ for all $m$--dimensional Lipschitz graphs $\Gamma$ by showing that the $m$--dimensional case can be reduced to the $1$--dimensional case via the method of rotations. 

The $L_2$--boundedness of Riesz transforms on Lipschitz graphs has been pivotal for the research program which started in the early 80s with the aim of relating the analytic behavior of singular integrals on subsets of $\R^n$ to the geometric structure of these sets. In particular, David and Semmes \cite{dsbook, DS1} developed the theory of uniform rectifiability hoping to characterize the $m$--Ahlfors regular sets $E \subset \R^n$ on which the Riesz transforms $T^m, m \in (0,n),$ are bounded in $L_2(E)$; uniformly rectifiable sets can be built out of Lipschitz graphs and can be approximated by Lipschitz graphs at most locations and scales. David proved in \cite{david88} that if $E$ is $m$--uniformly rectifiable  then $T^m$ is bounded in $L^2(E)$. He and Semmes \cite{DS1} conjectured that the converse is also true. That is, if $E$ is an $m$--Ahlfors regular set such $T^m$ is bounded in $L^2(E)$ then $E$ is $m$--uniformly rectifiable. The conjecture was proved by Mattila, Melnikov and Verdera in \cite{mmv} for $m=1$ and by Nazarov, Tolsa and Volberg \cite{ntv} for $m=n-1$. It remains open for integers $m \in (1,n-1)$.

Riesz transforms have also played a crucial role in characterizing removable sets for Lipschitz harmonic functions. A compact set $E \subset \R^n$ is removable for Lipschitz harmonic functions  if  whenever $U \supset E$ is open and $f\from U \to \R$ is Lipschitz and harmonic in $U \setminus E$, then $f$ is harmonic in $U$. Uy \cite{uy} showed that if $\mathcal{H}^{n-1}(E)=0$ then $E$ is removable, while $\dim_{H}(E)>n-1$ implies that $E$ is not removable. 

Characterizing the removable sets $E$ with $\mathcal{H}^{n-1}(E)>0$ involves the Riesz transform $T^{n-1}$. If $E$ is $(n-1)$--upper regular and $T^{n-1}$ is bounded on $L_2(E)$ then $E$ is \textbf{not} removable for Lipschitz harmonic functions, see  \cite[Theorem 4.4]{MR1372240}. On the other hand, if $\mathcal{H}^{n-1}(E)<\infty$ and $E$ is not removable for Lipschitz harmonic functions, then there exists some Borel set $F \subset E$  with $\mathcal{H}^{n-1}(F)>0$ such that $T^{n-1}$ is bounded in $L_2(F)$, see \cite{volberg}. 

Due to important contributions from several people it is now known that a compact set $E\subset \R^n$ with $\mathcal{H}^{n-1}(E)>0$ is removable for Lipschitz harmonic functions if and only if $E$ is purely $(n-1)$--unrectifiable, that is, $E$ intersects every $\mathcal{C}^1$ hypersurface in a set of vanishing $(n-1)$--dimensional Hausdorff measure. One of the key ingredients in the proof of the  ``only if'' direction is the $L_2(\Gamma)$--boundedness of $T^{n-1}$ for Lipschitz graphs of codimension $1$. The harder ``if'' direction was proved by David and Mattila \cite{dm} (for $n=2$), and Nazarov, Tolsa and Volberg \cite{ntv, ntv2} for $n\geq 3$. We also mention that the $L_2$--boundedness of the Cauchy transform/$1$--dimensional Riesz transform was the key tool in geometrically characterizing removable sets for bounded analytic functions, see \cite{tolsabook, verdsurvey} for the long and interesting history of this problem.

There is a natural analogue of the codimension--$1$ Riesz kernel in the Heisenberg group $\HH$. Recall that in $\R^n$ the Riesz kernel $R_{n-1}(x):=x|x|^{-n}$ is a constant multiple of the gradient of the fundamental solution of the Laplacian. 
 Sub-Riemannian analogues of the Laplacian, known as sub-Laplacians, have been extensively studied in Carnot groups and sub-Riemannian manifolds since the early 70s and the works of Stein, Folland, and others \cite{stfol, fol, fol}. A thorough treatment of this fully-fledged theory can be found in \cite{BLU}. In particular, the (canonical) sub-Laplacian in $\HH$ is defined as 
$$\Delta_{\HH}=\XL^2+\YL^2,$$
where 
$$\XL f(h):=\frac{\partial f}{\partial x} (h)-\frac{1}{2}y(h) \frac{\partial f}{\partial z} (h)\mbox{ and }\YL f(h):=\frac{\partial f}{\partial y} (h)+\frac{1}{2}x(h) \frac{\partial f}{\partial z} (h)$$ are the left invariant vector fields which generate the horizontal distribution in $\HH$. By a classical result of Folland \cite{fol}, see also \cite[Example 5.4.7]{BLU}, the fundamental solution of $\Delta_{\HH}$ is $\|\cdot\|_{\Kor}^{-2}$ where $\|\cdot\|$ is the Koranyi norm in $\HH$. One then defines the Heisenberg Riesz kernel in $\HH$ as 
$$ \mathsf{R}(x):=\frac{\nabla_{\HH} \|x\|_{\Kor}^{-2}}{2}= \left( \frac{x (x^2+y^2)-4yz}{\|v\|_{\Kor}^6}, \frac{y(x^2+y^2)+4xz}{\|v\|_{\Kor}^6} \right),$$
where $\nabla_{\HH}f=(\XL f, \YL f)$ is the horizontal gradient in $\HH$. We note that $\mathsf{R}$ is a smooth, $(-3)$--homogenous, Calder\'on-Zygmund kernel, see Section \ref{sec:ker} for more details.

Given a Radon measure in $\HH$, the corresponding Heisenberg Riesz transform is the convolution-type singular integral formally defined by
$$T^{\mathsf{R}} \nu (p)= \int  \mathsf{R} (y^{-1}p) \ud \nu (y).$$
It is natural to ask whether this transform is related to rectifiability and uniform rectifiability in the same way that the Euclidean Riesz transform is, and to describe the sets $E\subset \HH$ such that $T^{\mathsf{R}}$ is bounded in $L_2(E):=L_2(\mathcal{H}^3|_{E})$, where $E\subset \HH$ and $\mathcal{H}^3$ is the $3$--dimensional Hausdorff measure induced by the metric $d(x,y)=\|x^{-1}y\|_{\Kor}$.

The first difficulty in this project is defining analogues of Lipschitz graphs in $\HH$. Unlike the Euclidean case, we cannot define Lipschitz graphs as the images of Lipschitz maps from $\R^2$ to $\HH$ or $\R^3$ to $\HH$; by a result of Ambrosio and Kirchheim \cite{AK}, $\mathcal{H}^3(f(\R^3))=0$ for all Lipschitz functions $f\from \R^3 \rightarrow \HH$. Franchi, Serapioni and Serra Cassano  \cite{FSS1} introduced an \textit{intrinsic} notion of Lipschitz graphs in Carnot groups which has been very influential in the development of sub-Riemannian geometric measure theory, see e.g.\ \cite{SCnotes,perttirev} and the references therein. Intrinsic Lipschitz graphs satisfy a cone condition which will be defined in Section \ref{sec:ilgs}. Moreover, they are $3$--Ahlfors regular and thus the question of the $L_2$--boundedness of the Heisenberg Riesz transform on intrinsic Lipschitz graphs makes sense.


Indeed, if $\Gamma$ is an intrinsic Lipschitz graph of a bounded function and $\nu_{\Gamma}=\mathcal{H}^3|_{\Gamma}$ the double truncations $$T^{\mathsf{R}}_{r,R}[f \ud \nu_{\Gamma}](p):=\int_{B(p,R) \setminus B(p,r)}\mathsf{R}(y^{-1}p)f(y) \ud \nu_{\Gamma} (y)$$ are well defined for $f \in L_2(\Gamma)$, $x \in \Gamma$ and $0<r<R<\infty$. As usual, we do not know \emph{a priori} that the principal values
$$T^{\mathsf{R}} [f \ud \nu_{\Gamma}] (p)= \pv (p)\int  \mathsf{R} (y^{-1}p) f(y) \ud \nu_{\Gamma} (y):=\lim_{\substack{r\to 0 \\ R\to \infty}} T^{\mathsf{R}}_{r,R}[f \nu_{\Gamma}](p)$$
exist for $\mathcal{H}^3$--a.e. $p \in \Gamma$, so we say that the Heisenberg Riesz transform $T^{\mathsf{R}}$ is bounded in $L_2(\Gamma)$ if the truncations $T^{\mathsf{R}}_{r,R}$ are uniformly bounded in $L_2(\Gamma)$; that is if there exists some $C>0$ such that
$$\|T^{\mathsf{R}}_{r,R} [f \ud \nu_{\Gamma}]\|_{L_2(\Gamma)} \leq C \|f\|_{L_2(\Gamma)}$$
for all $f \in L_2(\Gamma)$ and $0<r<R<\infty$.

The question of the boundedness of the Heisenberg Riesz transform was first discussed in \cite{CM}, where it was noted that the Heisenberg Riesz transform is $L_2$--bounded on the simplest examples of intrinsic Lipschitz graphs: the vertical planes (planes in $\HH$ which contain the center $\langle Z \rangle=\{(0,0,z):z\in \R\}$). Recently, some partial results provided hope that, as in the Euclidean case, the Heisenberg Riesz transform might be $L_2$--bounded on intrinsic Lipschitz graphs. First, in \cite{CFO2} it was shown that the Heisenberg Riesz transform is $L_2$--bounded on compactly supported intrinsic $C^{1,\alpha}$ graphs, and in \cite{FO} it was shown that it is also $L_2$--bounded on intrinsic Lipschitz graphs of the form $\Gamma_{\R^2} \times \R \subset \HH$ where $\Gamma_{\R^2}$ is a Euclidean Lipschitz graph in $\R^2$. In this paper we prove that, surprisingly and unlike the Euclidean case, $T^{\mathsf{R}}$ is not $L_2$--bounded on certain intrinsic Lipschitz graphs.

\begin{thm} 
\label{mainthmintro}There exists a compactly supported intrinsic Lipschitz graph $\Gamma$ such that the Heisenberg Riesz transform is unbounded in $L_2(\Gamma)$.
\end{thm}
We also record that if $\Gamma$ is the intrinsic Lipschitz graph from Theorem \ref{mainthmintro} then the Heisenberg Riesz transform is unbounded in $L_p(\Gamma)$ for all $p \in (1,\infty)$. This follows by its unboundedness in $L_2(\Gamma)$ combined with \cite[Theorem 1.1]{ntv-cotlar} and the remark right after that theorem. 

The need to characterize the lower-dimensional sets on which the (Euclidean) Riesz transform and other singular integrals are bounded in $L_2$ led to the development of uniform rectifiability in Euclidean spaces. In the Heisenberg group, intrinsic Lipschitz graphs have been used to study rectifiability \cite{MSS,FSSCDiff} and quantitative rectifiability \cite{CFO1,NY1, NY2, CLY, FORig,Rig}, and, although not explicitly stated, it has been anticipated that intrinsic Lipschitz graphs should be the building blocks of uniformly rectifiable sets. Theorem \ref{mainthmintro} suggests that in $\HH$, notions of uniform rectifiability based on intrinsic Lipschitz graphs and notions of uniform rectifiability based on singular integrals may diverge, and points to deep differences between the theory of uniform rectifiability in $\HH$ and its Euclidean counterpart.

On the way to proving Theorem \ref{mainthmintro} we also prove that the strong geometric lemma fails in the first Heisenberg group, thus further highlighting the divergence between Euclidean and Heisenberg concepts of uniform rectifiability. In order to make our statement precise we first introduce codimension--1 $\beta$--numbers. If $E$ is a Borel subset of the $(2n+1)$-dimensional Heisenberg group $\HH_n$, \(x\in \HH_n\), and $r>0$ 
we define
\begin{equation}\label{eq:def-beta} \beta_{E}(x,r)=\inf_{L \in \mathsf{VP}} r^{-2n-1} \int_{B(x,r)\cap E} \frac{d (y,L)}{r}  \, \ud  \mathcal{H}^{2n+1}(y)
\end{equation}
where in the infimum, $\mathsf{VP}$ stands for \textit{vertical planes} and denotes the set of codimension--$1$ planes which are parallel to the $z$-axis. 

In \cite{CLY} we proved that if $\Gamma$ is an intrinsic $\lambda$--Lipschitz graph in $\HH_n, n \ge 2,$ then, for any ball $B=B(y,R)\subset \HH_n$,
  \begin{align}
  \int_0^R \int_{B \cap \Gamma} \beta_{\Gamma}(x,r)^2 \ud \mathcal{H}^{2n+1}(x) \frac{\ud r}{r} \lesssim_{\lambda} R^{2n+1}. \label{eq:betasgl}
  \end{align}
This is called the strong geometric lemma. (We actually established \eqref{eq:betasgl} for an $L_2$ version of $\beta$--numbers, which easily implies \eqref{eq:betasgl} as it is stated here.)

The strong geometric lemma holds for Lipschitz graphs in $\R^n$ by a result of Dorronsoro, obtained in \cite{dor}, and is one of the foundations of uniform rectifiability in $\R^n$. In particular, an Ahlfors regular subset of $\R^n$ satisfies a Euclidean analogue of \eqref{eq:betasgl}, with constants depending only on $n$ and the Ahlfors regularity constant of the set, if and only if it is uniformly rectifiable, see \cite{DS1}. 

However, the next theorem shows that the situation is very different in $\HH_1$. In fact, the strong geometric lemma fails in $\HH_1$ for all exponents $s \in [2,4)$.

\begin{thm} \label{thm:beta-theorem}There exist a constant $\lambda>0$, a radius $R>0$, and a sequence of $\lambda$--intrinsic Lipschitz graphs $(\Gamma_n)_{n \in \N}$ such that $\zero \in \Gamma_n$ for all $n$ and
 $$\lim_{n \to \infty} \int_{0}^R\int_{B(0,R) \cap \Gamma_n}  \beta_{\Gamma_n}(x,r)^s \ud \mathcal{H}^{3}(x) \frac{\ud r}{r}=+\infty $$
 for all $s \in [2,4)$.
\end{thm}

The intrinsic Lipschitz graphs in Theorems \ref{mainthmintro} and \ref{thm:beta-theorem} are obtained by modifying a process for constructing intrinsic graphs which appeared recently in \cite[Section 3.2]{NY2}. The method introduced in \cite{NY2} produces bumpy intrinsic graphs which are far from  vertical planes at many scales. However, the intrinsic gradients of the intrinsic graphs produced in \cite{NY2} are $L_2$--bounded but not bounded, so the resulting intrinsic graphs are not intrinsic Lipschitz. We overcome this obstacle by applying a stopping time argument leading to intrinsic Lipschtz graphs which retain key properties of the examples from \cite{NY2}.  

The intrinsic Lipschitz graphs that we construct are determined by the following parameters:
\begin{enumerate}
\item $i \in \N$; the number of steps in the construction,
\item $A \in \N$; the aspect ratio of the initial bumps, and
\item a scaling factor $\rho>1$.
\end{enumerate}
In particular, our intrinsic Lipschitz graphs are intrinsic graphs of functions $f_{i,A, \rho}\from V_0 \to \R$, where $V_0=\{y=0\}$ and where $f_{i,A, \rho}$ is supported on the unit square $[0,1]\times \{0\}\times [0,1]$.

For $i\ll A^4$, we show that the intrinsic Lipschitz graph $\Gamma=\Gamma_{f_{i,A, \rho}}$ has many bumps at scale $r_i := A^{-1}\rho^{-i}$, so 
$$\int_{B(0,R) \cap \Gamma} \beta_{\Gamma}(x,r_i)^s \ud \mathcal{H}^{3}(x) \approx A^{-s}.$$
Since there are roughly $A^4$ such scales, this implies Theorem~\ref{thm:beta-theorem}.

Theorem \ref{mainthmintro} takes much longer to prove and it employs several novel arguments. 
We first perform a ``reduction to vertical planes'' by proving that the principal value of singular integrals with smooth, orthogonal, and $(-3)$--homogeneous kernels on intrinsic Lipschitz graphs can be expressed as the principal value of a related singular integral on a vertical plane. This is achieved in Section \ref{sec:redvertplane}. 

More precisely, let $\phi\from \HH\to \R$ be a smooth and bounded intrinsic Lipschitz function with intrinsic graph $\Gamma_\phi$. Denote by $\Psi_\phi:\HH \to \Gamma_{\phi}$ the projection of $\HH$ to $\Gamma_\phi$ along cosets of $\paramY$. The projection restricts to a homeomorphism from $V_0$ to $\Gamma_\phi$ (but not a biLipschitz map), and we let $\eta_\phi:=(\Psi_\phi)_{\ast} \cL|_{V_0}$ be the pushforward of the Lebebegue measure $\cL|_{V_0}$ to $\Gamma_\phi$. Then $\eta_\phi$ is bounded above and below by multiples of $\mathcal{H}^3|_{\Gamma_\phi}$, see Section \ref{sec:ilgs}. It follows from our results in Section \ref{sec:redvertplane} that if $g\from \HH\to \R$ is a Borel function which is constant on cosets of $\paramY$ then
\begin{equation}
\label{eq:paramrieszintro}
   \Rz_{\phi} g = T^\mathsf{R}[g\ud \eta_\phi], 
\end{equation}
where $\Rz_{\phi} (g)$ is the \textit{parametric Riesz transform} of $g$ defined for $p \in \HH$ by
\begin{equation}\label{eq:paramrieszintrodef}
\Rz_{\phi} g(p)  = \pv(\Psi_\phi(p)) \int_{\Psi_\phi(p) V_0} \mathsf{R}(\Psi_{\phi}(v)^{-1}\Psi_\phi(p)) g(v)\ud v. 
\end{equation}

We then obtain $L_2$ bounds on the parametric Riesz transform of the identity function on the intrinsic Lipschitz graphs $f_{i,A ,\rho}$ produced by our construction. More precisely, we obtain the following proposition.  
\begin{prop} \label{prop:mainBounds-tilde-T} There is a $\delta>0$ such that for all sufficiently large $A>1$, there is a $\rho_A>1$ such that if $N=\lfloor \delta A^4\rfloor$, $\phi_A=f_{N,A, \rho_A}$ is the function produced in the construction of Section~\ref{sec:construction} and $U$ is the unit square $[0,1]\times \{0\}\times [0,1] \subset V_0$, then 
$$\|\Rz_{\phi_A}\one\|_{L_2(U)} \gtrsim A,$$
where $\one$ is the function equal to $1$ on all of $\HH$.
\end{prop}
Proposition \ref{prop:mainBounds-tilde-T} is the most crucial part in the proof of Theorem \ref{mainthmintro}  and combined with \eqref{eq:paramrieszintro} leads relatively quickly to the proof of Theorem \ref{mainthmintro}; see Section \ref{sec:proofofmainthm}.

We prove Proposition~\ref{prop:mainBounds-tilde-T} by analyzing the family of singular integrals $\Rz_{\alpha + t \gamma}$ that arises from a perturbation of an intrinsic Lipschitz function $\alpha$ by a smooth function $\gamma$. This requires new methods to handle the noncommutativity of $\HH$. That is, for functions $a,b \from \R^{n-1}\to \R$, let $\Rz^{\text{Euc}}_{a}$ denote the Euclidean parametric Riesz transform, defined as in \eqref{eq:paramrieszintrodef}. The translation-invariance of the Riesz transform implies that $\Rz^{\text{Euc}}_{a}\one = \Rz^{\text{Euc}}_{a+c}\one$ for any $c\in \R$, so  
$$\Rz^{\text{Euc}}_{a + t b}\one = \Rz^{\text{Euc}}_{a_0 + t b_0}\one,$$
where $a_0=a-a(\zero)$ and $b_0=b-b(\zero)$ both vanish at $\zero$.

This identity does not hold in $\HH$. In $\HH$, translation-invariance implies that if $\Gamma_{\alpha_1}$ is a left-translate of $\Gamma_{\alpha_2}$, then $\Rz_{\alpha_1}\one$ is a left-translate of $\Rz_{\alpha_2}\one$. Unfortunately, $\Gamma_{\alpha + c}$ is a right-translate of $\Gamma_\alpha$, so there is typically no relationship between $\Rz_{\alpha}\one$ and $\Rz_{\alpha + c}\one$.

We solve this problem by writing  $\Rz_{\alpha+t\gamma}\one$ in two ways: first, the direct calculation \eqref{eq:paramrieszintrodef}, and second, $\Rz_{\alpha+t\gamma}\one = \Rz_{\alpha_t}\one \circ \lambda_t$, where each $\lambda_t$ is a left-translation and $\alpha_t$ is a family of functions such that $\alpha_t(\zero)=0$ for all $t$ and $\Gamma_{\alpha_t}=\lambda_t(\Gamma_{\alpha+t\gamma})$. Though these expressions represent the same function, one is easier to estimate at large scales and one is easier to estimate at small scales, and many of the bounds used in the proof of Proposition~\ref{prop:mainBounds-tilde-T} will use one expression at large scales and the other expression at small scales.

Our results lead naturally to several new questions. For example, it is well known \cite[Theorem 20.15]{Mabook}, that if $\Gamma \subset \R^n$ is an $m$--dimensional Lipschitz graph and $f \in L_1(\Gamma)$  then the principal values of the Riesz transform $T^{m}[f \ud \nu_m] (p)$, exist for $\mathcal{H}^m$--a.e. $x \in \Gamma$. The proof uses that $T^m$ is $L_2$--bounded
and in light of Theorem \ref{mainthmintro}, it is quite unclear if the same result holds in $\HH$. We do anticipate that a modification of the construction in the current paper might be used to produce an intrinsic Lipschitz graph $\Gamma$ such that principal values of $T^{\mathsf{R}}[f \ud \nu_{\Gamma}]$ fail to exist $\nu_{\Gamma}$--a.e.\ for (certain) functions  $f \in L_1$, but we will not consider this problem here.

Another interesting problem is the following. Theorem \ref{mainthmintro} asserts that intrinsic Lipschitz regularity is not sufficient for the $L_2$--boundedness of the Heisenberg Riesz transform. On the other hand, according to \cite{CFO2}, intrinsic $C^{1,\alpha}$ regularity is indeed sufficient. Therefore, one could look for ``intermediate'' geometric regularity conditions on intrinsic graphs that would imply the $L^2$--boundedness of the Heisenberg Riesz transform. In particular, and in light of Theorem \ref{thm:beta-theorem}, it would be interesting to answer the following questions:
\begin{question} Let $\Gamma \subset \HH$ be an intrinsic Lipschitz graph which satisfies the Carleson condition \eqref{eq:betasgl}. Is it true that $T^{\mathsf{R}}$ is bounded in $L_2(\Gamma)$?
\end{question}
\begin{question} What natural classes of surfaces satisfy \eqref{eq:betasgl}?
\end{question}
The bounds in Section~\ref{sec:sing-int-perturbed} suggest possible connections between the norm of $T^{\mathsf{R}}$ and the sum of the squares of the $\beta$--numbers in \eqref{eq:betasgl}; see Question~\ref{q:betasgl-q}.

Finally, we note that Theorem~\ref{mainthmintro} is related to the problem of geometrically characterizing removable sets for Lipschitz harmonic functions (RLH sets) in $\HH$. The definition of an RLH set in $\HH$ is completely analogous to its Euclidean counterpart, except that, in $\HH$, a function is called harmonic if it is a solution to the  sub-Laplacian equation $\Delta_{\HH}u=0$. RLH sets in Heisenberg groups were introduced in \cite{CM} and it was shown there that if $E \subset \HH$ is compact, then it is RLH if $\mathcal{H}^3(E)=0$, while it is not RLH if $\dim_H (E)>3$. Moreover, totally disconnected RLH sets with positive $3$-dimensional Hausdorff measure were produced in  \cite{CM, CMT}. On the other hand, it was proved in \cite{CFO2} that if $\mu$ is a non-trivial compactly supported Radon measure in $\HH$ with $3$-upper growth, such that $T^\mathsf{R}$ is bounded in $L_2(\mu)$ then $\spt \mu$ is not RLH. An analogous result holds in $\R^n$, see  \cite[Theorem 4.4]{MR1372240}, and combined with the $L_2$--boundedness of Riesz transforms on Lipschitz graphs implies that compact subsets of $1$-codimensional Lipschitz graphs with positive $(n-1)$-Hausdorff measure are not RLH. This can be used to show that if a compact set $E \subset \R^n$ with $H^{n-1}(E)<\infty$ is RLH then it is purely $(n-1)$--unrectifiable.  To our knowledge, this is the only known proof for this implication.

Theorem \ref{mainthmintro} shows that such a scheme cannot be used in the Heisenberg group, and naturally leads to the following fascinating question:
\begin{question}
\label{que:rem}Does there exist a compact subset of an intrinsic Lipschitz graph in $\HH$ with positive $3$--dimensional Hausdorff measure which is removable for Lipschitz harmonic functions?
\end{question}
If the answer to Question \ref{que:rem} is positive it will imply that the geometric characterization of RLH sets in $\HH$ varies significantly from the analogous characterization in $\R^n$. On the other hand, a negative answer to Question \ref{que:rem} would require a completely new proof method.  

\subsection{Roadmap}
In Section~\ref{sec:prelims}, we establish some definitions and notation for the Heisenberg group and for intrinsic Lipschitz graphs. Even if the reader has seen these notions before, we introduce some new notation for intrinsic Lipschitz graphs in Section~\ref{sec:ilgs}, so we suggest that readers look through this section.

After these preliminaries, the paper can be broken into three rough parts: constructing the family of functions $f=f_{i,A,\rho}$ and graphs $\Gamma=\Gamma_{i,A,\rho}$ that we will use in Theorems~\ref{mainthmintro} and \ref{thm:beta-theorem}, proving lower bounds on the $\beta$--numbers of these surfaces, and estimating the Riesz transform on these surfaces. In Section~\ref{sec:construction}, we construct a family of intrinsic Lipschitz graphs based on the construction in \cite{NY2}. These graphs have bumps at many different scales, and in Section~\ref{sec:beta-numbers}, we calculate the effect of these bumps on the $\beta$--numbers and prove Theorem~\ref{thm:beta-theorem}. 

In Section~\ref{sec:redvertplane}, we start to study the Riesz transform on $\Gamma$ and other intrinsic Lipschitz graphs. Specifically, for an intrinsic Lipschitz function $\phi$, we define $\eta_\phi$ as the pushforward of $\cL|_{V_0}$ as above and study the function $T\eta_\phi$. In general, $T\eta_\phi$ need not be defined everywhere on $\Gamma_\phi$, but in Section~\ref{sec:redvertplane}, we show that if $\phi$ is smooth, bounded, and has bounded derivatives, then $T\eta_\phi$ is defined everywhere on $\Gamma_\phi$.
We also introduce a singular integral operator $\tilde{T}_\phi$ which is defined as a singular integral on a vertical plane and satisfies $\tilde{T}_\phi\one = T \eta_\phi$. Let $F_\phi := \tilde{T}_\phi\one$.

Our main goal in these sections is to prove Proposition~\ref{prop:mainBounds-tilde-T}. We prove Proposition~\ref{prop:mainBounds-tilde-T} by considering the construction of $f_i=f_{i,A,\rho}$ as a sequence of perturbations, starting with $f_{0}=0$, so that for each $i\ge 0$, we obtain $f_{i + 1}$ by adding bumps of scale $r_i:=A^{-1}\rho^{-i}$ to $f_{i}$. Let $\nu_i=f_{i+1}-f_i$. Then we can prove Proposition~\ref{prop:mainBounds-tilde-T} by bounding the derivatives $\frac{\ud}{\ud t} [F_{f_i + t \nu_i}]$ and $\frac{\ud^2}{\ud t^2} [F_{f_i + t \nu_i}]$ and using Taylor's theorem.

We state bounds on the derivatives of $G_{f_i,\nu_i}(t) := F_{f_i + t \nu_i}$ in Section~\ref{sec:sing-int-perturbed}. Because of the scale-invariance of the Riesz transform, we can rescale $f_i$ and $\nu_i$ by a factor $r_i$ to obtain functions $\alpha$ and $\gamma$ such that $\alpha$ varies on scale roughly $\rho$ and $\gamma$ varies on scale roughly $1$ (Section~\ref{subsec:rescaling}). The derivatives of $\alpha$ and $\gamma$ are bounded (Lemma~\ref{lem:gamma-deriv-bounds} and Appendix~\ref{ap:A}), and in fact we prove bounds on derivatives of $G_{\zeta,\psi}(t)$ for any functions that satisfy the same bounds.

In the remaining sections, we prove the bounds in Section~\ref{sec:sing-int-perturbed}. First, in Section~\ref{sec:perturbations}, we write $G_{\zeta,\psi}'(0)$ as an integral in two ways, one which is easier to control for large scales and one  for small scales (Lemma~\ref{lem:grrab-large} and Lemma~\ref{lem:grrab-small}). In Euclidean space, these two formulas would be the same; the difference between them comes from the noncommutativity of the Heisenberg group. We use these formulas to prove an upper bound on $G'_{f_i,\nu_i}(0)$ (Lemma~\ref{lem:G'-formula-and-bound-rev}).

In Section~\ref{sec:perturbations}, we define translation-invariant approximations of $G'_{\zeta,\psi}$ by showing that when $\lambda$ is a linear function approximating $\zeta$ to first order at $p$, then $G'_{\zeta,\psi}(0)$ is close to $G'_{\lambda,\psi}(0)$ on a neighborhood of $p$. We use this approximation to prove lower bounds on $G'_{f_i,\nu_i}(0)$ in Section~\ref{sec:first-lower} and to bound inner products of the form $\langle G'_{f_i,\nu_i}(0), G'_{f_j,\nu_j}(0)\rangle$ in Section~\ref{sec:quasi-orthog}. 

In Section~\ref{sec:second-deriv}, we use the formulas from Section~\ref{sec:perturbations} again to bound $G''_{f_i,\nu_i}$.
By Taylor's theorem, 
$$F_{f_N} = \sum_{i=0}^{N-1} G'_{f_i,\nu_i}(0) + \sum_{i=0}^{N-1} O(\|G''_{f_i,\nu_i}\|_\infty).$$
Our bounds on $G'_{f_i,\nu_i}(0)$ and $\langle G'_{f_i,\nu_i}(0), G'_{f_j,\nu_j}(0)\rangle$ lead to a lower bound on the first term, and our bounds on $G''_{f_i,\nu_i}$ bound the error term. This proves Proposition~\ref{prop:mainBounds-tilde-T} (see Section~\ref{sec:sing-int-perturbed} for details).

Finally, in Section~\ref{sec:proofofmainthm}, we use Proposition~\ref{prop:mainBounds-tilde-T} to prove Theorem~\ref{mainthmintro}. We first show that when $\phi_A$ is as in Proposition~\ref{prop:mainBounds-tilde-T}, the $L_2$ norm of the Riesz transform on $L_2(\Gamma_{\phi_A})$ is large. We then combine scaled copies of the $\Gamma_{\phi_A}$'s to obtain a single compactly supported intrinsic Lipschitz graph $\Gamma$ such that the Riesz transform is unbounded on $L_2(\Gamma)$, as desired.

\section{Preliminaries}\label{sec:prelims}
Throughout this paper, we will use the notation $f\lesssim g$ to denote that there is a universal constant $C>0$ such that $f\le Cg$ and $f\lesssim_{a_1,a_2,\dots} g$ to denote that there is a function $C(a_1,a_2,\dots)>0$ such that $f\le C(a_1,a_2,\dots) g$. The notation $f\approx g$ is equivalent to $f\lesssim g$ and $g\lesssim f$. We will also use the big--$O$ notation $O(f)$ to denote an error term which is at most $Cf$ for some constant $C>0$ and $O_a(f)$ for an error term which is at most $C(a) f$.

\subsection{Heisenberg group}\label{sec:prelim-Heis}

The three dimensional Heisenberg group $\HH$ is the Lie group on $\R^{3}$ defined by the multiplication
\begin{align}\label{eq:heismult}
  (x,y,z) (x',y',z') = \left(x+x',y+y',z+z' + \frac{xy' - x'y}{2} \right)
\end{align}
The identity element in $\HH$ is $\zero:=(0,0,0)$ and the inverse of $v=(x,y,z) \in \HH$ is $v^{-1}:=(-x,-y,-z)$. 
We denote by $X=(1,0,0),Y=(0,1,0),Z=(0,0,1),$ the coordinate vectors of $\HH$ and we let $x,y,z \from \HH\to \R$ be the coordinate functions. The center of the group is $\langle Z \rangle = \{(0,0,z) : z \in \R\}$.  An element $v \in \HH$ is called a {\it horizontal vector} if $z(v) = 0$, and we denote by $A$ the set of horizontal vectors.

Since $\HH$ is a torsion-free nilpotent Lie group, the exponential map is a bijection between $\HH$ and the nilpotent Lie algebra $\mathfrak{h}=\langle X, Y, Z\mid [X,Y]=Z\rangle$; namely, $\exp(x X + y Y + z Z) = (x,y,z)$. Then \eqref{eq:heismult} is a consequence of the Baker--Campbell--Hausdorff formula
$$\exp(V)\exp(W) = \exp\left(V+W+\frac{[V,W]}{2} + \dots \right).$$
We will frequently identify $\HH$ and $\mathfrak{h}$ and use the same notation for generators of $\HH$ and of $\mathfrak{h}$. In particular, for $V_i\in \mathfrak{h}$, we write the linear span of the $V_i$ as $\langle V_1,V_2,\dots\rangle$, so that the set of horizontal vectors is
$$A=\langle X,Y\rangle = \{x X + yY \mid x,y\in \R\}.$$

Since \eqref{eq:heismult} is based on the Baker--Campbell--Hausdorff formula, for any $v\in \HH$, the span $\langle v\rangle$ is the one-parameter subgroup containing $v$. Since we typically write the group operation in $\HH$ as multiplication, we will often write $w^t = t w$ for $w\in \HH$ and $t\in \R$.

Given an open interval $I \subset \R$, we say that $\gamma: I \to \HH$ is a {\it horizontal curve} if the functions $x \circ \gamma, y \circ \gamma, z \circ \gamma : I \to \R$ are Lipschitz (hence $\gamma'$ is defined almost everywhere on $I$) and $$\frac{\ud}{ \ud s} \left[\gamma (t)^{-1} \gamma(s)\right]\big|_{s=t} \in A,$$
for almost every $t \in I$. Notice that left translations of horizontal curves are also horizontal. 

Given $(a,b) \in \R^2 \setminus \{(0,0)\}$ and $v \in \HH$ we will call the coset $L=v \langle aX+bY \rangle$ a {\it horizontal line}. We define the \emph{slope} of $L$ as $\slope L = \frac{b}{a}$ when $a\ne 0$ and $\slope L = \infty$ when $a=0$. This is the slope of the projection of $L$ to the $xy$--plane. Note that for $t\in \R$,  $(X+\sigma Y)^t$ is a point in the horizontal line through the origin with slope $\sigma$.

Let $\XL$, $\YL$ be the left-invariant vector fields
$$\XL(v)=\left(1,0,-\frac{y(v)}{2}\right) \mbox{ and } \YL(v)=\left(0,1,\frac{x(v)}{2}\right),$$
and let $\XR$, $\YR$ be the right-invariant vector fields
$$\XR(v)=\left(1,0,\frac{y(v)}{2}\right) \mbox{ and } \YR(v)=\left(0,1,-\frac{x(v)}{2}\right).$$
Note that $\XL$ and $\XR$ commute, as do $\YL$ and $\YR$. We let $\partial_x, \partial_y, \partial_z:=Z$ be the usual partial derivatives in $\R^3$. Given any vector field $\mathsf{V}=(\mathsf{V}_x,\mathsf{V}_y, \mathsf{V}_z)\from \R^3 \rightarrow \R^3$ and any smooth function $f\from \R^3 \rightarrow \R$ we let 
$$\mathsf{V}f (v)= \mathsf{V} \cdot \nabla f (v):= \mathsf{V}_x(v) \partial_x f (v)+ \mathsf{V}_y(v) \partial_y f (v)+\mathsf{V}_z(v) \partial_z f (v).$$ So for example,
$$\XL f(v) = \frac{\ud}{\ud t}f(vX^t)\big|_{t=0} = \partial_x f (v)-\frac{y(v)}{2} Z f(v), v \in \HH.$$
We also define the {\it horizontal gradient} of $f$ as $\nabla_{\HH} f= (\XL f, \YL f)$. For clarity, we will typically use square brackets for the object of a differential operator and use $\cdot$ as a low-precedence multiplication operator, so that $\mathsf{V}f\cdot \mathsf{W}g$ is equal to $\mathsf{V}[f] \mathsf{W}[g]$, not $\mathsf{V}[f\mathsf{W}g]$.

The \emph{Kor\'anyi metric} on $\HH$ is the left-invariant metric defined by
$$d_{\Kor}(v,v'):=\|v^{-1} v'\|_{\Kor},$$
where
$$\|(x,y,z)\|_{\Kor}:=\sqrt[4]{(x^2+y^2)^2+16 z^2}.$$
Note that $\|a X + b Y\|_{\Kor} = \sqrt{a^2 + b^2}$, so the Korányi length of a horizontal line segment is equal to the Euclidean length of its projection to the $xy$--plane.

We also define a family of automorphisms  
$s_t \from  \HH \to \HH, t \in \R,$ 
\begin{align*}
 s_t(x,y,z) = (t x, t y, t^2 z).
\end{align*}
The mappings $s_t$ dilate the metric; for $t \geq 0$ and $p, p' \in \HH$,
\begin{align*}
d_{\Kor}(s_t (p), s_t(p'))= t d_{\Kor} (p,p').
\end{align*}
When $w\in A$ is a horizontal vector, the one-parameter subgroup generated by $w$ can be written in terms of $s_t$, i.e., $s_t(w)=w^t$, but this is not true when $w$ is not horizontal.

We can also define the reflection through the $z$--axis $\theta\from \HH \to \HH$ by
$$\theta(x,y,z) = (-x,-y,z).$$
Note that $\theta=s_{-1}.$

A \emph{vertical plane} $V$ is a plane that is parallel to the $z$--axis. For any such plane, the intersection $V\cap A$ is a horizontal line $v \langle a X + b Y\rangle$, and we can write $V = v \langle a X + b Y, Z\rangle$. We define the slope of $V$ as $\slope V:=\slope (V \cap A)$.

We will frequently refer to the vertical plane $V_0 = \{y = 0\}$. We will also use the following projections. First, we define the natural (nonlinear) projection $\Pi \from \HH \to V_0$ along cosets of $\paramY$ by $\Pi(v)=v Y^{-y(v)}, v \in \HH$. Equivalently,
$$\Pi(x,y,z)=\left(x,0,z-\frac{1}{2} xy\right).$$
Note that $\Pi$ is not a homomorphism, but it commutes with scaling because $s_t$ sends cosets of $\paramY$ to cosets of $\paramY$, i.e., 
$$\Pi(s_t(v)) = s_t(v) Y^{-y(s_t(v))} = s_t(v) Y^{-ty(v)} = s_t(v Y^{-y(v)}) = s_t(\Pi(v))$$
for all $v\in \HH$ and $t\in \R$.

Moreover, if $V$ is a vertical plane which is not a coset of the $yz$--plane we define the projection $\Pi_{V}\from \HH \to V$ along cosets of $\paramY$ by setting $\Pi_{V}(v)$ to be the unique point of intersection of the coset $v\langle Y\rangle$ and $V$. In particular, given $p \in \HH$ the projection $\Pi_{p V_0}\from \HH\to p V_0$ is given by $\Pi_{p V_0}(v) = v Y^{y(p)-y(v)}, v \in \HH$. When $\zero\in V$, this likewise commutes with $s_t$.

\subsection{Kernels and symmetries}
\label{sec:ker}
In this paper, we will consider kernels on $\HH$ which are either $\R$-- or $\R^2$--valued continuous functions on $\HH \backslash \{0\}$. Given a kernel $K$, let $\widehat{K}$ denote the kernel $\widehat{K}(v)=K(v^{-1})$ for all $v\in \HH$. Given a Borel measure $\nu$ on $\HH$, we formally define the singular integral operator $T^K$ by letting $T^K\nu(p)$ be the principal value
$$T^K\nu (p):=\pv (p) \int \widehat{K} (p^{-1}w) \ud \nu (w),$$
where 
\begin{equation}\label{eq:def-pvp}
  \pv(p) \int g(w) \ud \nu (w):=\lim_{\substack{r\to 0 \\ R\to \infty}}\int_{B(p,R) \setminus B(p,r)} g(w) \ud \nu (w).
\end{equation}
For a Borel set $A \subset \HH$ we denote
$$\pv(p) \int_A g(w) \ud \nu (w):=\pv(p) \int g(w) \one_{A}(w) \ud \nu (w).$$

This definition gives rise to several operators. For $0<r<R$, we define truncated convolution operators $T^K_r$ and $T^K_{r,R}$ by
\begin{align*}
  T_r^K \nu(p) &:= \int_{\HH \setminus B(p,r)} K(w^{-1}p) \ud \nu (w), \\
  T_{r, R}^K \nu(p) &:= \int_{B(p,R) \setminus B(p,r)} K(w^{-1}p) \ud \nu (w),
\end{align*}
for any Borel measure $\nu$ on $\HH$ and any $p\in \HH$ such that these integrals are defined. Likewise we define operators $T^K_{\nu;r,R}f = T^K_{r,R} [f\ud \nu]$ and $T^K_{\nu}f = T^K [f\ud \nu]$. When $K$ is understood, we will write $T=T^K$.

 

For $\alpha \in \Z$, a kernel is said to be $\alpha$--{\it homogeneous} or {\it of degree $\alpha$} if
\begin{align*}
  K(s_t(p)) = t^\alpha K(p), \qquad \forall t \in \R, p \in \HH.
\end{align*}
A function $f\from \HH \to \R^n$ is \emph{$\HH$--odd} if $f(\theta(p))=-f(p)$ for all $p$; it is \emph{$\HH$--even} if $f(\theta(p))=f(p)$ for all $p$, and since $\theta=s_{-1}$, a homogeneous kernel is $\HH$--odd or $\HH$--even if it is homogeneous for an odd or even power, respectively.

\begin{lemma}\label{lem:deriv-homog}
  Let $W$ be a left-invariant vector field corresponding to a horizontal element of \(\HH\).
    If $K$ is $\alpha$--homogeneous, then $WK$ is $(\alpha-1)$--homogeneous.
\end{lemma}

\begin{proof}
  Suppose that $K$ is $\alpha$--homogeneous.  Let $u \in \HH$ be an element of norm 1 and let $t \in \R$.  Then
  \begin{align*}
    WK(u) &= \lim_{h \to 0} \frac{K(uW^h) - K(u)}{h} \\
    &= t^{-\alpha} \lim_{h \to 0} \frac{K(s_t(u) W^{th}) - K(s_t(u))}{h} \\
    &= t^{-\alpha+1} \lim_{h \to 0} \frac{K(s_t(u) W^{th}) - K(s_t(u))}{th} \\
    &= t^{-\alpha+1} WK(s_t(u)).
  \end{align*}
  We now get that $WK$ is $(\alpha - 1)$--homogeneous.
\end{proof}

Likewise, derivatives of $\HH$--odd kernels are $\HH$--even and vice versa.

Given an orthogonal matrix $M\in \mathrm{O}(2)$, $M$ acts on $\HH$ as an isometry
$$\tilde{M}(x,y,z) = (M(x,y), \det(M) z).$$
Given an $\R^2$--valued kernel, we say that it is {\it orthogonal} if
\begin{align*}
  K(\tilde{M}(p)) = M(K(p))
\end{align*}
for all $p\in \HH$ and all $M\in \mathrm{O}(2)$.

We now define a specific kernel that is the main object of our study.  Let $\Psi\from \HH\to \R$,
$\Psi(v)=\|v\|_\Kor^{-2}$. By a celebrated result of Folland, see \cite{fol} and \cite[Theorem  5.15]{cdpt}, we know that the fundamental solution of the sub-Laplacian equation
$$\XL^2+\YL^2=0,$$
is $(8 \pi)^{-1} \Psi$. Analogously to the Euclidean case,  the Riesz kernel $\mathsf{R}$ is defined as
\begin{align}
\label{eq:riesz-kernel}
  \mathsf{R}(v) & :=-\nabla_{\HH}\Psi =-\left(\XL \Psi, \YL \Psi\right)=
  \left( \frac{2x (x^2+y^2)-8yz}{\|v\|_\Kor^6}, \frac{2y(x^2+y^2)+8xz}{\|v\|_\Kor^6} \right).
\end{align}
Since $\Psi$ is symmetric around the origin and homogeneous of degree $-2$, its gradient $\mathsf{R}(v)$ is an $\HH$--odd orthogonal kernel of degree $-3$. The smoothness and the $-3$--homogeneity of $\mathsf{R}$ easily imply that it is a $3$--dimensional standard Calder\'on--Zygmund kernel, see e.g.\ \cite[Chapter 6]{christ}. Therefore, if $\nu$ is a Borel measure  on $\HH$ such that
$$\nu (B(x,r)) \leq C r^3, \forall x \in \HH, r>0,$$
 then 
$|T_r^{\mathsf{R}} (f \ud \nu)|<\infty$ for $f \in L_p(\nu), p \in [1,\infty)$ and  $|T_{r,R}^{\mathsf{R}} (f \ud \nu)|<\infty$ for $f \in L_p(\nu), p \in [1,\infty]$. In fact, truncated singular integrals (with respect to $\nu$) are finite for any Borel kernel which satisfies $|K(v)| \lesssim \|v\|^{-3}, v \in \HH \setminus \{0\}$.

\subsection{Intrinsic graphs and intrinsic Lipschitz graphs}\label{sec:ilgs}

In previous papers, intrinsic graphs have been defined as graphs of functions from the vertical plane $V_0=\langle X,Z\rangle$ to $\R$. In this paper, we introduce new notation that defines them in terms of functions from $\HH$ to $\R$ that are constant along cosets of $\paramY$. Any function from $V_0$ to $\R$ can be extended to a function that is constant along cosets of $\paramY$, so the two definitions give the same class of graphs, but this definition streamlines some notation. Di~Donato and Le~Donne have used similar techniques to define intrinsically Lipschitz sections in \cite{DDLD-sections}.

For any function $f\from \HH \to \R$ which is constant on cosets of $\paramY$, we define the \emph{intrinsic graph} of $f$ as
$$\Gamma_f=\{vY^{f(v)}\mid v\in V_0\}=\{p\in \HH\mid f(p)=y(p)\}.$$
We define $\Psi_f\from \HH\to \Gamma_f$ by $\Psi_f(p)=pY^{f(p)-y(p)}$ for all $p$.
This map projects $\HH$ to $\Gamma_f$ along cosets of $\paramY$. It is constant along cosets of $\paramY$ and satisfies 
\begin{equation}\label{eq:projection-identity}
  y(\Psi_f(p))=f(p) \text{\qquad for all $p\in \HH$}.
\end{equation}
Left-translations and scalings of intrinsic graphs are also intrinsic graphs, and we can use \eqref{eq:projection-identity} to determine the corresponding functions.

\begin{lemma}\label{lem:ilg-translations}
  Let $f\from \HH \to \R$ be a function which is constant on cosets of $\langle Y\rangle$ and let $g\in \HH$. Let $h\from \HH\to \R$,
  $$h(p)=y(g) + f(g^{-1}p).$$
  Then $h$ is constant on cosets of $\langle Y\rangle$ and satisfies $\Gamma_h = g\Gamma_f$ and $\Psi_h(p) = g \Psi_f(g^{-1}p)$
  for any $p\in \HH$.
\end{lemma}
\begin{proof}
  Since $\Psi_f$ is the unique map from $\HH$ to $\Gamma_f$ that satisfies $\Psi_f(p)\paramY = p\paramY$ for all $p\in \HH$, the map $\hat{\Psi}(p) = g \Psi_f(g^{-1}p)$ sends $\HH$ to $g\Gamma_f$ and satisfies
  $$\hat{\Psi}(p)\langle Y \rangle = g \Psi_f(g^{-1}p) \langle Y \rangle = g g^{-1}p \langle Y \rangle = p\paramY$$
  for all $p\in \HH$. Therefore, $\hat{\Psi}=\Psi_h$ where
  $$h(p) = y(\hat{\Psi}(p)) = y(g\Psi_f(g^{-1}p)) = y(g) + f(g^{-1}p)$$
  and $g\Gamma_f =\Gamma_h$.
\end{proof}

\begin{lemma}\label{lem:ilg-scaling}
  Let $f\from \HH \to \R$ be a function which is constant on cosets of $\langle Y\rangle$. Let $t\ne 0$ and let $h\from \HH\to \R$,
  $$h(p)=t f(s_t^{-1}(p)).$$
  Then $\Gamma_h = s_t(\Gamma_f)$ and $\Psi_h(p) = s_t(\Psi_f(s_t^{-1}(p)))$
  for any $p\in \HH$.
\end{lemma}
\begin{proof}
  As above, $\hat{\Psi}(p) = s_t(\Psi_f(s_t^{-1}(p)))$ has image $s_t(\Gamma_f)$ and satisfies
  $$\hat{\Psi}(p)\langle Y \rangle = s_t(\Psi_f(s_t^{-1}(p))) \langle Y \rangle = s_t(s_t^{-1}(p) \langle Y \rangle) = p\paramY$$
  for all $p\in \HH$. Therefore, $\hat{\Psi}=\Psi_h$ where
  $$h(p) = y(\hat{\Psi}(p)) = t f(s_t^{-1}(p))$$
  and $s_t(\Gamma_f) = \Gamma_h$.
\end{proof}

Let $\XL=X$ and $\XR= \XL-y Z$ be the left-invariant and right-invariant vector fields defined in Section~\ref{sec:prelim-Heis}. For a smooth function $f\from \HH\to \R$, we have $\XL[f](v) = \frac{\ud}{\ud t}f(vX^t)|_{t=0}$ and $\XR[f](v) = \frac{\ud}{\ud t}f(X^tv)|_{t=0}$.
If $\phi\from \HH\to \R$ is constant on cosets of $\paramY$, we define the intrinsic gradient $\nabla_\phi$ as the vector field
\begin{equation}\label{eq:def-nabla-phi}
  \nabla_\phi(p) = \XR(p) - \phi(p) Z(p) = \XL(p) + (y(p) - \phi(p))Z.
\end{equation}
When $v\in V_0$, this agrees with the usual definition of the intrinsic gradient 
$\nabla_\phi(v) = \XL(v) - \phi(v)Z(v)$; equation \eqref{eq:def-nabla-phi} is the extension of $\nabla_\phi$ that is right-invariant with respect to the action of $\paramY$.
If $\phi$ and $\beta$ are smooth and constant on cosets of $\paramY$, then for all $p\in \HH$ and $t\in \R$,
$$\nabla_\phi \beta(pY^t) = \frac{\ud}{\ud u}\beta((X-\phi(pY^t)Z)^u pY^t)\big|_{u=0} = \frac{\ud}{\ud u}\beta((X-\phi(p)Z)^u p)\big|_{u=0} = \nabla_\phi \beta(p),$$
so $\nabla_\phi \beta$ is constant on cosets of $\paramY$. 

The intrinsic gradient $\nabla_\phi$ can also be interpreted in terms of the horizontal curves that foliate $\Gamma_\phi$. When $\phi$ is smooth, the restriction of $\nabla_\phi$ to $V_0$ is the smooth vector field $\nabla_\phi(v)=\XL - \phi(v) Z$. It follows that $V_0$ is foliated by integral curves of $\nabla_\phi$; we call these the \emph{characteristic curves} of $\Gamma_\phi$.  If $g\from \R\to V_0$ is such a curve then $\gamma=\Psi_\phi\circ g$ is a horizontal curve in $\Gamma_\phi$ with 
\begin{equation}\label{eq:horiz-veloc}
  \gamma'(t) = \XL + \nabla_\phi \phi(\gamma(t)) \YL,
\end{equation}
and the following lemma holds.
\begin{lemma}\label{lem:nabla-on-curves}
  Let $\phi, m\from \HH\to \R$ be smooth functions which are constant on cosets of $\paramY$, let $g\from \R\to V_0$ be a characteristic curve of $\Gamma_\phi$, and let $\gamma=\Psi_\phi\circ g$. For any $t\in \R$ and any $k\ge 1$,
  $$\nabla^k_\phi m(\gamma(t)) = \nabla^k_\phi m(g(t)) = \frac{\ud^k}{\ud t^k}[m\circ \gamma(t)] = \frac{\ud^k}{\ud t^k}[m\circ g(t)].$$
\end{lemma}
\begin{proof}
  Since $g$ is an integral curve of $\nabla_\phi$, we have
  $$\nabla^k_\phi m(g(t)) = \frac{\ud^k}{\ud t^k}[m\circ g(t)]$$
  for any $k\ge 1$. Since $m$ and $\nabla^k m$ are constant on cosets of $\paramY$, we have $\nabla^k_\phi m(\gamma(t)) = \nabla^k_\phi m(g(t))$ and $m\circ g = m\circ \gamma$, which implies the lemma.
\end{proof}
In particular, if $\gamma$ is as above, then $\nabla_\phi \phi(\gamma(0))$ is the slope of the tangent line to $\gamma$ at $\gamma(0)$. This implies that the intrinsic gradient is invariant under translations and scalings. That is, if $\Gamma_{\hat{\phi}} = gs_t(\Gamma_\phi)$, then  $\nabla_{\hat{\phi}} \hat{\phi}(gs_t(p)) =\nabla_\phi\phi(p)$.

For $0<\lambda<1$, we define the open double cone
$$\Cone_\lambda = \{p\in \HH \mid \lambda d_\Kor(\zero, p) < |y(p)|\}.$$
This is a scale-invariant cone, and when $\lambda$ is close to $1$, it is a small neighborhood of $\langle Y\rangle \setminus \{\zero\}$. An intrinsic graph $\Gamma_\phi$ is a \emph{$\lambda$--intrinsic Lipschitz graph} if $p\Cone_\lambda\cap \Gamma_\phi=\emptyset$ for all $p\in \Gamma_\phi$. Equivalently, $\Gamma_\phi$ is $\lambda$--intrinsic Lipschitz if and only if $\Lip(y|_{\Gamma_\phi})\le \lambda$. If $\phi\from \HH \to \R$ is constant on cosets of $\paramY$ and $\Gamma_\phi$ is a $\lambda$--intrinsic Lipschitz graph, we say that $\phi$ is a \emph{$\lambda$--intrinsic Lipschitz function}.

\begin{lemma}\label{lem:y-distance}
  Let $\lambda\in (0,1)$ and let $\Gamma_f$ be a $\lambda$--intrinsic Lipschitz graph. Then
  \begin{align*}
    |y(v) - f(v)| \approx_\lambda d_\Kor(v, \Gamma_f), \qquad \forall v \in \HH.
  \end{align*}
\end{lemma}

\begin{proof}
  On one hand, $d_\Kor(v, \Gamma_f) \le d_\Kor(v, \Psi_f(v)) = |y(v) - f(v)|$, so it suffices to show that $d_\Kor(v, \Gamma_f) \gtrsim |y(v) - f(v)|$. It suffices to show that there is some $C>0$ depending only on $\lambda$ such that $d_\Kor(p Y^\alpha,\Gamma_f) \ge C |\alpha|$ for all $p\in \Gamma_f$ and $\alpha\in \R$; the lemma then follows by taking $p=\Psi_f(v)$ and $\alpha = y(v) - f(v)$.
  
  Let $C=\frac{1-\lambda}{1+\lambda}$, so that $\lambda = \frac{1-C}{1+C}$. Let $B(Y,C)$ be the open ball of radius $C$ around $Y$. If $q\in B(Y,C)$, then 
  $$\lambda d_\Kor(\zero,q) < \lambda(1+C) = 1-C < |y(q)|,$$
  so $q\in \Cone_\lambda$. Since $\Cone_\lambda$ is scale-invariant, this implies that $B(Y^\alpha, C|\alpha|)\subset \Cone_\lambda$.
  
  By the intrinsic Lipschitz condition, $\Gamma_f \cap p \Cone_\lambda = \emptyset$, so 
  $$\Gamma_f \cap p B(Y^\alpha, C|\alpha|) = \Gamma_f \cap B(p Y^\alpha, C|\alpha|) = \emptyset.$$ Therefore, $d_\Kor(p Y^\alpha, \Gamma_f) \ge C|\alpha|$, as desired.
\end{proof}

By \cite{CMPSC}, if $\Gamma_\phi$ is $\lambda$--intrinsic Lipschitz, then $\|\nabla_{\phi} \phi\|_\infty$ is bounded by a function of $\lambda$. Indeed,
\begin{equation}\label{eq:ilg-deriv-bound}
  \|\nabla_{\phi} \phi\|_\infty\le \frac{\lambda}{\sqrt{1-\lambda^2}},
\end{equation}
see \cite[Sec. 2.2]{NY2}. 
Conversely, if $\phi$ is defined on all of $\HH$ and $\nabla_\phi \phi$ is bounded, then $\phi$ is $\lambda$--intrinsic Lipschitz for some $0<\lambda<1$ depending on $\|\nabla_\phi \phi\|_\infty$ \cite{CMPSC}.

When $\phi$ is smooth and $p\in \Gamma_\phi$, we define the \emph{tangent plane} to $\Gamma_\phi$ at $p$ to be the vertical plane $P_p = p \langle X+\nabla_\phi \phi(p) Y, Z\rangle$ with slope $\nabla_\phi \phi(p)$. For $t>0$, $p s_t(p^{-1} \Gamma_\phi)$ is the scaling of $\Gamma_\phi$ centered at $p$, and as $t\to \infty$, $p s_t(p^{-1} \Gamma_\phi)$ converges to $P_p$. More generally, when $\phi$ is intrinsic Lipschitz, a Rademacher-type theorem holds for almost every $p\in \Gamma_\phi$, so the definition of $\nabla_\phi \phi(p)$ can be extended so that $p s_t(p^{-1} \Gamma_\phi)$ converges to $P_p$ for almost every $p\in \Gamma_\phi$ \cite{FSSCRectifiability}.

The following lemma, based on Lemma~2.3 of \cite{NY2}, is helpful for bounding intrinsic Lipschitz functions.
\begin{lemma}\label{lem:ilg-line-distance}
  Let $0\le \lambda\le 1$ and let $\psi\from \HH \to \R$ be a $\lambda$--intrinsic Lipschitz function. Let $\Gamma=\Gamma_\psi$. Let $g\in \Gamma$. For any $h\in \HH$,
  \begin{equation}\label{eq:ilg-line-distance}
    |\psi(g)-\psi(h)|\le \frac{2}{1-\lambda} d_{\Kor}(g,h\langle Y\rangle) \le \frac{2}{1-\lambda} d_{\Kor}(g,h).
  \end{equation}
  Furthermore, for any $t\in \R$ and any $p\in \HH$,
  $$|\psi(p)-\psi(pZ^t)|\le \frac{4\sqrt{|t|}}{1-\lambda}.$$
\end{lemma}
\begin{proof}
  Since $\psi$ is constant on cosets of $\paramY$, it suffices to prove \eqref{eq:ilg-line-distance} when $h\in \Gamma$.
  Let $m=d_{\Kor}(g, h\langle Y\rangle)$. Let $c\in h\langle Y\rangle$ be such that $d_\Kor(g,c)=m$. By the intrinsic Lipschitz condition,
  $$|y(h)-y(c)|\le |y(h)-y(g)| + m \le \lambda d_\Kor(g,h) + m \le \lambda (m + |y(h)-y(c)|) + m.$$
  This simplifies to give
  $$|y(h)-y(c)|\le \frac{1+\lambda}{1-\lambda}m,$$
  and thus
  \begin{equation*}
    |\psi(g)-\psi(h)| = |y(g)-y(h)| \le |y(g)-y(c)|+|y(c)-y(h)|\le \frac{2m}{1-\lambda}.
  \end{equation*}
 
  For any $t\in \R$ and any $p\in \HH$,
  $$|\psi(p)-\psi(pZ^t)|=|\psi(\Psi_\psi(p))-\psi(\Psi_\psi(p)Z^t)|\le \frac{2}{1-\lambda} \|Z^t\|_\Kor = \frac{4\sqrt{|t|}}{1-\lambda}.$$
\end{proof}

This implies the following lemma, whose proof we omit; see also \cite{FS}.
\begin{lemma} \label{lem:proj-ball-contain}
  Let $\phi$ be a $\lambda$--intrinsic Lipschitz function, let $p\in \Gamma_\phi$, and let $r>0$. There is a $c>0$ depending on $\lambda$ such that
  $$\Pi(B(p,cr)) \subset \Pi(B(p,r)\cap \Gamma_\phi)\subset \Pi(B(p,r)).$$
\end{lemma}
In particular, $\cH^3(\Pi(B(p,r)))\approx \cL(\Pi(B(p,r))) \approx r^3$, where $\cL$ is Lebesgue measure on $V_0$. 
\begin{lemma}\label{lem:def-dv}
  There is a left-invariant Borel measure $\mu$ on $\HH$ such that $\mu(S)=\cL(\Pi(S))$ for any intrinsic Lipschitz graph $\Gamma$ and any Borel set $S\subset \Gamma$.
  Further, if $m\from \HH \to \R$ is a Borel function which is constant on cosets of $\paramY$ and $S\subset \Gamma$ is Borel, then 
  \begin{equation}
  \label{eq:def-dv}
  \int_S m(v) \ud\mu(v) = \int_{\Pi(S)} m(v) \ud\cL(v)
  \end{equation}
  if the integrals exist. If $\lambda\in (0,1)$ and $\Gamma$ is $\lambda$--intrinsic Lipschitz, then $\cH^3(S) \approx_\lambda \mu(S)$. In particular, $\mu|_{\Gamma}$ is Ahlfors $3$--regular with constants only depending on $\lambda$.
\end{lemma}
This will be our ``default'' measure on intrinsic Lipschitz graphs, and we will abbreviate $\ud\mu(v)$ by $\ud v$.
\begin{proof}
  For $S\subset \HH$, let
  \begin{equation}\label{eq:def-mu}
    \mu(S) = \lim_{\epsilon \to 0} \inf_{C_\epsilon(S)} \sum_{U\in C_\epsilon (S)} \cL(\Pi(U))
  \end{equation}
  where $C_\epsilon(S)$ is the set of covers of $S$ by sets of diameter at most $\epsilon$. This is a Borel measure on $\HH$ by \cite[Theorem 4.2]{Mabook}, and the restriction of $\mu$ to any intrinsic Lipschitz graph is the pullback of $\cL|_{V_0}$, i.e., $\mu(S)=\cL(\Pi(S))$ for any intrinsic Lipschitz graph $\Gamma$ and any Borel set $S\subset \Gamma$. Consequently, $\mu$ satisfies \eqref{eq:def-dv}. 
  
  By the area formula, \cite[Theorem 1.6]{CMPSC}, if $\phi$ is a $\lambda$--intrinsic Lipschitz function and $S \subset \Gamma_\phi$ is Borel, then
  \begin{equation*}
  \cH^3(S) \approx \int_{\Pi(S)} \sqrt{1+(\nabla_\phi \phi(v))^2} \ud v = \int_{S} \sqrt{1+(\nabla_\phi \phi(v))^2} \ud v \approx_\lambda \mu(S).
  \end{equation*}
  Since, by \cite[Theorem 3.9]{FS}, $\cH^{3}|_{\Gamma_\phi}$ is an Ahlfors 3--regular measure this implies that $\mu|_{\Gamma_\phi}$ is also Ahlfors $3$--regular with constants only depending on $\lambda$.

  Finally, we check that $\mu$ is left-invariant. It suffices to show that $\cL(\Pi(g U)) = \cL(\Pi(U))$ for any $g\in \HH$ and any Borel set $U\subset \HH$. First, for any $g, h\in \HH$,
  $$\Pi(gh)=gh Y^{-y(g)-y(h)}=g\left(h \cdot Y^{-y(h)}\right) Y^{-y(g)} = \Pi(g\Pi(h)).$$
  Let $\beta_g\from V_0\to V_0$, $\beta_g(v) = \Pi(g v)$, so that $\Pi(gU) = \beta_g(\Pi(U))$. Let $g=(x,y,z)\in \HH$ and $v=(x',0,z')\in V_0$. Then 
  $$\beta_g(v) = \Pi(g v) = g v Y^{-y} = \left(x + x', 0, z + z' - y x'-\frac{1}{2}xy\right).$$
  That is, $\beta_g$ is an affine transformation of $V_0$ with determinant $1$. Thus $\cL(\Pi(gU))= \cL(\beta_g(\Pi(U)))=\cL(\Pi(U))$. By \eqref{eq:def-mu}, $\mu$ is a left-invariant measure.
\end{proof}
  
\subsection{Taylor series estimates}
In this section, we prove a Taylor-type estimate for functions on intrinsic Lipschitz graphs, which we will use extensively in the rest of the paper. Let $a$ be a smooth intrinsic Lipschitz function and let $m$ be a smooth function which is constant on cosets of $\paramY$. (In particular, we can take $m=a$.) We will show that $m$ is close to a constant function or an affine function when the derivatives $\nabla_a m$, $\nabla_a^2 m$, and $Zm$ are small.

\begin{lemma}\label{lem:a-Taylor}
  Let $0<\lambda<1$ and let $a\from \HH\to \R$ be a smooth $\lambda$--intrinsic Lipschitz function. 
  Let $m\from \HH\to \R$ be a smooth function. Suppose that $a$ and $m$ are constant on cosets of $\paramY$. Let $p\in \Gamma_a$ and let $q\in \HH$. Let $r=d_\Kor(p,q)$, $L=\frac{\lambda}{\sqrt{1-\lambda^2}}$, and $B=B(p,2(L+1)r))$. Then
  $$m(q) = m(p) + O_\lambda\left(r \|\nabla_a m\|_{L_\infty(B)} + r^2 \|\partial_z m\|_{L_\infty(B)} \right)$$
  and
  $$m(q) = m(p) + (x(q)-x(p)) \nabla_a m(p) + O_\lambda\left(r^2\big[\|\nabla_a^2 m\|_{L_\infty(B)} + \|\partial_z m\|_{L_\infty(B)}\big]\right).$$
  In particular, if \(\zero\in \Gamma_a\) and $p=\zero$, then $a(\zero)=0$, so
  \begin{align*}
    |a(q) + a(\theta(q))|
    &= O_\lambda\left(r^2(\|\nabla_a^2 a\|_{L_\infty(B)} + \|\partial_z a\|_{L_\infty(B)})\right),
  \end{align*}
  where $\theta(x,y,z)=(-x,-y,z)$.
\end{lemma}

\begin{proof}
  By \eqref{eq:ilg-deriv-bound}, we have $\|\nabla_a a\|_\infty\le L$.
  Let $\gamma\from \R\to \Gamma_a$ be a horizontal curve through $p$. We parametrize $\gamma$ so that $x(\gamma(t))=t$ for all $t\in \R$. In particular $\gamma(x(p)) = p$. By \eqref{eq:horiz-veloc}, $\|\gamma'(t)\| \le L+1$ for all $t$, so $\Lip(\gamma)\le L+1$. In particular, $\gamma(x(q)) \in B(p,(L+1)r)$.
  
  Recall that $\Pi_{p V_0}\from \HH\to p V_0$ is the projection $\Pi_{p V_0}(s) = s Y^{y(p)-y(s)}$; for any $s\in \HH$, we have 
  $$d_\Kor(p,\Pi_{p V_0}(s)) \le d_\Kor(p,s) + |y(p)-y(s)|\le 2 d_\Kor(p,s).$$
  Then $g':=\Pi_{p V_0}(\gamma(x(q)))$ and $q':=\Pi_{p V_0}(q)$ are two points in $B \cap p V_0$ with the same $x$--coordinate, so $g' = q' Z^{z_0}$ for some $z_0\in \R$ such that $|z_0|\lesssim_\lambda r^2$. Since $m(q)=m(q')$ and $m(\gamma(x(q))) = m(g')$,
  $$m(q) = m(\gamma(x(q))) + O_\lambda (r^2 \|\partial_z m\|_{L_\infty(B)}).$$
  
  Since $(m\circ \gamma)'(t) = \nabla_a m(\gamma(t))$ and $(m\circ \gamma)''(t) = \nabla^2_a m(\gamma(t))$, 
  the Mean Value Theorem implies that
  $$m(\gamma(x(q))) = m(p) + O(r \|\nabla_a m\|_{L_\infty(B)}),$$
  so
  $$m(q) = m(p) + O_\lambda\left(r \|\nabla_a m\|_{L_\infty(B)} + r^2 \|\partial_z m\|_{L_\infty(B)}\right).$$
  Taylor's theorem implies 
  $$m(\gamma(x(q))) = m(p) + (x(q) - x(p)) \nabla_a m(p) + O_\lambda(r^2 \|\nabla_a^2 m\|_{L_\infty(B)}),$$
  so
  \begin{equation*}
    m(q) 
    = m(p) + (x(q) - x(p)) \nabla_a m(p) + O_\lambda\left(r^2(\|\nabla_a^2 m\|_{L_\infty(B)} + \|\partial_z m\|_{L_\infty(B)})\right),
  \end{equation*}
  as desired.
\end{proof}

\section{Construction} \label{sec:construction}

In this section, we construct the family of graphs that we will study in the rest of this paper. Our construction is based on the construction in Section 3.2 of \cite{NY2}. The authors of \cite{NY2} introduced a process to construct an intrinsic graph $\Gamma_\psi$ that is far from a vertical plane at many scales (see Proposition 3.4 of \cite{NY2}). Unfortunately for our purposes, the intrinsic gradient $\nabla_\psi \psi$ is $L_2$--bounded but not bounded, so $\Gamma_\phi$ is not intrinsic Lipschitz. In this section, we will modify that construction via a stopping time argument so that it produces an intrinsic Lipschitz function with similar properties. To keep this paper self-contained, we will  reproduce the construction of \cite{NY2} in parallel with our modification.

The construction depends on three parameters: an integer aspect ratio $A>1$, an integer scale factor $\rho>1$, and a number of steps $i$. In \cite{NY2}, one starts with a function $\psi_0=0$ and constructs $\psi_{i+1}$ by perturbing $\psi_i$. The difference $\psi_{i+1}-\psi_i$ is a sum of bump functions supported on regions in $V_0$ with aspect ratio $A$, and the scale of the perturbations decreases by a factor of $\rho$ at each step. 

Recall that if $\psi \from V_0 \to \R$ is a smooth function, then it induces a smooth vector field $\nabla_\psi = \partial_x - \psi \partial_z$ on $V_0$, and we call integral curves of $\nabla_\psi$ characteristic curves. Since $\psi$ is smooth, there is a unique characteristic curve of $\Gamma_\psi$ through each point of $V_0$.
A \emph{pseudoquad} $Q\subset V_0$ for $\Gamma_\psi$ is a region of the form
$$Q=\left\{(x,0,z)\in V_0\mid x\in [a,b], z\in [g_1(x),g_2(x)]\right\}$$
where $g_1, g_2\from [a,b]\to \R$ are functions whose graphs are characteristic curves, i.e., $g_i'(x)=-\psi(x,0,g_i(x))$ for all $x$. In particular, $g_i\in C^1([a,b])$.  We define the \emph{width} of $Q$ to be $\delta_x(Q)=b-a$ and we define the \emph{height} to be $\delta_z(Q) = g_2(a)-g_1(a)$. Since the distance between the top and bottom boundary varies, there is no single canonical height, but this choice is enough for many applications. The \emph{aspect ratio} of $Q$ is the ratio $\frac{\delta_x(Q)}{\sqrt{\delta_z(Q)}}$; the square root in the denominator makes this ratio scale-invariant. 

We say that two pseudoquads are \emph{disjoint} if and only if their interiors are disjoint. We say that $U=Q_1\cup \dots\cup Q_n$ is a \emph{partition} of $U$ if the $Q_i$'s are disjoint. 

Let $U = [0,1] \times \{0\} \times [0,1]$ and let $\kappa \from [0,1]^2 \to \R$ be a nonnegative smooth function with $\supp \kappa\subset (0,1)^2$. We require that $\|\kappa\|_\infty \le 1$, $\kappa(s,t) > 0$ for $s,t\in [\frac{1}{5}, \frac{4}{5}]$, and that the partial derivatives of $\kappa$ of order at most 2 are all in the interval $[-1,1]$. (The  assumption on partial derivatives is used in \cite{NY2} to bound certain derivatives when $\rho\ge 8$; it can be dropped at the cost of changing some constants.)

We will use induction to construct functions $f_i$ and $\psi_i$ supported on $U$. We start with $f_0=0$ and $\psi_0=0$, and for each $i\ge 0$, we let $r_i:=A^{-1}\rho^{-i}$ and construct:
\begin{itemize}
\item a partition $U = Q_{i,1}\cup \dots \cup Q_{i,k_i}$ such that each $Q_{i,j}$ is a pseudoquad for $\Gamma_{\psi_i}$ with width $\delta_x(Q_{i,j}) = A r_i$, height $\delta_z(Q_{i,j}) = r_i^2$, and aspect ratio $A$,
\item a collection of bump functions $\kappa_{i,j}$ such that $\kappa_{i,j}$ is supported on $Q_{i,j}$ and $\|\kappa_{i,j}\|_\infty \approx A^{-1}r_i$,
\item a set $J_i\subset \{1,\dots, k_i\}$ such that $|\nabla_{f_i}f_i|\le \frac{1}{2}$ on $Q_{i,j}$ for every $j\in J_i$. Furthermore, we let $S_i := \bigcup_{j\in \{1,\dots, k_i\}\setminus J_i} Q_{i,j}$ and require that $S_{i}\supset S_{i-1}$ (where $S_{-1}=\emptyset$).
\end{itemize}
We then define $\kappa_i:=\sum_{j=1}^{k_i} \kappa_{i,j}$, 
$$\nu_i:=\sum_{j\in J_i} \kappa_{i,j}=\one_{S_i^c} \kappa_i,$$ $\psi_{i+1} := \psi_i + \kappa_i$ and $f_{i+1} := f_i + \nu_i$. The $\psi_i$'s are the functions constructed in \cite{NY2}, and the $f_i$'s are a ``stopped'' version of the $\psi_i$'s. That is, when $|\nabla_{f_i}f_i|$ gets too large on a pseudoquad, that pseudoquad is added to $S_i$, and the construction ensures that $f_{k}|_{S_i} = f_{i}|_{S_i}$ for all $k>i$.

We first construct the $Q_{i,j}$'s. Suppose that we have already defined $f_i$. Let
\begin{align*}
  G_i := \left\{\left(mAr_i, 0, nr_i^2\right) : m,n \in \Z\right\},
\end{align*}
let $k_i=A^{-1}r_i^3$, and let $v_{i,1},v_{i,2},\dots,v_{i,k_i}$ be an enumeration of $G_i \cap \big([0,1)\times \{0\} \times [0,1)\big)$. Let $\Phi(\psi_i)_s$ be the flow map of $\nabla_{\psi_i}$ on $V_0$; so that $\Phi(\psi_i)_0(v)=v$ for all $v\in V_0$ and the map $s \mapsto \Phi(\psi_i)_s(v)$, $s\in \R$ is a characteristic curve of $\Gamma_{\psi_i}$. In particular, $x(\Phi(\psi_i)_s(v)) = x(v) + s$. Let
\begin{align*}
  R_{i,j}(s,t) := \Phi(\psi_i)_s(v_{i,j}Z^t)
\end{align*}
and let
\begin{align*}
  Q_{i,j} := R_{i,j}([0,Ar_i] \times [0,r_i^2]).
\end{align*}
This is a pseudoquad of width $Ar_i$ and height $r_i^2$. Because the top and bottom edges of $U$ are characteristic curves of $\psi_i$, we have $Q_{i,j}\subset U$ for all $j$. Indeed, $U=Q_{i,1}\cup \dots \cup Q_{i,k_i}$ is a locally finite partition of $U$. (Local finiteness follows for instance from Lemma~\ref{l:cube-height}.) Let $\mathcal{Q}_i=\{Q_{i,1}, \dots, Q_{i,k_i}\}$.

For each $Q_{i,j}$, we define
\begin{align}\label{eq:def-kappa}
  \kappa_{i,j}(R_{i,j}(s,t)) := A^{-1}r_i \kappa(A^{-1} r_i^{-1} s, r_i^{-2} t),
\end{align}
and let $\kappa_i := \sum_j \kappa_{i,j}$. Note that $\kappa_i$ is smooth and that it is zero in a neighborhood of $\partial Q_{i,j}$ for each $j$.

To define $S_{i}$ and $\nu_i$, we will need some notation.
For every $k$, we say that a pseudoquad $Q'\in \mathcal{Q}_{k+1}$ is a \emph{child} of $Q\in \mathcal{Q}_k$ if $\interior(Q') \cap \interior(Q) \ne \emptyset$. 
Note that this does not necessarily mean that $Q'\subset Q$; the pseudoquads in $\cQ_{k+1}$ do not subdivide the pseudoquads in $\cQ_{k}$.
Nonetheless, by the local finiteness of $\mathcal{Q}_{k+1}$, every $Q\in \mathcal{Q}_k$ has only finitely many children.

Let $\cC(Q)$ be the set of children of $Q$ and define $\cC^n(Q)$ inductively so that $\cC^0(Q) = \{Q\}$ and $\cC^{n}(Q) = \bigcup_{Q'\in \cC^{n-1}(Q)} \cC(Q')$. For any set $\cM$ of pseudoquads, we let
$$\bigcup(\cM):=\bigcup_{Q\in \cM} Q.$$ 
For $Q\in \mathcal{Q}_k$ and $l>k$, let 
$Q^{(l)} = \bigcup(\cC^{l-k}(Q)).$
Let $\cD(Q)=\bigcup_{n=0}^\infty \cC^n(Q)$ be the set of \emph{descendants} of $Q$.

If $Q\in \mathcal{Q}_k$ and $v\in \interior(Q)$, then any neighborhood of $v$ intersects the interior of some child of $Q$. It follows that $v$ lies in the closure of $\bigcup(\cC(Q))$, and since $\cC(Q)$ is finite, $v\in \bigcup(\cC(Q))$. Since the closure of $\interior(Q)$ is $Q$, we have $Q\subset \bigcup(\cC(Q))$; in fact, $\bigcup(\cC^{n}(Q))\subset \bigcup(\cC^{n+1}(Q))$ for all $n$.

Let $\mathcal{S}_{i}\subset \mathcal{Q}_i$ be the set
\begin{align*}
  \mathcal{S}_{i} = \bigcup_{Q\in \mathcal{S}_{i-1}} \cC(Q) \cup \left\{Q \in \mathcal{Q}_i : \max_{x \in Q_{i,j}} |\nabla_{f_i} f_i(x)| \geq \frac{1}{2}\right\},
\end{align*}
where $\mathcal{S}_{-1}=\emptyset$. Let $S_{i}=\bigcup(\mathcal{S}_{i})$ and let $J_i=\{j:Q_{i,j}\not\in \mathcal{S}_i\}$.
Then $S_{i+1}\supset S_i$ for all $i$.
Let $\nu_i := \sum_{j\in J_i} \kappa_{i,j} = \one_{S_i^c}\kappa_i$
Since $\kappa_{i,j}$ is zero on a neighborhood of $\partial Q_{i,j}$, this is smooth. We define $f_{i+1} = f_i + \nu_i$ and $\psi_{i+1} = \psi_i + \kappa_i$.

Note that $S_i\subset S_j$ for all $i < j$, so $\nu_{i}$ is zero on a neighborhood of $S_i$. Therefore,
\begin{align}\label{eq:fi-equals-fj}
  f_j|_{S_i} &= f_{i}|_{S_i} &
  \nabla_{f_j}f_j|_{S_i} &= \nabla_{f_{i}}f_{i}|_{S_i}. 
\end{align}
Conversely, if $v\not\in S_{i-1}$, then for all $k\le i-1$, $v\not \in S_k$ and $\kappa_k(v) = \nu_k(v)$, so
\begin{equation}\label{eq:f-equals-psi}
  f_{i}|_{S_{i-1}^c} = \sum_{k=0}^{i-1} \nu_k|_{S_{i-1}^c} = \sum_{k=0}^{i-1} \kappa_k|_{S_{i-1}^c} = \psi_{i}|_{S_{i-1}^c}.
\end{equation}

The functions $\psi_i$ are exactly the same as those defined in Section 3.2 of \cite{NY2} and our $\kappa_i$ correspond to their $\nu_i$. We will show that if $\rho$ is sufficiently large, $\epsilon>0$ is sufficiently small, and $i<\epsilon N^4$, then $f_i$ is intrinsic Lipschitz and the set on which $f_i$ and $\psi_i$ differ is small.

\begin{prop} \label{p:stopping-time}
  Let $A > 1$ be sufficiently large. If $\rho$ is sufficiently large (depending on $A$), then for each $i$, $f_i$ is a smooth function supported in $U$ such that $\|\nabla_{f_i} f_i\|_\infty \le 1$. In particular, $f_i$ is intrinsic Lipschitz. Furthermore, $f_i|_{S_{i-1}^c} = \psi_i|_{S_{i-1}^c}$ and $\mu(S_{i}) \lesssim iA^{-4}$.
\end{prop}

By \eqref{eq:f-equals-psi}, it suffices to show that $\|\nabla_{f_i} f_i\|_\infty \le 1$ and $\mu(S_{i-1}) \lesssim i A^{-4}$.

We will need some bounds from \cite{NY2}.
As in \cite{NY2}, let
\begin{align*}
  D_i &= \nabla_{\psi_{i+1}} \psi_{i+1} - \nabla_{\psi_i}\psi_i &
  \tilde{D}_i &= \nabla_{f_{i+1}} f_{i+1} - \nabla_{f_i}f_i.
\end{align*}
By  \eqref{eq:fi-equals-fj} and \eqref{eq:f-equals-psi}, we have $f_{i+1}|_{S_{i}^c} = \psi_{i+1}|_{S_{i}^c}$ and $f_{i+1}|_{S_{i}} = f_i|_{S_{i}}$. Therefore,
$\tilde{D}_i = \one_{S_{i}^c} D_i$. In particular, $\|\tilde{D}_i\|_\infty\le\|D_i\|_\infty$.

The following bounds on $\tilde{D}_i$ are based on the bounds on $D_i$ proved in \cite{NY2}.
\begin{lemma} \label{l:D-holder}
  Let $j \leq i$ and $x,y \in Q_{i,k}$.  Then
  \begin{align*}
    |\tilde{D}_j(x) - \tilde{D}_j(y)| \lesssim A^{-2} \rho^{j-i}.
  \end{align*}
\end{lemma}

\begin{proof}
  The $D_j$--version of this inequality is Lemma~3.12 of \cite{NY2}.  The proof only uses the $L_\infty$ bounds of $D_j$ and derivatives of $D_j$.  As $\tilde{D}_j$ satisfy those same bounds, the proof also works for $\tilde{D}_j$.
\end{proof}

\begin{lemma} \label{l:NY-D-bounds}
  For every $\rho \geq 8$ and $A \geq 1$, we have
  \begin{align}
\notag    \|\tilde{D}_i\|_{\infty} &\lesssim A^{-2}, \qquad \forall i \geq 0, \\
\label{eq:new-D-orthog}    |\langle \tilde{D}_i, \tilde{D}_{j} \rangle| &\lesssim A^{-4} \rho^{j - i}, \qquad \forall 0 \leq i \leq j.
  \end{align}
\end{lemma}

\begin{proof}
  The corresponding $D_i$ version of the inequalities is Lemma~3.9 from \cite{NY2}.  The first inequality now follows from the bound $\|\tilde{D}_i\|_\infty\le\|D_i\|_\infty$. The proof of the second bound in \cite{NY2} uses the bound
  \begin{equation*}
    \left|\int_{Q_{n,k}} D_m(w) D_n(w) \ud w\right|\lesssim A^{-4} \rho^{m-n} \cH^3(Q_{n,k})
  \end{equation*}
  for $n\ge m$. 
  
  Let $0\le i\le j$. For each $1\le k\le k_j$, we consider two cases. If $Q_{j,k}\in \cS_j$, then $\tilde{D}_j =  0$ on $Q_{j,k}$, so $\int_{Q_{j,k}} \tilde{D}_i \tilde{D}_j \ud w = 0$. Otherwise, if $Q_{j,k}\not \in \cS_j$, then $\interior Q_{j,k} \cap S_i = \interior Q_{j,k} \cap S_j = \emptyset$, so $\tilde{D}_i = D_i$ and $\tilde{D}_j = D_j$ on $Q_{j,k}$. Therefore,     
  $$\left|\int_{Q_{j,k}} \tilde{D}_i(w) \tilde{D}_j(w) \ud w\right|\lesssim A^{-4} \rho^{j-i} \cH^3(Q_{j,k}).$$
  Since the $Q_{j,k}$'s partition $U$, we sum this inequality over $k$ to obtain \eqref{eq:new-D-orthog}.
\end{proof}

Now we use these bounds to show that $\|\nabla_{f_i} f_i\|_\infty \le 1$. 
\begin{lemma}\label{l:stopping-time-Lip}
  If $A$ is sufficiently large, then for all $i$, $\|\nabla_{f_i} f_i\|_\infty \le 1$.
\end{lemma}
\begin{proof}
  We suppose that $A$ is large enough that $\|\tilde{D}_i\|_{\infty} \le \frac{1}{4}$ for all $i$ and proceed by induction on $i$. Since $f_0=0$, the lemma is clear for $i=0$.
  
  Suppose that $i\ge 0$ and $\|\nabla_{f_i} f_i\|_\infty \le 1$. On one hand, if $v\not \in S_i$, then $|\nabla_{f_{i}}f_{i}(v)|\le \frac{1}{2}$, and 
  $$|\nabla_{f_{i+1}}f_{i+1}(v)|\le |\nabla_{f_{i}}f_{i}(v)| + \|\tilde{D}_i\|_{\infty} < 1.$$
  On the other hand, if $v\in S_i$, then $|\nabla_{f_{i+1}} f_{i+1}(v)| = |\nabla_{f_{i}} f_{i}(v)| \le 1$ by  \eqref{eq:fi-equals-fj}.
\end{proof}

It remains to bound $\mu(S_i)$. Let $i\ge 0$. Recall that $S_i=\bigcup(\cS_i)$ and that any pseudoquad $Q\in \cS_i$ either satisfies $\|\nabla_{f_i} f_i\|_{L_\infty(Q)}\ge \frac{1}{2}$ or is a child of some pseudoquad of $\cS_{i-1}$. Let
$$\cM_{i} = \mathcal{S}_{i} \setminus \bigcup_{Q\in \mathcal{S}_{i-1}} \cC(Q).$$
Then if $M\in \cM_{i}$, then $\|\nabla_{f_{i-1}} f_{i-1}(x)\|_{L_\infty(M)}\ge \frac{1}{2}$. If $Q\in \cS_i\setminus \cM_i$, then $Q$ is a child of an element of $\cS_{i-1}$. By induction, any $Q\in \cS_i$ is a descendant of an element of $\cS_j$ for some $j\le i$, i.e., $Q$ is a descendant of an element of $\mathcal{B}_{i} := \bigcup_{j=0}^{i} \cM_{j}$. Furthermore, if $M,M'\in \mathcal{B}_i$, $M\ne M'$, then neither is a descendant of the other, so $M$ and $M'$ are disjoint. 

We will thus bound $\mu(S_i)$ by bounding the size of $\cB_i$, then bounding the size of the set of descendants of pseudoquads in $\cB_i$. We bound $\cB_{i}$ by showing that $\nabla_{f_i}f_i$ is large on the pseudoquads in $\cB_{i}$. 
\begin{lemma} \label{l:nabla-min-max}
  Suppose $\rho$ is sufficiently large.
  Let $Q\in \cB_{i}$.  Then $|\nabla_{f_i}f_i(v)|\ge \frac{1}{4}$ for all $v\in Q$.
\end{lemma}

\begin{proof}
  If $Q\in \cB_{i}$, then $Q\in \cM_{j}$ for some $j\le i$, so $\|\nabla_{f_j} f_j\|_{L_\infty(Q)}\ge \frac{1}{2}$. Furthermore, since $Q\subset S_{j}$, \eqref{eq:fi-equals-fj} implies that $\nabla_{f_j}f_j = \nabla_{f_i}f_i$ on $Q$.
  
  Let $y\in Q$ be such that $|\nabla_{f_j}f_j(y)| \geq 1/2$, and let $x\in Q$. By Lemma \ref{l:D-holder},
  \begin{multline*}
    |\nabla_{f_i}f_i(x) - \nabla_{f_i}f_i(y)| = |\nabla_{f_j}f_j(x) - \nabla_{f_j}f_j(y)| \\ \leq \sum_{k=0}^{j-1} |\tilde{D}_k(x) - \tilde{D}_k(y)| \lesssim \sum_{k=0}^{j-1} A^{-2} \rho^{k-j} \le 2 A^{-2}\rho^{-1}.
  \end{multline*}
  If $\rho$ is sufficiently large, this gives $|\nabla_{f_i}f_i(x) - \nabla_{f_i}f_i(y)|$ and thus $|\nabla_{f_i}f_i(x)|\ge \frac{1}{4}$, as desired.
\end{proof}

Thus, we can bound the size of $\cM_k$ using the following bound on $\nabla_{f_k}f_k$.
\begin{lemma}\label{l:l2-fi}
  For all $k$, 
  $$\|\nabla_{f_k}f_k\|_2 \lesssim A^{-2} \sqrt{k}.$$
\end{lemma}
\begin{proof}
  By Lemma~\ref{l:NY-D-bounds},
  \begin{multline*}
    \int_{V_0} |\nabla_{f_k}f_k(x)|^2 \ud x = \int_{V_0} \left( \sum_{j=0}^{k-1} \tilde{D}_j \right)^2 \ud x \\
    = \sum_{j=0}^{k-1} \|\tilde{D}_j\|_{L_2}^2 + 2 \sum_{0 \leq i < j \leq k-1} \langle \tilde{D}_i, \tilde{D}_j \rangle 
    \lesssim kA^{-4} + \sum_{i=0}^{k-1} \sum_{k=1}^\infty A^{-4} \rho^{-k} \lesssim kA^{-4},
  \end{multline*}
  so $\|\nabla_{f_k}f_k\|_2 \lesssim A^{-2} \sqrt{k}$.
\end{proof}

Now we bound the size of the set of descendants of a pseudoquad.
We will need the following lemma, which is part of Lemma 3.10 of \cite{NY2}.
\begin{lemma} \label{l:R-vert-bound}
  Let $R_z$ be the $z$-coordinate of any of the maps $R_{i,j}$. If $\rho>8$, then for all $(s,t)\in [0,Ar_i] \times [0,r_i^2]$,
  \begin{align*}
    \frac{3}{4} \leq \frac{\partial R_z}{\partial t}(s,t) \leq \frac{4}{3}.
  \end{align*}
\end{lemma}

The following bound on the heights of pseudoquads follows immediately.

\begin{lemma} \label{l:cube-height}
  Let $i\ge 0$ and let $1 \le j \le k_i$. Let $I\subset \R$ and $g_1,g_2\from I\to \R$ be such that
  $$Q_{i,j} = \{(x,0,z)\mid x\in I, z\in [g_1(x),g_2(x)]\}.$$
  Then 
  \begin{align}
    \frac{3}{4} r_i^2\le g_2(x)-g_1(x) \le \frac{4}{3} r_i^2, \qquad \forall x \in I. \label{eq:pseudoquad-height}
  \end{align}
\end{lemma}

\begin{proof}
  Let $x_0= \min(I)$. Then in fact, $Q_{i,j}$ is the image of a map $R_{i,j}\from [0,Ar_i]\times [0,r_i^2] \to V_0$ such that $R_{i,j}(s,t) = (x_0 + s, 0, R_z(s,t))$.  In particular, $g_1(x) = R_z(x-x_0,0)$ and $g_2(x) = R_z(x-x_0, r_i^2)$. The Mean Value Theorem along with Lemma \ref{l:R-vert-bound} then gives the desired bound.
\end{proof}

As each $Q_{i,j}$ has width $Ar_i$, we immediately get the following corollary.

\begin{cor} \label{c:cube-area}
  For any $i,j \geq 0$, we have that $\frac{3}{4} A r_i^3\le |Q_{i,j}| \le \frac{4}{3} Ar_i^3$.
\end{cor}

For each pseudoquad $Q$, let $\tilde{Q}=\bigcup(\cD(Q))$ so that
\begin{equation}\label{eq:Si-desc}
  S_{i}\subset \bigcup_{Q\in \cB_{i}} \tilde{Q}.
\end{equation}
Our next lemma bounds $\mu(\tilde{Q})$.

\begin{lemma} \label{l:neighbor-expand}
  For any $i$ and any $Q\in \cQ_i$, $\mu(\tilde{Q}) \lesssim \mu(Q)$.
\end{lemma}
\begin{proof}
  Let $I\subset \R$ and $g_1,g_2\from I\to \R$ be such that $I$ is an interval of length $Ar_i$ and 
  $$Q=\left\{(x,0,z)\in V_0\mid x\in I, z\in [g_1(x),g_2(x)]\right\}.$$
  We consider $\bigcup(\cC(Q))$. If $Q'$ is a child of $Q$, then 
  $$Q'=\left\{(x,0,z)\in V_0\mid x\in I', z\in [g'_1(x),g'_2(x)]\right\}$$
  for some $I',g_1'$, and $g_2'$ such that $I'\subset I$. 
  
  By our choice of $Q$ and $Q'$, the top and bottom curves of $Q$ are characteristic curves of $\psi_i$ and the top and bottom curves of $Q'$ are characteristic curves of $\psi_{i+1}$.  Since $\kappa_i$ is 0 on a neighborhood of $\partial Q$, we have $\psi_i=\psi_{i+1}$ on $\partial Q$, so the top and bottom curves of $Q$ are also characteristic curves of $\psi_{i+1}$. Since $\psi_{i+1}$ is smooth, its characteristic curves don't intersect, so the top and bottom edges of $Q'$ don't cross $\partial Q$. Thus, since there is some $x\in I'$ such that $[g_1(x),g_2(x)]$ intersects $[g_1'(x),g_2'(x)]$, it must be true that $[g_1(x),g_2(x)]$ intersects $[g_1'(x),g_2'(x)]$ for all $x\in I'$. By Lemma~\ref{l:cube-height}, this implies
  $[g_1'(x),g_2'(x)]\subset [g_1(x)-\frac{4}{3} r_{i+1}^2,g_2(x)+\frac{4}{3} r_{i+1}^2]$ and thus
  $$\bigcup(\cC(Q)) \subset \left\{(x,0,z)\in V_0: x\in I, z\in \left[g_1(x)-\frac{4}{3} r_{i+1}^2,g_2(x)+\frac{4}{3}r_{i+1}^2\right]\right\}.$$
  By induction, 
  $$\tilde{Q} \subset \left\{(x,0,z)\in V_0: x\in I, z\in \left[g_1(x)-\sum_{j=i}^\infty \frac{4}{3} r_{j+1}^2,g_2(x)+\sum_{j=i}^\infty \frac{4}{3} r_{j+1}^2\right]\right\}.$$
  The upper and lower bounds are geometric series, so by Corollary~\ref{c:cube-area},
  $$\mu(\tilde{Q}) \le \mu(Q) + \frac{16}{3} r_{i+1}^2\cdot Ar_i \le \mu(Q) + \frac{16}{3} \rho^{-2} Ar_{i}^3\le 4 \mu(Q).$$
\end{proof}

Finally, we prove the proposition.
\begin{proof}[Proof of Proposition \ref{p:stopping-time}]
  Let $i\ge 0$. By Lemma~\ref{l:stopping-time-Lip}, we have $\|\nabla_{f_i} f_i\|_\infty\le 1$. It remains to bound the measure of $S_{i}$.
  
  By \eqref{eq:Si-desc}, we have 
  $S_{i}\subset \bigcup_{Q\in \cB_{i}} \tilde{Q},$ where $\cB_i$ is a collection of disjoint pseudoquads. Furthermore, by Lemma~\ref{l:nabla-min-max}, we have $|\nabla_{f_i}f_i(v)|\ge \frac{1}{4}$ for all $v\in \bigcup(\cB_i)$. By Lemma~\ref{l:neighbor-expand},
  $$\mu(S_i)\le \sum_{Q \in \mathcal{B}_i} \mu(\tilde{Q}) \lesssim \sum_{Q \in \mathcal{B}_i} \mu(Q) = \mu\left(\bigcup(\cB_i)\right).$$
  By Chebyshev's Inequality and Lemma~\ref{l:l2-fi}, 
  \begin{align*}
    \mu\left(\bigcup(\cB_i)\right) \le 16 \|\nabla_{f_i}f_i\|_2^2 \lesssim iA^{-4},
  \end{align*}
  so $\mu(S_i)\lesssim iA^{-4}$, as desired.
\end{proof}

In addition to the intrinsic Lipschitz condition, $f_i$ satisfies a higher-order Sobolev condition. We state this condition in terms of a family of differential operators on smooth functions $V_0 \to \R$.  Let $Z$ be the operator $Z = \frac{\partial}{\partial z}$. The pseudoquads in $\cQ_i$ have width $Ar_i$ and height $r_i^2$, and we define rescaled operators
\begin{align*}
  \hat{Z}_i &= r_i^2 Z &
  \hat{\partial}_i &= Ar_i \nabla_{f_i}.
\end{align*}

For $i \geq 0$ and $n \geq 1$, we let $\{\hat{Z}_i, \hat{\partial}_i\}^n$ denote the differential operators $E$ that can be expressed as $E = E_1 \cdots E_n$
where $E_j \in \{\hat{Z}_i, \hat{\partial}_i\}$ for all $1 \leq j \leq n$. We call these \emph{words of length $n$} in the alphabet $\{\hat{Z}_i, \hat{\partial}_i\}$. As a special case, $\{\hat{Z}_i, \hat{\partial}_i\}^0=\{\id\}$.

The following lemma bounds $Ef_i$ when $E\in \{\hat{Z}_i, \hat{\partial}_i\}^k$.  This generalizes the bounds in \cite[Lemma 3.10]{NY2}.

\begin{lemma}\label{lem:word-sup-bound}
  Given $d \geq 2$, there exists $\rho_0 > 0$ so that if $\rho \geq \rho_0$, $i \geq 0$, $k \leq d$, and $E \in \{\hat{Z}_i, \hat{\partial}_i\}^k$, then
  \begin{align*}
    \| E \nu_i \|_\infty \lesssim_d A^{-1}r_i.
  \end{align*}
  Furthermore, if $E \not\in \{\id, \hat{\partial}_i\}$, then
  \begin{align*}
    \| Ef_i \|_\infty \lesssim_d A^{-1} r_i \rho^{-1}.
  \end{align*}
  In particular,
  \begin{align}\label{eq:flatness-fi}
    \| \nabla_{f_i}^2 f_i \|_\infty & \lesssim A^{-3} r_i^{-1} \rho^{-1} & \| Z f_i \|_\infty & \lesssim A^{-1} r_i^{-1} \rho^{-1}. 
  \end{align}
\end{lemma}
The coefficients in this lemma are related to the dimensions of the pseudoquads in $\cQ_i$. As noted above, these pseudoquads have width and height corresponding to $\hat{\partial}_i$ and $\hat{Z}_i$. The coefficient $A^{-1}r_i$ comes from the fact that $\|\nu_i\|_\infty = A^{-1}r_i\|\kappa\|_\infty \approx A^{-1}r_i$. Thus, when $\rho$ is large, $f_i$ is close to affine on any of the pseudoquads in $\cQ_i$. 

The proof of Lemma \ref{lem:word-sup-bound} is rather technical, and we leave it to Appendix \ref{ap:A}.

\subsection{Rescaling}\label{subsec:rescaling}
Let $\Sigma_i = \Gamma_{f_i}$. Because the singular integrals we consider are scale-invariant and translation-invariant, it will be convenient to define rescaled and translated versions of $f_i$ and $\nu_i$. Let $i\ge 0$ and let $p_0\in \Sigma_i$. 

Let $s_i:=s_{r_i^{-1}}$. Let  $\alpha=\alpha_{p_0,i}\from \HH\to \R$,
$$\alpha(p) = r_i^{-1} \left(-y(p_0) + f_i(p_0 s^{-1}_{i}(p))\right).$$
By Lemmas~\ref{lem:ilg-translations} and \ref{lem:ilg-scaling}, we have $\Gamma_\alpha = s_i(p_0^{-1} \Sigma_i)$ and 
$$\Psi_\alpha(p) = s_i\left(p_0^{-1} \Psi_{f_i}(p_0 s^{-1}_i(p))\right).$$
In particular, we have $\zero\in \Gamma_\alpha$ and $\alpha(\zero)=0$.

Let $\gamma=\gamma_{p_0,i}\from \HH\to \R$,
$\gamma(p) = r_i^{-1} \nu_i(p_0 s^{-1}_{i}(p)).$
Then for any $t\in \R$, we have
$$\alpha(p) + t \gamma(p) = r_i^{-1}(-y(p_0) + (f_i+t \nu_i)(p_0 s^{-1}_{i}(p)))$$
and
$$\Gamma_{\alpha + t\gamma} = s_{i}(p_0^{-1}\Gamma_{f_i + t \nu_i}).$$

These functions satisfy the following consequence of Lemma \ref{lem:word-sup-bound}.

\begin{lemma}\label{lem:gamma-deriv-bounds}
  There exists $\rho_0 > 0$ and $c>0$ such that if $\rho \geq \rho_0$, $i \geq 0$, $k \leq 3$, and $\alpha$ and $\gamma$ are defined as above for some $p_0\in \Sigma_i$, then $\|\gamma\|_\infty\le c A^{-1}$, $\|\nabla_\alpha \alpha\|_\infty \le 1$, and
  \begin{equation}\label{eq:gamma-jk-bounds}
      \|F \gamma\|_{\infty} \le c A^{-\# \nabla_\alpha(F)-1}, \qquad \forall F \in \{\nabla_\alpha, Z\}^k,
  \end{equation}
  where \(\# \nabla_\alpha(F)\) is the number of occurrences of \(\nabla_\alpha\) in $F$.
  Moreover, if $F \notin \{\id, \nabla_\alpha\}$, then
  \begin{equation}\label{eq:alpha-jk-bounds} 
    \|F \alpha\|_{\infty} \le c A^{-\# \nabla_\alpha(F)-1} \rho^{-1}.
  \end{equation}
\end{lemma}

\begin{proof}
  In fact, we will show that for any $d\ge 2$, there is a $\rho_0$ such that if $\rho \geq \rho_0$, $i \geq 0$, then \eqref{eq:gamma-jk-bounds} and \eqref{eq:alpha-jk-bounds} hold for all $k \leq d$.
  Let $\hat{\partial}=A\nabla_\alpha$. It suffices to show that if $\rho$ is sufficiently large, then $\|F \gamma\|_{\infty} \lesssim_d A^{-1}$ for all $F \in \{\hat{\partial}, Z\}^k$ and, if $F\notin\{\id,\hat{\partial}\}$, $\|F \alpha\|_{\infty} \lesssim_d \rho^{-1}$.
  
  Let $r=r_i$ and $s=s_{r^{-1}}$. Let $L(g) = s(p_0^{-1}g)$ so that $\Gamma_\alpha = L(\Sigma_i)$
  and 
  \begin{align*}
    \alpha(p) &= r^{-1}\left(-y(p_0) + (f_i \circ L^{-1})(p)\right), \\
    \gamma(p) &= r^{-1} \nu_i \circ L^{-1}(p).
  \end{align*}  
  Since $L$ sends horizontal curves in $\Sigma_i$ to horizontal curves in $\Gamma_\alpha$, it sends integral curves of $\nabla_{f_i}$ to integral curves of $\nabla_\alpha$. Therefore, $L_* \nabla_{f_i} = r^{-1} \nabla_\alpha$, and
  \begin{align}
    L_* \hat{\partial}_i = L_*(A r \nabla_{f_i}) = A r r^{-1} \nabla_\alpha = \hat{\partial}. \label{e:pushforward}
  \end{align}
  Likewise, $L_* \hat{Z}_i = L_*(r^2 Z) = Z$.

  Let $F=F(\hat{\partial}, Z)$ be a word of length at most $d$ in the letters $\hat{\partial}$ and $Z$ and let $F'=F(\hat{\partial}_i, \hat{Z}_i)$ be $F$ with $\hat{\partial}$ replaced by $\hat{\partial}_i$ and $Z$ by $\hat{Z}_i$. Then $L_*(F')=F$, so
  $$F \gamma = F[r^{-1} (\nu_i \circ L^{-1})] = r^{-1} F'[\nu_i]\circ L^{-1}.$$
  Lemma~\ref{lem:word-sup-bound} implies that if $\rho$ is sufficiently large, then $\|F\gamma\|_\infty = r^{-1} \|F'[\nu_i]\|_\infty \lesssim_d A^{-1}$. This proves \eqref{eq:gamma-jk-bounds}.
  
  Similarly, if $F\ne \id$, then 
  $$F \alpha = r^{-1} F[-y(p_0) + f_i \circ L^{-1}] = r^{-1} F'[f_i]\circ L^{-1}.$$
  If $F=\hat{\partial}$, this implies that
  $$\nabla_\alpha \alpha = A^{-1} \hat{\partial}\alpha = A^{-1}r^{-1} \hat{\partial}_i f_i \circ L^{-1}= \nabla_{f_i}f_i \circ L^{-1},$$
  so $\|\nabla_\alpha \alpha\|_\infty = \|\nabla_{f_i} f_i\|_\infty \leq 1$.
  Otherwise, if $F\not\in\{\id,\hat{\partial}\}$, Lemma~\ref{lem:word-sup-bound} implies that $\|F \alpha\|_\infty\lesssim_d A^{-1} \rho^{-1}\le \rho^{-1}$. This proves \eqref{eq:alpha-jk-bounds}.
\end{proof}

\section{Lower bounds on $\beta$--numbers}\label{sec:beta-numbers}

In this section, we prove Theorem~\ref{thm:beta-theorem}.
In fact, Theorem~\ref{thm:beta-theorem} is an immediate consequence of  the following bound.
\begin{prop}\label{prop:beta-theorem}
  There is a $\delta_0 > 0$ with the following property. Let $0<\delta<\delta_0$, $A>1$, and $p>0$. If $\rho>1$ is sufficiently large, $N=\lfloor \delta A^4\rfloor$, $f_i$ is constructed as in Section~\ref{sec:construction}, $\Gamma = \Gamma_{f_N}$, and $U=[0,1]\times\{0\}\times [0,1]$, then 
  $$\int_{0}^R\int_{\Psi_{f_N}(U)}  \beta_{\Gamma}(v,r)^p \ud v \frac{\ud r}{r}\gtrsim N A^{-p} \gtrsim \delta A^{4-p}.$$
\end{prop}

We prove this by introducing a parametric version of $\beta_\Gamma(v, r)$. For any measurable function $\psi\from \HH\to\R$ which is constant on cosets of $\paramY$, we define $V(v,r) =\Pi(B(v,r))$ and 
$$\gamma_\psi(v,r) = r^{-4} \inf_{h\in \Aff} \|\psi - h\|_{L_1(V(v,r))},$$
where $\Aff$ denotes the set of functions of the form $\alpha(v) = ax(v) + b$, $a,b\in \R$. Note that all vertical planes that are not parallel to the $yz$--plane are graphs of functions in $\Aff$.  

When $\psi$ is intrinsic Lipschitz, $\beta_{\Gamma_\psi}$ and $\gamma_\psi$ are comparable.
\begin{lemma}\label{l:parametric-beta}
  Let $\lambda\in (0,1)$. There is a $c>1$ such that for any $\lambda$--intrinsic Lipschitz function  $\psi\from \HH\to \R$, any $x\in \Gamma_\psi$, and any $r>0$,
  $$\beta_{\Gamma_\psi}(x,c^{-1} r) \lesssim_\lambda \gamma_\psi(x,r) \lesssim_\lambda \beta_{\Gamma_\psi}(x,c r).$$
\end{lemma}
The proof of this lemma uses the fact that if $h(v) = a x(v) + b$ is affine and $P=\Gamma_h$ is the corresponding vertical plane, then $d_\Kor(w,P) \approx_a |y(w)-h(w)|$ for all $w\in \HH$. Since the constant in this inequality depends on $a$, we will need the following lemma.

\begin{lemma} \label{l:ilg-planes}
  Let $\lambda\in (0,1)$. There exist $m > 0$ and $\epsilon > 0$ such that for any  $\lambda$\nobreakdash--intrinsic Lipschitz graph $\Gamma_\psi$, any $u\in \Gamma_\psi$ and any vertical plane $P$, if \begin{equation}\label{eq:ilg-planes}
    \int_{B(u,r) \cap \Gamma_\psi} d_\Kor(w,P) \ud \mu(w) < \epsilon r^4,
  \end{equation}
  then $|\slope P|<m$.
\end{lemma}

\begin{proof}
  Since $\Gamma_\psi$ is Ahlfors $3$--regular, there is a $c=c(\lambda)>0$ such that $\mu(\Gamma_\psi\cap B(w,s)) \ge c s^3$ for all $w\in \Gamma_\psi$ and $s>0$. 
 
  Let $\epsilon = c\delta^4$, let $\delta = \frac{1-\lambda}{100}$, let $m=(4\delta)^{-1}$, and let $P$ be a vertical plane satisfying \eqref{eq:ilg-planes}. 
  We claim that if $v\in \Gamma_\psi \cap B(u,\frac{r}{2})$, then $d_\Kor(v, P) < 2 \delta r$. Suppose not. Then $B(v,\delta r) \subset B(u,r)$ and 
  $$d_\Kor(\Gamma_\psi\cap B(v,\delta r),P)\ge \delta r,$$ so
  $$\int_{B(u,r) \cap \Gamma_\psi} d_\Kor(w,P) \ud \mu(w) \ge  \mu(\Gamma_\psi\cap B(v,\delta r)) \delta r \ge c\delta^4 r^4 = \epsilon r^4.$$
  This is a contradiction, so $d_\Kor(v, P) < 2 \delta r$. 
  
  In particular, this implies that $d_\Kor(u, P) < 2 \delta r$.   
  By Lemma \ref{lem:ilg-line-distance}, $|\psi(u X^{x}) - \psi(u)| \le \frac{2}{1-\lambda} |x|$ for any $x\in \R$, so if $v:= \Psi_\psi(u X^{8 \delta r})$, then 
  $$d_\Kor(u,v) \le 8 \delta r + |\psi(u X^{8 \delta r}) - \psi(u)| \le \frac{8}{100} r + \frac{16}{100} r \le \frac{r}{4}.$$
  That is, $v\in \Gamma_\psi \cap B(u, \frac{r}{2})$, so $d_\Kor(v, P) < 2 \delta r$.
  
  Let $u'\in P\cap B(u,2 \delta r)$ and let $v'\in P\cap B(v,2 \delta r)$. Then 
  $$|x(v') - x(u')| \ge |x(v)-x(u)| - 4\delta r = 4\delta r$$
  and
  $$|y(v') - y(u')| \le |y(v)-y(u)| + 4\delta r \le \frac{r}{4} + 4\delta r \le r.$$
  Thus
  $$|\slope(P)| = \frac{|y(v') - y(u')|}{|x(v') - x(u')|} \le (4\delta)^{-1} = m,$$
  as desired.
\end{proof}

We now prove Lemma \ref{l:parametric-beta}.
\begin{proof}[Proof of Lemma \ref{l:parametric-beta}]
  Let $\Gamma=\Gamma_\psi$, let $x\in \Gamma$, and $r>0$.
  Let $c=c(\lambda)>0$ be as in  Lemma~\ref{lem:proj-ball-contain}, so that $V(p,s) \subset \Pi(B(p,cs)\cap \Gamma)$ for all $p\in \Gamma$ and $s>0$.
  Note that by the area formula, we have
  $$\cH^3(S) \approx \int_{\Pi(S)} \sqrt{1+\nabla_\psi \psi(v)} \ud v = \int_{S} \sqrt{1+\nabla_\psi \psi(v)} \ud v \approx_\lambda \mu(S)$$
  for any Borel set $S\subset \Gamma$.

  We first prove that $\beta_{\Gamma}(x, c^{-1} r) \lesssim_\lambda \gamma_\psi(x,r)$.
  Let $h \from V_0 \to \R$ be an affine function so that $r^{-4} \|\psi - h\|_{L_1(V(x,r))} \leq 2\gamma_\psi(x,r)$ and let $P=\Gamma_h$. Then $d_\Kor(v,P)\le |\psi(v)-h(v)|$ for all $y\in \Gamma$ and $\Pi(B(x,c^{-1}r) \cap \Gamma)\subset V(x,r)$, so
  \begin{align*}
    \beta(x,c^{-1}r) & \lesssim_\lambda r^{-4} \int_{B(x,c^{-1}r) \cap \Gamma} d_\Kor(v,P) \ud \mu(v) \\
    &\leq r^{-4} \int_{V(x,r)} |\psi(v) - h(v)| \ud \mu(v) \\
    &\leq 2 \gamma_\psi(x,r).
  \end{align*}
  
  Next, we show that $\gamma_\psi(x,r)\lesssim_\lambda \beta_{\Gamma}(x,cr)$. Let $m=m(\lambda),\epsilon=\epsilon(\lambda)$ be as in Lemma~\ref{l:ilg-planes}.
  Suppose first that $\beta_\Gamma(x, cr) < \frac{\epsilon}{2}$. Then there is a vertical plane $P$ that satisfies \eqref{eq:ilg-planes} and thus $|\slope(P)| < m$. Let $g$ be the affine function such that $\Gamma_g=P$; then $d_\Kor(v,P)\approx_m |g(v)-\psi(v)|$ for all $v\in \Gamma$ (Lemma~\ref{lem:y-distance}). Therefore, since $V(x,r)\subset \Pi(B(p,cr)\cap \Gamma)$,
  \begin{align*}
    \gamma_\psi(x,r) &\leq r^{-4}  \int_{V(x,r)} |g(v) - \psi(v)| \ud \mu(v) \\
    &\lesssim_\lambda r^{-4} \int_{B(x,c r) \cap \Gamma} d_\Kor(v, P) \ud \mu(v) \\
    &\lesssim_\lambda \beta_\Gamma(x, c r).
  \end{align*}

  Now suppose $\beta_\Gamma(x, c r) \ge \frac{\epsilon}{2}\gtrsim_\lambda 1$. Let $h$ be the constant (affine) function $h(v) = \psi(x)$. By Lemma~\ref{lem:ilg-line-distance}, for $y\in B(x,r)$, we have 
  $$|h(\Pi(y)) - \psi(\Pi(y))| = |\psi(x) - \psi(y)| \le \frac{2}{1-\lambda} r,$$
  so $|h(v) - \psi(v)|\lesssim_\lambda r$ for all $v\in V(x,r)$. Therefore,
  $$\gamma_\psi(x,r) \le r^{-4} \|f - h\|_{L_1(V(x,r))} \lesssim_\lambda r^{-4} \cdot r \mu(V(x,r)) \lesssim 1 \lesssim_\lambda \beta_\Gamma(x,r),$$
  as desired.
\end{proof}

We can thus prove Proposition~\ref{prop:beta-theorem} by bounding $\gamma_{\nu_i}$ and $\gamma_{f_i}$. We will prove the following.
\begin{lemma}\label{l:gamma-number-bounds}
  For any $A>1$, the following properties hold for all sufficiently large $\rho$. 
  Let $i < k$. Let $v\in \Gamma_{f_k}$ and let $b>0$. Then
  \begin{equation}\label{eq:gamma-upper}
    \gamma_{f_i}(v, b r_i) \lesssim_b A^{-1} \rho^{-1}.
  \end{equation}
  
  Let $J_i$ be as in Section~\ref{sec:construction}, let $j \in J_i$, and let $s_0\in [\frac{1}{3} Ar_i,\frac{2}{3} A r_i]$, $t_0\in [0,r_i^2]$, and $w=\Psi_{f_k}(R_{i,j}(s_0,t_0))$. Then
  \begin{equation}\label{eq:gamma-lower}
    \gamma_{\nu_i}(w, 8 r_i) \gtrsim A^{-1}.
  \end{equation}
\end{lemma}
\begin{proof}
  First, we prove \eqref{eq:gamma-upper}. Let $L\from \HH\to \R$ be the affine function $L(p) = f_i(v) + (x(p) - x(v)) \nabla_{f_i} f_i(v)$ and let $v'=\Psi_{f_i}(v)$. Lemma~\ref{lem:a-Taylor} and Lemma~\ref{lem:word-sup-bound} applied to $f_i$ imply that for all $u \in \HH$
  \begin{align*} 
    |f_i(u) - L(u)| \lesssim_\lambda d_\Kor(v',u)^2 (\|\nabla_{f_i}^2 f_i\|_\infty + \|Zf_i\|_\infty) \overset{\eqref{eq:flatness-fi}}{\lesssim} d_\Kor(v',u)^2 A^{-1} \rho^{-1} r_i^{-1}.
  \end{align*}
  Since
  $$d_\Kor(v',v)\le |f_i(v)-f_k(v)| \le \sum_{m=i}^{k-1} \|\nu_m\|_\infty \le 2 A^{-1} r_i \|\kappa\|_\infty \le 2r_i,$$
  if $u\in B(v,br_i)$, then $d_\Kor(u,v')\le (b+2)r_i$ and 
  $$|f_i(u) - L(u)| \le (b+2)^2 A^{-1} \rho^{-1} r_i\lesssim_b A^{-1} \rho^{-1} r_i.$$
  Therefore,
  \begin{align*}
    \gamma_{f_i}(v,br_i) \le (br_i)^{-4} \|f_i - L\|_{L_1(V(v,br_i))} \lesssim_b r_i^{-1} \|f_i - L\|_{L_\infty(B(v,br_i))} \lesssim_b A^{-1} \rho^{-1},
  \end{align*}
  as desired.
  
  Now we prove \eqref{eq:gamma-lower}. First, let $w'=\Psi_{f_i}(w)$. Then, as above, $d_\Kor(w,w')\le 2r_i$, so $V(w',6r_i)\subset V(w,8r_i)$. Then $\gamma_{\nu_i}(w,8r_i)\gtrsim \gamma_{\nu_i}(w',6r_i)$, so it suffices to prove that $\gamma_{\nu_i}(w', 6 r_i) \gtrsim A^{-1}$.
  
  We first apply a change of coordinates. Let $j\in J_i$ and let $R=R_{i,j}$, so that $w'=\Psi_{f_i}(R(s_0,t_0))$. Let $D= R([s_0 - \frac{1}{12} r_i, s_0 + \frac{1}{12} r_i]\times [0,r_i^2])$. We claim that if $\rho$ is large enough, then $D\subset V(w',6r_i)$.
  
  Let $t \in [0,r_i^2]$. Lemma~\ref{l:R-vert-bound}
  implies that $|z(R(s_0,t))-z(R(s_0,t_0))|\le\frac{4}{3}r_i^2$, so by \eqref{eq:flatness-fi},
  $$|f_i(R(s_0,t)) - f_i(w')| \le \|Zf_i\|_\infty |z(R(s_0,t)) - z(R(s_0,t_0))| \lesssim 
  A^{-1} \rho^{-1} r_i.$$
  We suppose that $\rho$ is large enough that $|f_i(R(s_0,t)) - f_i(w')| \le r_i$; then
  \begin{align*}
    d_\Kor& (\Psi_{f_i}(R(s_0,t)), w') \\
    & \le 2\sqrt{|z(R(s_0,t))-z(R(s_0,t_0))|} + |f_i(R(s_0,t)) - f_i(w')| \\
    & \le 2\sqrt{\frac{4}{3}r_i^2} + r_i \le 4 r_i.
  \end{align*}
  
  For $t \in [0,r_i^2]$, the curve $\lambda_t(s) = \Psi_{f_i}(R(s,t))$ is a horizontal curve on $\Gamma_{f_i}$ with velocity $\lambda_t'(s) = X+\nabla_{f_i} f_i(\lambda_t(s)) Y$. Since $\|\nabla_{f_i} f_i\|_\infty\le 1$, we have 
  $$d_\Kor(\lambda_t(s),\lambda_t(s')) \le \sqrt{2}|s-s'|.$$
  If $|s - s_0|\le r_i$, then
  $$d_\Kor(w', \Psi_{f_i}(R(s,t))) \le d_\Kor(w', \Psi_{f_i}(R(s_0,t))) + d_\Kor(\lambda_t(s_0), \lambda_t(s)) \le 6 r_i,$$
  so $R(s,t)\in \Pi(B(w',6r_i))$. Thus $D\subset V(w',6r_i)$.
  
  For any $h\in \Aff$, 
  \begin{align*}
    \|\nu_i - h\|_{L_1(V(w', 6r_i))} & \ge 
  \int_{D} |\nu_i(v) - h(v)| \ud v\\
  & \ge \frac{3}{4} \int_{s_0 - \frac{1}{12}r_i}^{s_0 + \frac{1}{12}r_i} \int_0^{r_i^2} |\nu_i(R(s,t)) - h(R(s,t))| \ud t \ud s,
  \end{align*}
  where we used Lemma~\ref{l:R-vert-bound} to bound the Jacobian of $R$.
  
  Since $h\in \Aff$ is constant on vertical lines, there is an affine function $h_0\from \R\to \R$ such that $h(R(s,t)) = h_0(s)$. Since $j\in J_i$,
  $$\nu_i(R(s,t)) = \kappa_{i,j}(R(s,t)) \stackrel{\eqref{eq:def-kappa}}{=} A^{-1} r_i \kappa(\hat{s}, \hat{t}),$$
  where $\hat{s} = A^{-1}r_i^{-1} s$ and $\hat{t} = r_i^{-2} t$. 
  
  Let
  $$M=\min_{\substack{c\in \R \\ \hat{u} \in [\frac{1}{4},\frac{3}{4}]}} \int_0^1 
  |\kappa(\hat{u}, \hat{v}) - c| \ud \hat{v}.$$
  We chose $\kappa$ so that $\kappa$ is zero on $\partial [0,1]^2$ and positive on $[\frac{1}{5},\frac{4}{5}]^2$, so $M>0$ by compactness. 
  
  Since $s_0\in [\frac{1}{3},\frac{2}{3}]$, if $s \in [s_0 - \frac{1}{12} r_i, s_0 + \frac{1}{12} r_i]$, then
  $\hat{s}\in [\frac{1}{4},\frac{3}{4}]$.
  Therefore, substituting $\hat{t} = r_i^{-2} t$, we find
  \begin{align*}
    \|\nu_i - h\|_{L_1(V(w', 6r_i))} & \gtrsim r_i^2 \int_{s_0 - \frac{1}{12}r_i}^{s_0 + \frac{1}{12}r_i} \int_0^{1} |A^{-1} r_i \kappa(\hat{s},\hat{t}) - h_0(s)| \ud \hat{t} \ud s \\
    & \ge A^{-1} r_i^3 \int_{s_0 - \frac{1}{12}r_i}^{s_0 + \frac{1}{12}r_i} M \ud s \ge \frac{1}{6} A^{-1} r_i^4 M.
  \end{align*}
  This holds for all $h\in \Aff$, so $\gamma_{\nu_i}(w',6r_i)\gtrsim A^{-1}$.
\end{proof}

Finally, we prove Proposition~\ref{prop:beta-theorem} and Theorem~\ref{thm:beta-theorem}.
\begin{proof}[Proof of Proposition~\ref{prop:beta-theorem} and Theorem~\ref{thm:beta-theorem}]
  By Proposition \ref{p:stopping-time}, there is a $\delta>0$ such that if $0 \le i \le \delta A^4$, then $\sum_{j\in J_i} |Q_{i,j}| > \frac{1}{2}$. Let $N=\lfloor\delta A^4\rfloor$, let $f=f_N$, and let $\Gamma=\Gamma_f$. 
  
  Let $c$ be as in Lemma~\ref{l:parametric-beta}. Let $0\le i < N$ and $j\in J_i$. Let $s\in [\frac{1}{3}A r_i, \frac{2}{3}A r_i]$ and $t\in [0,r_i^2]$ and let $w=\Psi_f(R_{i,j}(s,t))$.
  For any $\psi,\phi\from \HH\to \R$, $\gamma$ satisfies the reverse triangle inequality
  $$\gamma_{\psi + \phi}(w,r) \ge \gamma_{\psi}(w,r) - \gamma_{\phi}(w,r),$$ 
  so since $f = f_{i} + \nu_i + \sum_{m=i+1}^{N-1} \nu_m$, 
  \begin{equation}\label{eq:beta-decomp}
    \gamma_{f}(w, 8 r_i) \ge \gamma_{\nu_i}(w, 8 r_i) - \gamma_{f_i}(w, 8 r_i) - \sum_{m=i+1}^\infty \gamma_{\nu_m}(w, 8 r_i).
  \end{equation}
  Lemma~\ref{l:gamma-number-bounds} implies that when $\rho$ is sufficiently large, $\gamma_{\nu_i}(w, 8 r_i) \gtrsim A^{-1}$ and $\gamma_{f_i}(w, 8 r_i)\lesssim A^{-1}\rho^{-1}$. Furthermore, for $m \ge 1$, $\|\nu_m\|_\infty \lesssim A^{-1} r_m$, so
  $$\gamma_{\nu_m}(w, 8 r_i) \le (8 r_i)^{-4} |V(w,8 r_i)| \cdot \|\nu_m\|_\infty \lesssim r_i^{-1} A^{-1} r_m \lesssim A^{-1}\rho^{i-m}.$$
  Therefore,
  $$\gamma_{f_i}(w, 8 r_i) + \sum_{m=i+1}^\infty \gamma_{\nu_m}(w, 8 r_i) \lesssim A^{-1}\rho^{-1}.$$
  When $\rho$ is large, this is small compared to $\gamma_{\nu_i}(w, 8 r_i)$, so
  $$\gamma_f(w, 8 r_i) \ge \frac{1}{2}\gamma_{\nu_i}(w, 8 r_i) \gtrsim A^{-1}$$
  and $\beta_{\Gamma_{f}}(w, 8 c r_i) \gtrsim \gamma_{f}(w, 8 r_i) \gtrsim A^{-1}$. In fact, for $r\in [8cr_i, 16 cr_i]$,
  $$\beta_\Gamma(w, r) \gtrsim \beta_\Gamma(w, 8 c r_i) \gtrsim A^{-1}.$$ 
  
  Therefore, by Lemma~\ref{l:R-vert-bound}, 
  \begin{align*}
    \int_{Q_{i,j}} \beta_\Gamma(\Psi_{f}(v), r)^p \ud v & \ge
    \frac{3}{4} \int_{\frac{1}{3}A r_i}^{\frac{2}{3}A r_i} \int_{0}^{r_i^2} \beta_\Gamma(\Psi_{f}(R_{i,j}(s,t)), r)^p \ud t \ud s\\
    & \gtrsim A r_i^3 \cdot A^{-p} \gtrsim |Q_{i,j}| A^{-p},
  \end{align*}
  and
  $$\int_{\Psi_f(U)} \beta_\Gamma(v, r)^p \ud v \\
    \ge \sum_{i\in J_i} \int_{\Psi_f(Q_{i,j})} \beta_\Gamma(v, r)^p \ud v \gtrsim \sum_{i\in J_i} A^{-p} |Q_{i,j}| \gtrsim A^{-p}.$$
  
  We suppose that $\rho>4$ so that the intervals $[8cr_i,16cr_i]$ are disjoint and let $R > 16c$, so that $R>16cr_0$ and $\Psi_f(U)\subset B(\zero,R)$. Then
  \begin{multline*}
    \int_0^R \int_{\Gamma \cap B(\zero,R)} \beta_\Gamma(v,r)^p \ud v \frac{\ud r}{r} \ge \sum_{i=0}^{N-1} \int_{8cr_i}^{16cr_i} \int_{\Psi_f(U)} \beta_\Gamma(v, r)^p \ud v \frac{\ud r}{r} \\
    \gtrsim \sum_{i=0}^{N-1} \int_{8cr_i}^{16cr_i} A^{-p} \frac{\ud r}{r} \ge N\log 2\cdot A^{-p} \gtrsim \delta A^{4-p}.
  \end{multline*}
  This proves Proposition~\ref{prop:beta-theorem}. By Lemma~\ref{lem:def-dv}, $\ud v\approx \ud \cH^3(v)$, so
  $$\int_0^R \int_{\Gamma \cap B(\zero,R)} \beta_\Gamma(v,r)^p \ud \cH^3(v) \frac{\ud r}{r} \approx \int_0^R \int_{\Gamma \cap B(\zero,R)} \beta_\Gamma(v,r)^p \ud v \frac{\ud r}{r} \gtrsim \delta A^{4-p}.$$
  When $p < 4$, this integral goes to infinity as $A\to \infty$, proving Theorem~\ref{thm:beta-theorem}. 
\end{proof}

\section{Reduction to vertical planes}
\label{sec:redvertplane} 
Now we begin the proof of Theorem~\ref{mainthmintro}, which will take up the rest of this paper. In this section and the rest of the paper, $K\from \HH \setminus \{\zero\} \to \R^2$ will denote a smooth orthogonal kernel which is homogeneous of degree $-3$ and $\widehat{K}$ will denote the (also orthogonal) kernel $\widehat{K}(v)=K(v^{-1})$. Many of our bounds will depend on $K$, so we omit $K$ in subscripts like $\lesssim_K$.

Let $\phi\from \HH \to \R$ be an intrinsic Lipschitz function. We define $\eta_\phi = \mu|_{\Gamma_\phi}$, where $\mu$ is the measure defined in Section~\ref{sec:ilgs}. Then $T\eta_\phi(p) = T^K\eta_\phi(p)$ is given by
$$T\eta_\phi(p) := \pv(p) \int_{\Gamma_\phi} \widehat{K}(p^{-1}w)\ud w$$
for all $p\in \HH$ such that the principal value on the right exists, i.e., for all $p\in \HH$ such that
\begin{equation}\label{eq:T-eta-phi}
  \lim_{\substack{r\to 0 \\ R\to \infty}} \int_{\Gamma_\phi\cap (B(p,R)\setminus B(p,r))} \widehat{K}(p^{-1}w) \ud w
\end{equation}
converges. Let $\one$ be the function equal to $1$ on all of $\HH$; then, using the operator notation in Section~\ref{sec:ker}, we can write $T\eta_\phi = T_{\eta_\phi}\one$. 


In this section, we will show that when $\phi$ is a bounded smooth function and $p\in \Gamma_\phi$, then $T\eta_\phi(p)$ is the principal value of a singular integral on a vertical plane.
For any $0<r<s$, any $p\in \HH$, and any vertical plane $Q$ through $p$, let 
$$A_{r,s}^Q(p) = Q \cap (B(p,s)\setminus B(p,r))\subset Q$$
and let $A_{r,s}^Q = A_{r,s}^Q(\zero)$.  When $Q=V_0$ we will suppress the superscripts.

For a point $p\in \HH$, a vertical plane $Q$ through $\zero$ with finite slope, a function $f\from \HH\to \R$ which is constant on cosets of $\paramY$, an intrinsic Lipschitz function $\phi$, and  $0<r<R$ we let
\begin{equation}\label{eq:def-tilde-T}
  \tilde{T}^{K;Q}_{\phi;r,R} f(p) = \tilde{T}^Q_{\phi;r,R} f(p) := \int_{\Psi_\phi(p) A_{r,R}^Q} \widehat{K}(\Psi_\phi(p)^{-1}\Psi_{\phi}(v)) f(v)\ud v,
\end{equation}
and
$$\tilde{T}^Q_{\phi} f(p) := \lim_{\substack{r\to 0 \\ R\to \infty}} \tilde{T}^Q_{\phi;r,R} f(p),$$
if this limit exists. Note that $\tilde{T}^Q_\phi f$ and $\tilde{T}^Q_{\phi;r,R} f$ are constant on cosets of $\paramY$. 

When $p\in \Gamma_\phi$ and $f=\one$ is a constant function, we can substitute $w=\Psi_\phi(v)$ to write $\tilde{T}^Q_\phi \one$ like the right side of \eqref{eq:T-eta-phi}:
\begin{equation} \label{eq:tilde-T-eta-phi}
  \tilde{T}^Q_\phi \one(p) = \lim_{\substack{r\to 0 \\ R\to \infty}} \int_{p A^Q_{r,R}} \widehat{K}(p^{-1}\Psi_{\phi}(v)) \ud v = \lim_{\substack{r\to 0 \\ R\to \infty}} \int_{\Psi_\phi\big(p A^Q_{r,R}\big)} \widehat{K}(p^{-1}w) \ud w.
\end{equation}

In this section, we will compare the integrals in \eqref{eq:T-eta-phi} and \eqref{eq:tilde-T-eta-phi} and prove the following proposition.

\begin{prop}\label{prop:annulus-change}
  Let $\phi\from \HH \to \R$ be a  smooth function which is constant on cosets of $\paramY$ and let $Q$ be a vertical plane through $\zero$. Let  
  $$C=\max \big\{\|\phi\|_\infty, \|\nabla_\phi \phi\|_\infty, \|\nabla^2_\phi \phi\|_\infty, \|\partial_z \phi\|_\infty,|\slope Q|\big\}$$
  and suppose that $C<\infty$.
  Then for any $p\in \Gamma_\phi$, the limits $T\eta_\phi(p)$ and $\tilde{T}^Q_\phi \one(p)$ exist and
  \begin{equation}\label{eq:annulus-T}
    T\eta_\phi(p)
    = \tilde{T}^Q_\phi \one(p).
  \end{equation}
  In fact, 
  $$\left| T\eta_\phi(p) - \tilde{T}^Q_{\phi;r,R}\one(p) \right|\lesssim_{C} r + R^{-1}.$$
\end{prop}
In particular, under these conditions, $\tilde{T}^Q_\phi \one(p)$ is independent of $Q$, so we write $\tilde{T}_\phi = \tilde{T}^{V_0}_\phi$.

The implicit constants in the proofs in this section almost all depend on $C$, so we will omit the dependence on $C$ from our notation.

We first establish bounds on $\tilde{T}^Q_{\phi;s,t} \one(p)$.

\begin{lemma} \label{lem:riesz-converge}
  Let $\phi$, $Q$, and $C$ be as in Proposition~\ref{prop:annulus-change}. Then for any $0 < r' < r < 1 < R < R'$ and $p \in \Gamma_\phi$,
  $$\left|\tilde{T}^Q_{\phi;r',r} \one(p)\right| \lesssim r$$
  and
  $$\left|\tilde{T}^Q_{\phi;R,R'} \one(p)\right| \lesssim R^{-1}.$$
  In particular, the principal value $\tilde{T}^Q_\phi\one(p)$ exists for all $p\in \Gamma_\phi.$
  
  Furthermore, for any bounded function $f\from \HH \to \R$ which is constant on cosets of $\paramY$ and any $0<s<t$, 
  $$\left|\tilde{T}^Q_{\phi;s,t} f(p) \right| \lesssim \|f\|_\infty \log\frac{t}{s}.$$
\end{lemma}

Let $Q$ be a vertical plane through $\zero$ with finite slope. Then, by Lemma~\ref{lem:def-dv}, $\mu|_Q$ is a $3$--regular left-invariant measure on $Q$. The uniqueness (up to scaling) of the Haar measure on $Q$ implies that $\mu|_{Q}$ is a constant multiple of $\cL|_{Q}$, i.e. $\mu|_{Q}$ is a $3$--uniform measure. Hence, the following useful lemma follows easily, see e.g.\ \cite{merlo} for the details.
\begin{lemma}\label{lem:polar}
  Let $Q$ be a vertical plane with finite slope. There exists a $c > 0$, depending on the slope of $Q$, so that for any $v\in Q$, $0\le r_1\le r_2\le \infty$, and any Borel integrable function $f\from \R\to \R$, 
  $$\int_{A^Q_{r_1,r_2}(v)}
  f(d_\Kor(v,w)) \ud w = c \int_{r_1}^{r_2} f(r) r^2 \ud r.$$
\end{lemma}
  
For $r>0$, let $B_r=B(\zero,r)\subset \HH$.

\begin{proof}[Proof of Lemma~\ref{lem:riesz-converge}]
  By translation, we may assume without loss of generality that $\zero\in \Gamma_\phi$ and $p = \zero$.
  For arbitrary $s < t$, we define 
  $$I_{s,t} = \tilde{T}^Q_{\phi;s,t}\one(p) = \int_{A_{s,t}^Q} \widehat{K}(\Psi_\phi(v)) \ud v.$$
  We will bound $|I_{r',r}| \lesssim r$ and $|I_{R,R'}| \lesssim R^{-1}$.
  
  We will use the following symmetrization argument.  Let $\theta(x,y,z)=(-x,-y,z)$. Then $\theta(A_{s,t}^Q)=A_{s,t}^Q$ and $K$ is $\HH$--odd, so
  \begin{align*}
    |I_{s,t}| \overset{\eqref{eq:int-even-odd}}{=} \frac{1}{2} \left| \int_{A_{s,t}^Q} \widehat{K}(\Psi_\phi(v)) + \widehat{K}(\Psi_\phi(\theta(v))) \ud v \right| = \frac{1}{2} \left| \int_{A_{s,t}^Q} - \widehat{K}(\theta(\Psi_\phi(v))) + \widehat{K}(\Psi_\phi(\theta(v))) \ud v \right|.
  \end{align*}
  Since $\Psi_\phi(v) = v Y^{\phi(v)-y(v)}$ and 
  $$\Psi_\phi(\theta(v)) = \theta(v) Y^{\phi(\theta(v))-y(\theta(v))} = \theta(v) Y^{\phi(\theta(v)) + y(v)},$$
  we have
  \begin{align*}
    \theta(\Psi_\phi(v)) = \theta(v)Y^{-\phi(v)+y(v)} = \Psi_\phi(\theta(v)) Y^{-\phi(v) - \phi(\theta(v))},
  \end{align*}
  and by the mean value theorem, 
  \begin{align*}
    \left| \widehat{K}(\Psi_\phi(\theta(v))) - \widehat{K}(\theta(\Psi_\phi(v))) \right| = \left|\YL\widehat{K}(m(v))\right|\cdot |\phi(v) + \phi(\theta(v))|,
  \end{align*}
  where $m(v)$ is a point on the horizontal line segment between $\theta(\Psi_\phi(v))$ and $\Psi_\phi(\theta(v))$. That is,
  \begin{align}
    |I_{s,t}| \le \frac{1}{2} \int_{A^Q_{s,t}} \left|\YL\widehat{K}(m(v))\right|\cdot|\phi(v) + \phi(\theta(v))| \ud v.
    \label{e:mirror-int}
  \end{align}
  Since $m(v)\in \theta(v)\paramY$ and because $Q$ has bounded slope, we have $\|m(v)\|_\Kor\gtrsim \|v\|_\Kor$ for all $v\in Q$. 
  Since $m(v)$ is between $\theta(\Psi_\phi(v))$ and $\Psi_\phi(\theta(v))$, 
  $$\|m(v)\|_\Kor\lesssim \max \{\|\Psi_\phi(v)\|_\Kor,\|\Psi_\phi(\theta(v))\|_\Kor\} \lesssim \|v\|_\Kor.$$ 
  That is, for all $v\in Q$,
  \begin{equation}\label{eq:m-dist}
    \|m(v)\|_\Kor\approx \|v\|_\Kor.
  \end{equation}
  
  Furthermore, the bounds on $\phi$ and its derivatives give bounds on $\phi(v) + \phi(\theta(v))$. On one hand, 
  $$|\phi(v) + \phi(\theta(v))| \le 2 \|\phi\|_\infty \le 2C\lesssim 1.$$
  On the other hand, by Lemma~\ref{lem:a-Taylor}, $|\phi(v) + \phi(\theta(v))| \lesssim \|v\|_\Kor^2$, so for all $v\in Q$,
  \begin{equation}\label{eq:even-phi-bounds}
    |\phi(v) + \phi(\theta(v))| \lesssim \min \{1,\|v\|_\Kor^2\}.
  \end{equation}
  
  Therefore, by these bounds, the $(-4)$--homogeneity of $\YL\widehat{K}$ and Lemma~\ref{lem:polar},
  \begin{equation}
  \label{eq:istbd}
  \begin{split}
    |I_{s,t}| &\lesssim \int_{A_{s,t}^Q} \left|\YL\widehat{K}(m(v))\right| \min \{1,\|v\|_\Kor^2\} \ud v\\ &\lesssim \int_{A_{s,t}^Q} \min \{\|v\|_\Kor^{-4},\|v\|_\Kor^{-2}\} \ud v \\ &\lesssim \int_{s}^{t} \min \{\rho^{-4},\rho^{-2}\} \cdot \rho^2 \ud \rho \lesssim \min \{s^{-1}-t^{-1}, t-s\}.
  \end{split}
  \end{equation}
  In particular, for any $0<r'<r<1<R<R'<\infty$, $|I_{r',r}| \lesssim r$ and $|I_{R,R'}|\lesssim R^{-1}$. Thus
  $$\left|\tilde{T}^Q_{\phi;r',R'} \one(p) - \tilde{T}^Q_{\phi;r,R}\one(p)\right| \le |I_{r',r}| + |I_{R,R'}| \lesssim r + R^{-1}.$$
  That is, $\tilde{T}^Q_{\phi;r,R} \one(p)$ converges as $r\to 0$ and $R\to \infty$, so the principal value $\tilde{T}^Q_{\phi} \one(p)$ exists.
  
  Finally, if $f$ is constant on cosets of $\paramY$ and $0<s<t$, 
  \begin{multline*}
    \left|\tilde{T}^Q_{\phi;s,t} f(p) \right| \le \int_{A^Q_{s,t}} \left|\widehat{K}(\Psi_\phi(v)) f(v)\right|\ud v \lesssim \int_{A^Q_{s,t}} \|v\|^{-3}_\Kor\|f\|_\infty \ud v \\ 
    \lesssim \|f\|_\infty \int_s^t \rho^{-3}\cdot \rho^2\ud \rho = \|f\|_\infty \log\frac{t}{s},
  \end{multline*}
  as desired.
\end{proof}

The next lemma lets us compare $\Gamma_\phi \cap B(p,r)$ and $\Psi_\phi(V\cap B(p,r))$ when $V$ is a vertical plane. Let $\Pi_V\from \HH\to V$ be the projection from $\HH$ to $V$ along cosets of $\paramY$, as in Section~\ref{sec:prelim-Heis}. Let $B_r=B(\zero,r)$.
\begin{lemma}\label{lem:annular-approx}
  Let $\phi$ and $C$ be as in Proposition~\ref{prop:annulus-change}. Let $\Gamma=\Gamma_\phi$ and $p\in \Gamma$. 
  Let $W$ be the vertical tangent plane to $\Gamma$ at $p$, so that $\slope W = \nabla_\phi \phi(p)$. Then there is a $c>0$ depending only on $C$ such that for $r>0$,
  $$W\cap B(p,r - cr^2)\subset \Pi_W(\Gamma\cap B(p,r)) \subset W\cap B(p,r + c r^2)$$
  and for $R>0$, 
  $$pV_0 \cap B(p,R - c)\subset \Pi_{pV_0}(\Gamma\cap B(p,R)) \subset pV_0 \cap B(p,R + c).$$
\end{lemma}
Though the inclusions hold for all $r$ and $R$, they are most useful when $r$ is small and $R$ is large.
\begin{proof}
  Without loss of generality, we may suppose that $p=\zero$ so that $\Pi_{pV_0}=\Pi$. 
  Let $\sigma=\nabla_\phi \phi(p)$. By Lemma~\ref{lem:a-Taylor}, we have $\phi(q) = \sigma x(q) + O(\|q\|_\Kor^2)$ for any $q\in \HH$. Then there is a $c>C$ such that
  $$d_\Kor(\Psi_\phi(q), \Pi_W(q)) = d_\Kor(q Y^{\phi(q)-y(q)}, q Y^{\sigma x(q)-y(q)}) = |\phi(q)-\sigma x(q)| \le c\|q\|_\Kor^2.$$
  
  In particular, for $r>0$, if $q\in W\cap B_{r-cr^2}$, then $d_\Kor(\Psi_\phi(q), q) \le c\|q\|_\Kor^2 \le c r^2$. Therefore, $\Psi_\phi(q)\in B_r$ and
  $$q = \Pi_W(\Psi_\phi(q))\in \Pi_W(\Gamma\cap B_{r}),$$
  so $W\cap B_{r-cr^2} \subset \Pi_W(\Gamma\cap B_r)$.  
  
  Conversely, if $q'\in \Gamma\cap B_r$, then $\Pi_W(q') \in W\cap B_{r + cr^2}$, so $\Pi_W(\Gamma\cap B_{r}) \subset W\cap B_{r + cr^2}$.
  This proves the first part of the lemma.
  
  Similarly, since $\|\phi\|_\infty \le C$, we have $$d_\Kor(\Psi_\phi(q), \Pi(q)) =d_\Kor(qY^{\phi(q)-y(q)}, q Y^{-y(q)})= |\phi(q)| \le C$$ for all $q$. Therefore,
  $$V_0 \cap B_{R - c}\subset \Pi(\Gamma\cap B_R) \subset V_0 \cap B_{R + c},$$
  as desired.
\end{proof}

This lets us write $T\eta_\phi(p)$ in terms of an integral on $pV_0$.
\begin{lemma}\label{lem:pv-independence-G}
  Let $\Gamma=\Gamma_\phi$, $p\in \Gamma$, and $W$ be as in Lemma~\ref{lem:annular-approx}. For $0<r<R$, let
  $$E_{r,R} = (p V_0\cap B(p,R))\setminus \Pi_{pV_0}(W \cap B(p,r)).$$ 
  Then $T\eta_\phi(\zero)$ exists and 
  \begin{equation}\label{eq:pv-independence}
    T\eta_\phi(\zero) = \lim_{\substack{r\to 0 \\ R\to\infty}} \int_{E_{r,R}} \widehat{K}(\Psi_{\phi}(v))\ud v.
  \end{equation}
\end{lemma}
\begin{proof}
  Again, we suppose that $p=\zero$, so that $E_{r,R} = (V_0\cap B_R)\setminus \Pi(W\cap B_r)$. 
  
  We first note that the limit on the right side of \eqref{eq:pv-independence} exists. If $r$ is sufficiently small and $R$ is sufficiently large, then $\Pi(W\cap B_r)\subset V_0\cap B_R$. If in addition $0<r'<r<R<R'$, then $E_{r,R}\subset E_{r',R'}$ and
  $E_{r',R'}\setminus E_{r,R} = \Pi(A^W_{r',r}) \cup A_{R,R'},$
  so by Lemma~\ref{lem:riesz-converge},
  $$\left|\int_{E_{r',R'}} \widehat{K}(\Psi_{\phi}(v)) \ud v - \int_{E_{r,R}} \widehat{K}(\Psi_{\phi}(v)) \ud v\right| \lesssim r + R^{-1}.$$
  As $r\to 0$ and $R\to \infty$, this goes to zero, so the limit in \eqref{eq:pv-independence} exists.
  
  Now we compare this limit with $T\eta_\phi$.  
  For any $r>0$ and any vertical plane $P$ through $\zero$, let 
  $$F^P_r = (P\cap B_r) \symdiff \Pi_P(\Gamma\cap B_r)$$
  where $A \symdiff B$ is the symmetric difference $(A\setminus B) \cup (B\setminus A)$.
  Comparing $E_{r,R}$ to $\Pi(\Gamma \cap (B_R \setminus B_r))$, we find that
  $$(V_0\cap B_R) \symdiff \Pi(\Gamma \cap B_R) = F^{V_0}_R,$$
  and 
  $$\Pi(W\cap B_r) \symdiff \Pi(\Gamma \cap B_r) = \Pi((W\cap B_r) \symdiff \Pi_W(\Gamma \cap B_r)) = \Pi(F^W_r),$$
  so
  $$E_{r,R} \symdiff \Pi(\Gamma \cap (B_R \setminus B_r)) \subset F^{V_0}_R \cup \Pi(F^W_r).$$
  
  Therefore, as in \eqref{eq:T-eta-phi} and \eqref{eq:tilde-T-eta-phi},
  \begin{align*}
    \left|T_{r,R} \eta_{\phi}(p) - \int_{E_{r,R}} \widehat{K}(\Psi_{\phi}(v))\ud v\right| & =  \left|\int_{\Gamma \cap (B_R \setminus B_r)} \widehat{K}(w) \ud w - \int_{E_{r,R}} \widehat{K}(\Psi_\phi(w)) \ud w\right| \\
    & \le \int_{F^{V_0}_{R}\cup \Pi(F^W_{r})} \left|\widehat{K}(\Psi_\phi(w))\right|\ud w.
  \end{align*}
  Let $c$ be as in Lemma~\ref{lem:annular-approx}. Then for $0<r<R<\infty$,
  $$F^{W}_r \subset W \cap (B_{r+cr^2} \setminus B_{r-cr^2}) = A^W_{r-cr^2, r+cr^2},$$
  and
  $$F_R^{V_0}\subset A_{R-c,R+c}.$$
  By Lemma~\ref{lem:polar}, and using that $\Psi_{\phi}$ is constant on cosets of $\paramY$, we get
  \begin{align*}
    \left|T_{r,R} \eta_{\phi}(p) - \int_{E_{r,R}} \widehat{K}(\Psi_{\phi}(v))\ud v\right| & \le \int_{A_{R-c,R+c}} \left|\widehat{K}(\Psi_\phi(w))\right|\ud w + \int_{\Pi(A^W_{r-cr^2,r+cr^2})} \left|\widehat{K}(\Psi_\phi(w))\right|\ud w \\ 
    &\overset{\eqref{eq:def-dv}}{=}\int_{A_{R-c,R+c}} \left|\widehat{K}(\Psi_\phi(w))\right|\ud w + \int_{A^W_{r-cr^2,r+cr^2}} \left|\widehat{K}(\Psi_\phi(w))\right|\ud w \\
    & \lesssim \int_{R-c}^{R+c} \rho^{-3} \cdot \rho^2 \ud \rho + \int_{r-cr^2}^{r+c r^2} \rho^{-3} \cdot \rho^2 \ud \rho \\
    & = \log \frac{R+c}{R-c} + \log \frac{r+cr^2}{r-cr^2}.
  \end{align*}
  This goes to zero as $r\to 0$ and $R\to \infty$, so it implies \eqref{eq:pv-independence}, as desired.  
\end{proof}

Finally, we prove Proposition~\ref{prop:annulus-change}.

\begin{proof}[Proof of Proposition \ref{prop:annulus-change}]
  Again, we suppose that $p=\zero$.
  Let $W$ and $E_{r,R} = (V_0\cap B_R)\setminus \Pi(W\cap B_r)$ be as in Lemma~\ref{lem:pv-independence-G}. We may suppose that $R$ is large enough and $r$ is small enough that $\Pi(W\cap B_r)\subset V_0\cap B_R$. 
  Let $P$ be a vertical plane through $\zero$ with $|\slope P|<C$ and let
  $$J_{r,R} := \left|\int_{E_{r,R}} \widehat{K}(\Psi_{\phi}(v))\ud v - \tilde{T}^P_{\phi;r,R} \one(\zero) \right| = \left|\int_{E_{r,R}} \widehat{K}(\Psi_{\phi}(v))\ud v - \int_{A^P_{r,R}} \widehat{K}(\Psi_{\phi}(v))\ud v \right|.$$
  We claim that $J_{r,R} \lesssim r + R^{-1}$.

  Since $\Pi\circ\theta=\theta\circ \Pi$, it follows that $E_{r,R}$ and $A^P_{r,R}$ are symmetric around the $z$--axis, i.e., $\theta(E_{r,R}) = E_{r,R}$ and $\theta(A^P_{r,R}) = A^P_{r,R}$. Let $D\subset V_0$ be a Borel set such that $D\subset A_{\epsilon,\epsilon^{-1}}$ for some $\epsilon>0$.
  We claim that if $\theta(D)=D$, then for any $t>0$,
  $$\int_{s_t(D)}  \widehat{K}(\Psi_\phi(v)) \ud v  \lesssim_\epsilon \min \{t,t^{-1}\}.$$
  
  Let $I_t=\int_{s_t(D)}  \widehat{K}(\Psi_\phi(v)) \ud v$. As in the proof of Lemma~\ref{lem:riesz-converge}, the $\HH$--oddness of $K$ and the mean value theorem imply that
  \begin{multline*}
    |I_t| = \frac{1}{2} \left| \int_{s_t(D)} - \widehat{K}(\theta(\Psi_\phi(v))) + \widehat{K}(\Psi_\phi(\theta(v))) \ud v \right| \\
    \le \frac{1}{2} \int_{A_{t\epsilon,t\epsilon^{-1}}} |\YL\widehat{K}(m(v))|\cdot|\phi(v) + \phi(\theta(v))| \ud v,
  \end{multline*}
  where for every $v$, $m(v)$ is a point on the horizontal line segment between $\theta(\Psi_\phi(v))$ and $\Psi_\phi(\theta(v))$ with $\|m(v)\|_\Kor\approx \|v\|_\Kor$. As in \eqref{eq:istbd}, by \eqref{eq:even-phi-bounds}, the $(-4)$--homogeneity of $\YL\widehat{K}$, and Lemma~\ref{lem:polar},
  \begin{multline*}
    |I_t| \lesssim \int_{A_{t\epsilon, t\epsilon^{-1}}} \min \{\|v\|_\Kor^{-4},\|v\|_\Kor^{-2}\} \ud v \lesssim \int_{t\epsilon}^{t\epsilon^{-1}} \min \{\rho^{-4},\rho^{-2}\} \cdot \rho^2 \ud \rho \\
    \lesssim \min \{t^{-1}\epsilon^{-1}, t\epsilon^{-1}\} \lesssim_\epsilon \min \{t^{-1},t\}.
  \end{multline*}
  
  It follows that if $g$ is an $\HH$--even bounded Borel function supported on $A_{\epsilon,\epsilon^{-1}}$ and $t>0$, then 
  \begin{equation} \label{eq:symmetric-function}
   \left| \int_{V_0} \widehat{K}(\Psi_{\phi}(v)) g(s_t(v)) \ud v\right| \lesssim \|g\|_\infty \min \{t^{-1},t\}.
  \end{equation}
  
  Now we apply this to $J_{r,R}$. The supports of $\one_{E_{r,R}}$ and $\one_{\Pi(A^P_{r,R})}$ are too large to apply \eqref{eq:symmetric-function} directly, but we can write
  $$\one_{E_{r,R}} = \one_{V_0\cap B_R} - \one_{\Pi(W\cap B_r)}$$
  $$\one_{\Pi(A^P_{r,R})} = \one_{\Pi(P\cap B_R)} - \one_{\Pi(P\cap B_r)}.$$
  Let $g=\one_{V_0\cap B_1} - \one_{\Pi(P\cap B_1)}$ and $h = \one_{\Pi(P\cap B_1)} - \one_{\Pi(W\cap B_1)}$ so that
  $$\one_{E_{r,R}} - \one_{\Pi(A^P_{r,R})} = (\one_{V_0\cap B_R} - \one_{\Pi(P\cap B_R)}) + (\one_{\Pi(P\cap B_r)} - \one_{\Pi(W\cap B_r)}) = g\circ s_{R^{-1}} + h\circ s_{r^{-1}}.$$
  Then $g$ and $h$ are $\HH$--even, and there is an $\epsilon>0$ such that both are supported in $A_{\epsilon,\epsilon^{-1}}$. Therefore,
  \begin{align*}
    J_{r,R} & = \left|\int_{V_0} \widehat{K}(\Psi_{\phi}(v)) \one_{E_{r,R}}\ud v - \int_{P} \widehat{K}(\Psi_{\phi}(v)) \one_{A^P_{r,R}}\ud v\right| \\
    &\overset{\eqref{eq:def-dv}}{=} \left|\int_{V_0} \widehat{K}(\Psi_{\phi}(v)) (\one_{E_{r,R}} - \one_{\Pi(A^P_{r,R})})\ud v\right| \\
    & = \left|\int_{V_0} \widehat{K}(\Psi_{\phi}(v)) (g(s_{R^{-1}}(v)) + h(s_{r^{-1}}(v))) \ud v\right| \\
    &\overset{\eqref{eq:symmetric-function}}{\lesssim} r + R^{-1}.
  \end{align*}
  
  This implies
  $$\tilde{T}^P_{\phi}\one(\zero) = \lim_{\substack{r\to 0 \\ R\to\infty}} \int_{A^P_{r,R}} \widehat{K}(\Psi_{\phi}(v))\ud v = \lim_{\substack{r\to 0 \\ R\to\infty}} \int_{E_{r,R}} \widehat{K}(\Psi_{\phi}(v))\ud v,$$
  so by Lemma~\ref{lem:pv-independence-G}, $\tilde{T}^P_{\phi}\one(\zero) = T\eta_{\phi}(\zero)$, as desired.
\end{proof}

\section{Singular integrals on perturbed surfaces and the proof of Proposition~\ref{prop:mainBounds-tilde-T}}\label{sec:sing-int-perturbed}

In Section~\ref{sec:construction}, we constructed intrinsic Lipschitz functions $f_i$ that depend on parameters $A$, $\rho$, and $i$. In that section, $A$ and $\rho$ were fixed while $i$ varies; in this section, we will need to vary $A$ and $\rho$, so we will write $f_i$ as $f_{i,A,\rho}$ when we need to specify $A$ and $\rho$.

Each surface $\Gamma_{f_{n,A,\rho}}$ can be constructed by starting with the vertical plane $V_0$, then repeatedly perturbing it at smaller and smaller scales. In this section, we will state bounds on the change in the singular integral $\tilde{T}_\zeta \one$ when $\zeta$ is perturbed and use these bounds to prove Proposition~\ref{prop:mainBounds-tilde-T}. 

For any intrinsic Lipschitz function $\zeta \from \HH\to \R$, we let $F_\zeta\from \HH \to \R$,
\begin{equation}\label{eq:def-F}
  F_\zeta(p) := \tilde{T}_{\zeta} \one(p) = \pv(\Psi_\zeta(p))\int_{\Psi_\zeta(p) V_0} \widehat{K}(\Psi_\zeta(p)^{-1}\Psi_{\zeta}(v)) \ud v.
\end{equation}
For any $\psi\from \HH\to \R$ which is constant on cosets of $\paramY$ and $t\in \R$, let $G_{\zeta,\psi}(t) := F_{\zeta+t \psi}$. We can then bound $F_{\zeta + \psi} - F_{\zeta}=G_{\zeta,\psi}(1)-G_{\zeta,\psi}(0)$ by bounding the derivatives of $G_{\zeta,\psi}$.
In our applications, $\zeta$ and $\psi$ will satisfy bounds like those in Lemma~\ref{lem:gamma-deriv-bounds}, so that the length scale of $\psi$ is much smaller than the length scale of $\zeta$.

We denote $G'_{\zeta,\psi}(t)(p) = \partial_t[G_{\zeta,\psi}(t)(p)]$ and $G''_{\zeta,\psi}(t)(p) = \partial^2_t[G_{\zeta,\psi}(t)(p)]$. (This is a slight abuse of notation because the limits in the partial derivatives may only converge pointwise and not uniformly.) 
For $r\le R$, we define truncations
\begin{equation}\label{eq:def-F-trunc}
  F^{r,R}_\zeta(p) := \tilde{T}^{r,R}_{\zeta} \one(p) = \int_{\Psi_\zeta(p) A_{r,R}} \widehat{K}(\Psi_\zeta(p)^{-1}\Psi_{\zeta}(v)) \ud v
\end{equation}
and $G^{r,R}_{\zeta,\psi}(t) := F^{r,R}_{\zeta + t \psi}$. We will prove the following formula for $G_{\zeta,\psi}'$.

\begin{prop}\label{prop:perturb-formula}
  Let $\zeta, \psi\from \HH \to \R$ be smooth functions that are constant on cosets of $\paramY$. Suppose that  $\|\psi\|_\infty <\infty$ and that $\zeta$ is intrinsic Lipschitz. Then, for any $p\in \Gamma_\zeta$,
  $$\lim_{\substack{r\to 0 \\ R\to \infty}} (G^{r,R}_{\zeta,\psi})'(0)(p) = G'_{\zeta,\psi}(0)(p).$$

  Furthermore, there is a Sobolev-type norm $\|\psi\|_{W_\zeta}$ depending on $\psi$ and its derivatives of order at most $2$ such that
  $$|G_{\zeta,\psi}'(0)(\zero)| \lesssim_{\zeta} \|\psi\|_{W_\zeta}.$$
  If $\alpha$ and $\gamma$ satisfy the bounds in Lemma~\ref{lem:gamma-deriv-bounds} for some $c>0$, then
  $$\|G_{\alpha,\gamma}'(0)\|_\infty \lesssim_c A^{-1}.$$
\end{prop}
We refer the reader to Lemma~\ref{lem:G'-formula-and-bound-rev} for the details of the bound on $|G_{\zeta,\psi}'(0)(\zero)|$.

To use this to bound $F_{\zeta+\psi}=G_{\zeta,\psi}(1)$, we need the following proposition, which likewise bounds $G''_{\zeta,\psi}$ in terms of a Sobolev-type norm on $\zeta$ and $\psi$. Let $A>1$ and let $\hat{\partial} = A\nabla_\zeta$. As in Section~\ref{sec:construction}, we let $\{ Z, \hat{\partial}\}^n$ denote the set of differential operators that can be written as words of length $n$. Let $\{ Z, \hat{\partial}\}^*$ denote the set of all words.

\begin{prop}\label{prop:second-bounds-v2}
  For any $A>1$ and any $C>0$, if $\rho$ is sufficiently large, then the following bounds hold. Let $\zeta, \psi\from \HH\to \R$ be constant on cosets of $\paramY$. Suppose that for any word $E\in \{ Z, \hat{\partial}\}^*$ of length at most 3, 
  \begin{equation}\label{eq:2der-gamma-bound}
    \|E \psi\|_\infty \le C A^{-1}
  \end{equation}
  and if $E\not \in \{\id, \hat{\partial}\}$,
  \begin{equation}\label{eq:2der-alpha-bound}
    \|E \zeta\|_\infty \le C \rho^{-1}.
  \end{equation}
  Then, for any $p\in \HH$, the function $t\mapsto G_{\zeta,\psi}(t)(p)$ is $C^2$ and satisfies
  $$\left\|G''_{\zeta,\psi}(t) \right\|_\infty \lesssim_C A^{-3}$$
  for all $t\in [0,1]$.
  If $\alpha$ and $\gamma$ satisfy the bounds in Lemma~\ref{lem:gamma-deriv-bounds} for some $c>0$, then
  $$\|G_{\alpha,\gamma}''(0)\|_\infty \lesssim_c A^{-3}.$$
\end{prop}

We will use Proposition~\ref{prop:perturb-formula} and Proposition~\ref{prop:second-bounds-v2} to prove the following bounds on $F_{f_i}$.
\begin{lemma}\label{lem:mainBounds}
  Let $A>0$ and suppose that $\rho>0$ is sufficiently large, depending on $A$.  
  Let $i\ge 0$ and let $f_i$ and $\nu_i$ be as in Section~\ref{sec:construction}. Then there is an $\epsilon>0$ such that:
  \begin{enumerate}
  \item \label{it:first-upper}
    $\|G_{f_i,\nu_i}'(0)\|_\infty \lesssim A^{-1}$.
  \item \label{it:second-upper}
    For each $v\in \HH$, the function $t\mapsto G_{f_i,\nu_i}(t)(v)$ is $C^2$, and for all $t\in [0,1]$, 
    $\|G_{f_i,\nu_i}''(t)\|_\infty\lesssim A^{-3}$.
  \item \label{it:quasi-ortho}
    For all $0\le i < j$,
    \begin{equation}\label{eq:quasi-ortho}
    |\langle G_{f_i,\nu_i}'(0), G_{f_j,\nu_j}'(0)\rangle| \lesssim \rho^{-\epsilon}.
    \end{equation}
  \item \label{it:first-lower}
    If $K$ is the Riesz kernel $\mathsf{R}$ and $i<\epsilon A^4$, then $\|G_{f_i,\nu_i}'(0)\|_{L_2(U)} \gtrsim A^{-1}$.
  \end{enumerate}
\end{lemma}
This lemma implies Proposition~\ref{prop:mainBounds-tilde-T}.
\begin{proof}[Proof of Proposition~\ref{prop:mainBounds-tilde-T}]
  In this proof, we use $\|\cdot\|_U$ to denote $\|\cdot \|_{L_2(U)}$. Let $g_i=G_{f_i,\nu_i}$.
  By Taylor's theorem,
  $$\|F_{f_{i+1}}-(F_{f_i}+g_i'(0))\|_U\lesssim \sup_{0\le t\le 1} \|g_i''(t)\|_U.$$
  Therefore, for any $n$,
  \begin{equation}\label{eq:sum-taylor-est}
    \left\|F_{f_n}-\sum_{i=0}^{n-1} g'_i(0)\right\|_U\lesssim \sum_{i=0}^{n-1} \sup_{0\le t\le 1} \|g_i''(t)\|_U.
  \end{equation}
  Furthermore,
  \begin{align*}
    \left\|\sum_{i=0}^{n-1} g_i'(0)\right\|_U^2 & = \sum_{\substack{i=0,\dots,n-1\\j=0,\dots,n-1}} \langle g_i'(0), g_j'(0)\rangle \\
    &= \sum_{i=0}^{n-1} \|g_i'(0)\|_U^2 + 2 \sum_{0\le i< j< n} \langle g_i'(0), g_j'(0)\rangle.
  \end{align*}
  
  Let $\epsilon$ be as in Lemma~\ref{lem:mainBounds} and suppose that $n<\epsilon A^4$. Then on one hand,
  $$\sum_{i=0}^{n-1} \left\|g_i'(0)\right\|_U^2 \approx n A^{-2}.$$
  On the other hand, 
  $$\sum_{0\le i< j< n} |\langle g_i'(0),g_j'(0)\rangle|\lesssim n^2\rho^{-\epsilon},$$
  so if $\rho$ is sufficiently large, then 
  \begin{equation} \label{eq:sum-of-first}
    \left\|\sum_{i=0}^{n-1} g_i'(0)\right\|_U \gtrsim \sqrt{n} A^{-1},
  \end{equation}
  while 
  \begin{equation} \label{eq:sum-of-second}
    \sum_{i=0}^{n-1} \sup_{0\le t\le 1} \|g_i''(t)\|_U \lesssim n A^{-3}.
  \end{equation}
  Combining these estimates with \eqref{eq:sum-taylor-est}, we see that there is some $c>1$ such that 
  $$\|F_{f_n}\|_U \ge c^{-1} \sqrt{n} A^{-1} - cn A^{-3}.$$
  
  Let $\delta=\min\{\epsilon, c^{-4}/16\}$ and take $N=\lfloor \delta A^4\rfloor$. When $A$ is sufficiently large, 
  $$\|F_{f_N}\|_U \ge \frac{c^{-3}}{5} A - \frac{c^{-3}}{15} A \gtrsim A,$$
  as desired.
\end{proof}

These bounds point to a possible link between the norm of $F_{f_n}$ and the $\beta$--numbers studied in Section~\ref{sec:beta-numbers}. The bounds in Section~\ref{sec:beta-numbers} show that there is a $\delta>0$ such that if $\rho$ is sufficiently large and $n<\delta A^4$, then
\begin{equation}\label{eq:betasgl-2}\int_{0}^R\int_{\Psi_{f_n}(U)}  \beta_{\Gamma_{f_n}}(v,r)^2 \ud v \frac{\ud r}{r}\gtrsim n A^{-2}.
\end{equation}
Each layer of bumps with aspect ratio $A$ contributes roughly $A^{-2}$ to the integral.

Similarly, the proof of Proposition~\ref{prop:mainBounds-tilde-T} shows that if $\rho$ is sufficiently large and $n<\delta A^4$, then 
$\|F_{f_n}\|_U^2\approx n A^{-2}.$
Indeed, the proof shows that
$\|F_{f_n}\|_U^2\approx \sum_{i=1}^n \|g_i'(0)\|^2_U$ when $\frac{n}{A^4}$ is small. Since $\|g_i'(0)\|^2_U\approx A^{-2}$ when $i<\epsilon A^4$, each step in the construction of $f_n$ contributes roughly $A^{-2}$ to $\|F_{f_n}\|_U^2$. 

This suggests the following question.
\begin{question}\label{q:betasgl-q}
How is the integral \eqref{eq:betasgl-2} for an intrinsic Lipschitz graph $\Gamma$ related to the $L_2$--norm of the Riesz transform for functions on $\Gamma$?
\end{question}

In the rest of this paper, we will prove Propositions~\ref{prop:perturb-formula} and \ref{prop:second-bounds-v2} and Lemma~\ref{lem:mainBounds}. 

We prove Proposition~\ref{prop:perturb-formula} in Section~\ref{sec:perturbations}. The key step is to write  $G^{r,R}_{\zeta,\psi}(t)$ in two ways, \eqref{eq:grrab-large} and \eqref{eq:grrab-small}. 
The Euclidean analogues of these expressions are identical, but since $\HH$ is noncommutative, they differ in $\HH$. In practice, \eqref{eq:grrab-large} is easier to bound when $r$ and $R$ are large and \eqref{eq:grrab-small} is easier when $r$ and $R$ are small, so the two expressions together let us bound $G^{r,R}_{\zeta,\psi}$ and its derivatives at all scales. 

By Section~\ref{subsec:rescaling}, we can rescale $f_i$ and $\nu_i$ to obtain functions $\alpha$ and $\gamma$ that satisfy the bounds in Lemma~\ref{lem:gamma-deriv-bounds}. By the scale-invariance of the Riesz transform, $\|G_{f_i,\nu_i}'(0)\|_\infty = \|G_{\alpha,\gamma}'(0)\|_\infty$, so part (\ref{it:first-upper}) of Lemma~\ref{lem:mainBounds} follows from Proposition~\ref{prop:perturb-formula}.

Similarly, in Section~\ref{sec:second-deriv}, we use \eqref{eq:grrab-large} and \eqref{eq:grrab-small} again to prove Proposition~\ref{prop:second-bounds-v2}. As before, $\|G_{f_i,\nu_i}''(0)\|_\infty = \|G_{\alpha,\gamma}''(0)\|_\infty$, so part (\ref{it:second-upper}) of Lemma~\ref{lem:mainBounds} follows from Proposition~\ref{prop:second-bounds-v2}.

To prove parts (\ref{it:quasi-ortho}) and (\ref{it:first-lower}) of Lemma~\ref{lem:mainBounds}, we approximate $G_{\alpha,\gamma}'(0)$ by a translation-invariant singular integral operator on a plane. For any vertical plane $P\subset \HH$, let $\lambda_P\from \HH\to \R$ be the affine function such that $\Gamma_{\lambda_P}=P$, and let $H_{P,\gamma}\from P\to \R$ be the function
$$H_{P,\gamma}(q) := G'_{\lambda_P,\gamma}(0)(q)$$
for all $q\in P$. The map $\gamma\mapsto H_{P,\gamma}$ is then a translation-invariant operator from functions on $P$ to functions on $P$. 

Given $v\in \Gamma_\alpha$, we let $P_v$ be the vertical tangent plane to $\Gamma_\alpha$ at $v$.
By Lemma~\ref{lem:a-Taylor} and Lemma~\ref{lem:gamma-deriv-bounds}, $P_v$ is close to $\Gamma_\alpha$ on a ball around $v$ whose radius grows with $\rho$. In Section~\ref{sec:approx-by-linear}, we show that $H_{P_v,\gamma}$ approximates $G'_{\alpha,\gamma}(0)$ on a ball around $v$ whose radius also grows with $\rho$. 
We use this approximation to prove the lower bound $\|G_{f_i,\nu_i}'(0)\|_{L_2(U)}\gtrsim A^{-1}$ (Section~\ref{sec:first-lower}), to prove that $G_{f_i,\nu_i}'(0)$ is continuous as a function from $V_0$ to $\R$ (Lemma~\ref{lem:Fprime-Lipschitz-dupe}), and to prove the orthogonality bound \eqref{eq:quasi-ortho} (Section~\ref{sec:quasi-orthog}). This completes the proof of Lemma~\ref{lem:mainBounds}.

\section{First-order estimates for $\tilde{T}_{\zeta + \psi}$} \label{sec:perturbations}
Let $\zeta, \psi \from \HH \to \R$ be smooth functions that are constant on cosets of $\paramY$ and suppose that $\zeta(\zero)=0$. Let $G_{\zeta,\psi}(t)=F_{\zeta+t\psi}=\tilde{T}_{\zeta+t\psi} \one$ and $G^{r,R}_{\zeta,\psi}=F^{r,R}_{\zeta+t\psi}$ be as in Section~\ref{sec:sing-int-perturbed}.

In this section, we will derive expressions for $(G^{r,R}_{\zeta,\psi})'(t)$ and prove the following lemma, which is a quantitative version of Proposition~\ref{prop:perturb-formula}. 
\begin{lemma}\label{lem:G'-formula-and-bound-rev}
  Let $\zeta, \psi\from \HH \to \R$ be smooth functions that are constant on cosets of $\paramY$. Suppose that  $\|\psi\|_\infty <\infty$ and that $\zeta$ is intrinsic Lipschitz. Then, for any $p\in \Gamma_\zeta$,
  $$\lim_{\substack{r\to 0 \\ R\to \infty}} (G^{r,R}_{\zeta,\psi})'(0)(p) = G'_{\zeta,\psi}(0)(p).$$

  Furthermore, let $0<r<1$. Let $L=\|\nabla_\zeta \zeta\|_\infty$ and let $B= B(p, (2L+1) r)$. For a smooth function $g\from \HH\to \R$, define
  $$\|g\|_{W_\zeta(B)} := \max_{\substack{E\in \{\partial_z, \nabla_\zeta\}^*\\ \ell(E)\le 2}} \|E\psi\|_{L_\infty(B)},$$
  and
  $$\|g\|_{W'_\zeta(B)} := \max_{\substack{E\in \{\partial_z, \nabla_\zeta\}^*\setminus \{\id,\nabla_\zeta\} \\ \ell(E)\le 2}} \|E\psi\|_{L_\infty(B)}.$$
  
  Then for any $C>0$ and any $s$ and $S$ such that $0<s\le r\le S$, if $\|\zeta\|_{W'_\zeta(B)}\le C$, then 
  $$\left|G_{\zeta,\psi}'(0)(p) - \big(G^{s,S}_{\zeta,\psi}\big)'(0)(p)\right| \lesssim_{L,C} \|\psi\|_{W_\zeta(B)} s + \|\psi\|_\infty S^{-1}.$$
  In particular, since $G^{r,r}_{\zeta,\psi}(t)(p)=0$ for all $t$ and $p$,
  $$|G_{\zeta,\psi}'(0)(p)| = |G_{\zeta,\psi}'(0)(p) - (G^{r,r}_{\zeta,\psi})'(0)(p)| \lesssim_{L,C}  (r+r^{-1})\|\psi\|_{W_\zeta(B)}.$$
\end{lemma}

Proposition~\ref{prop:perturb-formula} follows immediately.
\begin{proof}[Proof of Proposition~\ref{prop:perturb-formula}]
  Suppose that $\alpha$ and $\gamma$ are as in Lemma~\ref{lem:gamma-deriv-bounds}.  Then $\|\nabla_\alpha\alpha\|_\infty\le 1$, $\|\alpha\|_{W'_\alpha(\HH)}\lesssim A^{-1}\rho^{-1}$, and $\|\gamma\|_{W_\alpha(\HH)}\lesssim A^{-1}$ when $\rho$ is sufficiently large. Therefore, letting $r=1$, 
  $$|G_{\alpha,\gamma}'(0)(p)| \lesssim A^{-1}.$$
\end{proof}

We prove Lemma~\ref{lem:G'-formula-and-bound-rev} by calculating $G^{r,R}_{\zeta,\psi}(t)$ in two ways. By left-invariance, it suffices to consider the case that $\zeta(\zero)=0$ and $p=\zero$. Then on one hand, by \eqref{eq:def-F-trunc},
$$G^{r,R}_{\zeta,\psi}(t)(\zero) = \int_{Y^{t\psi(\zero)} A_{r,R}} \widehat{K}(\Psi_{\zeta + t\psi}(\zero)^{-1}\Psi_{\zeta + t\psi}(v)) \ud v.$$
The domain of integration depends on $t$, but since the integrand is constant on cosets of $\paramY$, we can replace $Y^{t\psi(\zero)} A_{r,R}$ by
$$A^t_{r,R} := \Pi(Y^{t\psi(\zero)} A_{r,R}) = Y^{t\psi(\zero)} A_{r,R}Y^{-t\psi(\zero)}.$$
This is a copy of $A_{r,R}$, sheared in the $z$--direction, and 
\begin{equation}\label{eq:grr-no-translate}
  G^{r,R}_{\zeta,\psi}(t)(\zero) = \int_{A^{t}_{r,R}} \widehat{K}(Y^{-t\psi(\zero)}\Psi_{\zeta}(v)Y^{t\psi(v)}) \ud v.
\end{equation}
Differentiating \eqref{eq:grr-no-translate} gives an expression for $(G^{r,R}_{\zeta,\psi})'$ which is found in Lemma~\ref{lem:grrab-large} below. The changing boundary will lead to boundary terms in the derivative, but we will see that when $r$ and $R$ are large, this derivative is small.

On the other hand, just as we translated $\zeta$ so that the graph of $\zeta$ goes through $\zero$, we can translate $\zeta + \tau \psi$ so that its graph goes through $\zero$. By Lemma~\ref{lem:ilg-translations}, there is a function $\zeta_\tau$ such that $\Gamma_{\zeta_\tau} = Y^{-\tau\psi(\zero)} \Gamma_{\zeta + \tau \psi},$ which can be written as follows. For any $\tau \in \R$ and $w\in V_0$, let
\begin{equation}\label{eq:def-wt}
  w_\tau := Y^{\tau \psi(\zero)} w Y^{-\tau \psi(\zero)} = w - \tau \psi(\zero) x(w) Z
\end{equation}
and
\begin{equation}\label{eq:def-zt}
  \zeta_\tau(w) := (\zeta + \tau \psi)(w_\tau) - (\zeta + \tau \psi)(\zero).
\end{equation}
Then $\zeta_0=\zeta$, $\zeta_\tau(\zero)=0$, and 
$\Gamma_{\zeta_\tau} = Y^{-\tau\psi(\zero)} \Gamma_{\zeta + \tau \psi}.$
By the left-invariance of $\tilde{T}$, for any $w\in V_0$ and $\tau\in \R$, we have
$$G^{r,R}_{\zeta,\psi}(\tau)(w) = F^{r,R}_{\zeta+\tau\psi}(w) = F^{r,R}_{\zeta_\tau}(Y^{-\tau\psi(\zero)}w) = F^{r,R}_{\zeta_\tau}(w_\tau),$$
where the last equality uses the fact that $F^{r,R}_{\zeta_\tau}$ is constant on cosets of $\paramY$.
In particular, 
\begin{equation}\label{eq:grr-translate}
  G^{r,R}_{\zeta,\psi}(\tau)(\zero) = F^{r,R}_{\zeta_\tau}(\zero) =
  \int_{A_{r,R}} \widehat{K}(\Psi_{\zeta_\tau}(\zero)^{-1}\Psi_{\zeta_\tau}(v)) \ud v = \int_{A_{r,R}} \widehat{K}(\Psi_{\zeta_\tau}(v)) \ud v,
\end{equation}
so we can compute $(G^{r,R}_{\zeta,\psi})'$ by differentiating \eqref{eq:grr-translate} (see Lemma~\ref{lem:grrab-small} below). This avoids the boundary terms in Lemma~\ref{lem:grrab-large}. We will see that when $r$ and $R$ are small, the derivative of \eqref{eq:grr-translate} is small.


We first consider the derivative of \eqref{eq:grr-no-translate}. For any $R>0$, let $M_R(x)=\frac{1}{4}\sqrt{R^4-x^4}$ so that $B(\zero,R)\cap V_0 = \{(x,0,z)\mid |x|\le R, |z| \le M_R(x)\}$. Recall that $\YL$ is the left-invariant vector field $\YL(x,y,z):=(0,1,\frac{x}{2})$ and $\YR$ is the right-invariant vector field $\YR(x,y,z):=(0,1,-\frac{x}{2})$. 

For the rest of this section, we suppose that $\zeta$ and $\psi$ are as in Lemma~\ref{lem:G'-formula-and-bound-rev} and that $\zeta(\zero)=0$ so that $\zero\in \Gamma_\zeta$. We let $0<r<1$, $L=\|\nabla_\zeta \zeta\|_\infty$, and $B=B(\zero,(2L + 1)r)$, and we suppose that $\|\zeta\|_{W'_\zeta(B)}\le C$. For $q\in \HH$, we let  $\overline{q}=\Psi_{\zeta}(q)$. 
\begin{lemma}\label{lem:grrab-large}
  Let $R' > R > 0$. Then
  \begin{multline}\label{eq:grrab-large}
    (G^{R,R'}_{\zeta,\psi})'(0)(\zero) = \int_{A_{R,R'}} \psi(w)\YL\widehat{K}(\overline{w}) - \psi(\zero)\YR\widehat{K}(\overline{w}) \ud q \\ - \psi(\zero) \int_{-R}^R x\cdot \widehat{K}(\Psi_\zeta(x,0,z)) \bigg|_{z=-M_R(x)}^{M_R(x)} \ud x + \psi(\zero) \int_{-r}^r x\cdot \widehat{K}(\Psi_\zeta(x,0,z)) \bigg|_{z=-M_r(x)}^{M_r(x)} \ud x.
  \end{multline}
  where $f(z)\big|_{z=a}^b$ denotes $f(b)-f(a)$.
  Further,
  $$\left|(G^{R,R'}_{\zeta,\psi})'(0)(\zero)\right| \lesssim \|\psi\|_\infty R^{-1}.$$
\end{lemma}
\begin{proof}
  We proceed by differentiating \eqref{eq:grr-no-translate}. By the definition of $M_R$, we have 
  $$A_{R,R'} = \{(x,0,z) \mid z\in [-M_{R'}(x), M_{R'}(x)] \setminus (-M_R(x), M_R(x))\}.$$
  Let $A_{R,R'}(x,t) := \{z\mid (x,0,z)\in A^t_{R,R'}\}$. Since $Y^{y} (x,0,z) Y^{-y} = (x, 0, z - yx)$, we have
  \begin{multline*} 
    A_{R,R'}(x,t) = [-M_{R'}(x) - t \psi(\zero) x, M_{R'}(x) - t \psi(\zero) x] \\ \setminus (-M_R(x) - t \psi(\zero) x, M_R(x) - t \psi(\zero) x).
  \end{multline*}
  Let $w=(x,0,z)$ and $\lambda_t(w)=Y^{-t \psi(\zero)} \overline{w}Y^{t \psi(w)}$. Then by \eqref{eq:grr-no-translate},
  \begin{align*}
    G^{R,R'}_{\zeta, \psi}(t) & = \int_{-R'}^{R'} \int_{A_{R,R'}(x,t)} \widehat{K}(\lambda_t(x,0,z))\ud z \ud x \\
    (G^{R,R'}_{\zeta, \psi})'(t) & = \int_{A^t_{R,R'}}  \frac{\ud}{\ud t} [\widehat{K}(\lambda_t(q))] \ud q - \int_{-R'}^{R'} \psi(\zero) x\cdot \widehat{K}(\lambda_t(x,0,z)) \bigg|_{z=-M_{R'}(x) - t \psi(\zero) x}^{M_{R'}(x) - t \psi(\zero) x} \ud x \\
                                & \qquad + \int_{-R}^R \psi(\zero) x\cdot \widehat{K}(\lambda_t(x,0,z)) \bigg|_{z=-M_R(x)- t \psi(\zero) x}^{M_R(x)- t \psi(\zero) x} \ud x.
  \end{align*}

  When $t=0$,
  \begin{align*}
    (G^{R,R'}_{\zeta,\psi})'(0)(\zero)
    & = \int_{A_{R,R'}} \psi(q)\YL\widehat{K}(\overline{q}) - \psi(\zero)\YR\widehat{K}(\overline{q}) \ud q \\
    & \qquad - \psi(\zero) \int_{-R'}^{R'} x \widehat{K}(\Psi_\zeta(x,0,z)) \bigg|_{z=-M_{R'}(x)}^{M_{R'}(x)} \ud x + \psi(\zero) \int_{-R}^R x \widehat{K}(\Psi_\zeta(x,0,z)) \bigg|_{z=-M_R(x)}^{M_R(x)} \ud x \\
    & =: I_1 - I_2^{R'} + I_2^R.
  \end{align*}
  This proves \eqref{eq:grrab-large}.

  We thus consider $I_1$, $I_2^{R'}$, and $I_2^R$. Since $\YL \widehat{K}$ and $\YR \widehat{K}$ are homogeneous of degree $-4$,
  \begin{align*}
    \left|\psi(w)\YL\widehat{K}(\overline{w}) - \psi(\zero)\YR\widehat{K}(\overline{w})\right| \lesssim \|\psi\|_\infty \|w\|_{\Kor}^{-4}.
  \end{align*}
  By Lemma~\ref{lem:polar},
  $$|I_1|\lesssim \int_{A_{R,R'}} \|\psi\|_\infty \|w\|_{\Kor}^{-4}\ud w \lesssim \int_R^{R'} \|\psi\|_\infty \kappa^{-4}\cdot \kappa^2 \ud \kappa \le \|\psi\|_\infty R^{-1}.$$

  Let $s\in [R,R']$. 
  Since $\|(x,0,M_s(x))\|_{\Kor}=s$, we have $|\widehat{K}(\Psi_{\zeta}(x,0, M_s(x)))| \lesssim s^{-3}$ and
  $$|I^s_2(t)|\lesssim \psi(\zero) \int_{-s}^s |x| s^{-3}\ud x \lesssim \|\psi\|_\infty s^{-1}.$$
  Putting these bounds together,
  $$|(G^{R,R'}_{\zeta,\psi})'(0)(\zero)|\lesssim \|\psi\|_\infty R^{-1} + \|\psi\|_\infty (R')^{-1} + \|\psi\|_\infty R^{-1} \lesssim \|\psi\|_\infty R^{-1},$$
  as desired.
\end{proof}

Now we differentiate \eqref{eq:grr-translate}.
\begin{lemma}\label{lem:grrab-small}
  Let $0 < s'< s \le r$.
  Then
  \begin{equation}\label{eq:grrab-small}
    (G^{s',s}_{\zeta,\psi})'(0)(\zero) = \int_{A_{s',s}} \left(\psi(w)-\psi(\zero)-\psi(\zero)x\partial_z \zeta(w)\right) \YL\widehat{K}(\overline{w})\ud w.
  \end{equation}
  Furthermore, 
  \begin{equation}\label{eq:grrab-small-ineq}
    |(G^{s',s}_{\zeta,\psi})'(0)(\zero)| \lesssim_C \|\psi\|_{W_\zeta(B)} s.
  \end{equation}
\end{lemma}

To prove this lemma, we will need the following bound, which will also be used in Section~\ref{sec:second-deriv}. Recall that $\theta(x,y,z)=(-x,-y,z)$ is rotation around the $z$--axis.

\begin{lemma} \label{lem:kernel-sym-bound}
  Let $j\ge 0$ and let $M$ be a smooth $(-j)$--homogeneous kernel. Then
  \begin{align*}
    |M(\Psi_a(w)) - (-1)^j M(\Psi_a(\theta(w)))| \lesssim_{M, L} \|\zeta\|_{W'_\zeta(B)} \|w\|_\Kor^{1-j}, \qquad \forall w \in B(\zero,r).
  \end{align*}
\end{lemma}

\begin{proof}
  
  Since $M$ is $(-j)$--homogeneous, we note that $M(\theta(p))=M(s_{-1}(p))=(-1)^jM(p)$ for all $p$.
  As in the proof of Lemma~\ref{lem:riesz-converge}, for all $w\in V_0$, the points $\theta(\Psi_\zeta(w))=\theta(w) Y^{-\zeta(w)}$ and $\Psi_\zeta(\theta(w))=\theta(w)Y^{\zeta(\theta(w))}$ lie in the same coset of $\paramY$. Furthermore, by Lemma~\ref{lem:a-Taylor}, if $w\in B(\zero, r)$, then
  $$\zeta(w) = x(w) \nabla_\zeta \zeta(\zero) + O_L(C \|w\|_{\Kor}^2),$$
  so the distance between these points satisfies
  \begin{align}
    |\zeta(w) + \zeta(\theta(w))| = (x(w) - x(w)) \nabla_\zeta \zeta(\zero) + O_L(C \|w\|_{\Kor}^2) \lesssim_L C \|w\|_{\Kor}^2. \label{eq:ab-sym}
  \end{align}
  
  By the Mean Value Theorem, there is a point $k(w)$ lying on the horizontal line between $\theta(\Psi_\zeta(w))$ and $\Psi_\zeta(\theta(w))$ such that 
  \begin{multline*}
    |M(\Psi_\zeta(w)) - (-1)^j M(\Psi_\zeta(\theta(w)))| = |M(\theta(\Psi_\zeta(w))) - M(\Psi_\zeta(\theta(w)))| \\
    = |\YL M(k(w))| |\zeta(w) + \zeta(\theta(w))| \lesssim_L C \|w\|_{\Kor}^2 |\YL M(k(w))|.
  \end{multline*}
  Since $\zeta$ is intrinsic Lipschitz with constant depending on $L$, we have  
  $$\|k(w)\|_\Kor \approx_L \|\Pi(w)\|_\Kor\le 2\|w\|_\Kor.$$
  By Lemma~\ref{lem:deriv-homog}, $\YL M$ is $(-j-1)$--homogeneous, so
  \begin{align*}
    |M(\Psi_a(w)) - (-1)^j M(\Psi_a(\theta(w)))| \lesssim_{L,M} C \|w\|_{\Kor}^2 \|k(w)\|_{\Kor}^{-j-1} \lesssim_L C \|w\|_\Kor^{1-j}
  \end{align*}
  as desired.
\end{proof}

Now we prove Lemma~\ref{lem:grrab-small}. We take advantage of the symmetry of $A_{s',s}$ by decomposing functions into odd and even parts. For any function $f \from \HH \to \R$, we have the following even-odd decomposition:
\begin{align}
  f(w) = \frac{1}{2} (f(w) + f(\theta(w))) + \frac{1}{2} (f(w) - f(\theta(w))) =: f^\sE(w) + f^\sO(w). \label{eq:pointwise-even-odd}
\end{align}
Let $E \subseteq \HH$ be a subset for which $\theta(E) = E$.  As $\int_E f(w) \ud w = \int_E f(\theta(w)) \ud w$, we get that $\int_E f^\sO(w) \ud w = 0$ and so if $f$ is integrable on $E$, then
\begin{align}
  \int_E f(w) \ud w = \int_E f^\sE(w) \ud w. \label{eq:int-even-odd}
\end{align}
Moreover, if $g\from \HH \to \R$ and $fg$ is integrable on $E$, then
\begin{equation}\label{eq:even-odd-product}
  \int_E f g\ud w = \int_E (f^\sE + f^\sO)(g^\sE + g^\sO) \ud w = \int_E f^\sE g^\sE + f^\sO g^\sO \ud w.
\end{equation}

\begin{proof}[Proof of Lemma~\ref{lem:grrab-small}]
  By \eqref{eq:grr-translate},
  \begin{equation}\label{eq:atau-cons}
    (G^{s',s}_{\zeta,\psi})'(0)(\zero) = \frac{\ud}{\ud \tau}\left[\int_{A_{s',s}} \widehat{K}(\Psi_{\zeta_\tau}(w))\ud w\right]_{\tau=0} = \int_{A_{s',s}} \partial_\tau[\zeta_\tau(w)](0) \cdot \YL\widehat{K}(\overline{w})\ud w.
  \end{equation}
  We differentiate \eqref{eq:def-zt} to get 
  \begin{align*}
    \partial_\tau[\zeta_\tau(w)]
    &= - \psi(\zero) x(w) \partial_z [\zeta+\tau \psi](w_\tau) + \psi(w_\tau) - \psi(\zero)
  \end{align*}
  where $w_\tau=w - \tau \psi(\zero) x(w) Z$ is as in \eqref{eq:def-wt}. Let 
  $$m(w) = \partial_\tau[\zeta_\tau(w)](0) = \psi(w) - \psi(\zero) - \psi(\zero) x(w) \partial_z \zeta(w),$$ so that
  $$(G^{s',s}_{\zeta,\psi})'(0)(\zero) = \int_{A_{s',s}} m(w) \YL\widehat{K}(\overline{w})\ud w;$$
  this is \eqref{eq:grrab-small}.

  For $w\in V_0$, let $N_{\YL \widehat{K}}(w) :=\YL \widehat{K}(\overline{w})$. We will estimate $I:=(G^{s',s}_{\zeta,\psi})'(0)(\zero)$ by decomposing $m$ and $N_{\YL \widehat{K}}$ into odd and even parts.
  
  By \eqref{eq:even-odd-product},
  $$I = \int_{A_{s',s}} m^\sE(w) N_{\YL \widehat{K}}^{\sE}(w) + m^\sO(w) N_{\YL \widehat{K}}^{\sO}(w)\ud w.$$
  Let $w\in B(\zero,r)$ and let $\kappa=\|w\|_\Kor$. Note that $0\le \kappa\le r\le 1$.
  By Lemma~\ref{lem:a-Taylor},
  $$\psi(w) - \psi(\zero) = x(w) \nabla_{\zeta} \psi(\zero) + O(\|\psi\|_{W_\zeta(B)} \kappa^2),$$
  and
  $$\partial_z \zeta(w) = \partial_z \zeta(\zero) + O\left(\|\partial_z^2 \zeta\|_{L_\infty(B)}\kappa^2 + \|\nabla_\zeta \partial_z \zeta\|_{L_\infty(B)}\kappa\right) = \partial_z \zeta(\zero) + O(C\kappa).$$
  Therefore,
  \begin{multline}\label{eq:m-affine}
    m(w) = x(w) \nabla_{\zeta} \psi(\zero) - \psi(\zero) x(w) \partial_z \zeta(\zero) + O(\|\psi\|_{W_\zeta(B)} \kappa^2) + O(\psi(\zero) x(w) C \kappa) \\
    = x(w) \nabla_{\zeta} \psi(\zero) - \psi(\zero) x(w) \partial_z \zeta(\zero) + O\big(\|\psi\|_{W_\zeta(B)}(1 + C) \kappa^2\big).
  \end{multline}
  Thus $|m^\sE(w)|\lesssim_C \|\psi\|_{W_\zeta(B)} \kappa^2$.
  
  Similarly, $\psi(w) - \psi(\zero) = O(\|\psi\|_{W_\zeta(B)} \kappa)$ and $\partial_z\zeta = O(\|\zeta\|_{W'_\zeta(B)})$, so
  \begin{equation}\label{eq:m-zero}
    m(w) = O(\|\psi\|_{W_\zeta(B)} \kappa) + O(\psi(0) \kappa \|\zeta\|_{W'_\zeta(B)}) = O(\|\psi\|_{W_\zeta(B)}(1 + C) \kappa),
  \end{equation}
  and $|m^\sO(w)|\lesssim_C \|\psi\|_{W_\zeta(B)} \kappa$. Since $\YL \widehat{K}$ is $-4$--homogeneous, $|N_{\YL \widehat{K}}^{\sE}(w)| \lesssim \kappa^{-4}$. By Lemma~\ref{lem:kernel-sym-bound}, $|N_{\YL \widehat{K}}^{\sO}(w)| \lesssim_C \kappa^{-3}$.

  Therefore,
  \begin{equation*}
    |m^\sE(w) N_{\YL \widehat{K}}^{\sE}(w) + m^\sO(w) N_{\YL \widehat{K}}^{\sO}(w)|\lesssim_C 
    \|\psi\|_{W_\zeta(B)} \kappa^{-2}.
  \end{equation*}
  By Lemma~\ref{lem:polar},
  $$|I|\lesssim_C \int_{s'}^s \|\psi\|_{W_\zeta(B)} \kappa^{-2}\cdot \kappa^2 \ud \kappa \le \|\psi\|_{W_\zeta(B)} s,$$
  as desired.
\end{proof}

Finally, we prove Lemma~\ref{lem:G'-formula-and-bound-rev}.
\begin{proof}[Proof of Lemma~\ref{lem:G'-formula-and-bound-rev}]
  For $w\in V_0$ and $\tau\in \R$, let $\psi_\tau(w):= \psi(w_\tau)$, where $w_\tau$ is as in \eqref{eq:def-wt}. Let $t\in \R$. Then by \eqref{eq:def-zt},
  $$\zeta_{\tau + t}(w) = 
  (\zeta + \tau \psi + t\psi)(w_{\tau + t}) - (\zeta + \tau \psi + t\psi)(\zero) = (\zeta_\tau + t \psi_\tau)(Y^{t\psi(\zero)} w) - t \psi(\zero),$$
  so by Lemma~\ref{lem:ilg-translations},
  $$\Gamma_{\zeta_{\tau+t}} = Y^{-t\psi(\zero)} \Gamma_{\zeta_\tau + t \psi_\tau}.$$
  Therefore, for $s<S$, 
  $$G^{s,S}_{\zeta, \psi}(\tau + t)(\zero)=G^{s,S}_{\zeta_\tau,\psi_\tau}(t)(\zero).$$
  Differentiating with respect to $t$ gives
  $$(G^{s,S}_{\zeta,\psi})'(\tau)(\zero) = (G^{s,S}_{\zeta_\tau,\psi_\tau})'(0)(\zero).$$

  Let $0<s'<s \le r$ and let $r\le S <S'$. Let $B'=B(\zero,(2L+3) r)$.
  If $\tau$ is sufficiently small, then $\|\nabla_{\zeta_\tau}\zeta_\tau\|_\infty < L+1$, $\|\zeta_\tau\|_{W'_{\zeta_\tau}(B')} \le C+1$, and $\|\psi_\tau\|_{W_{\zeta_\tau}(B')} \le 2 \|\psi\|_{W_{\zeta}(B)}$.
  Then Lemma~\ref{lem:grrab-large} and Lemma~\ref{lem:grrab-small} imply that
  \begin{equation}\label{eq:G'-unif-Cauchy}
    \left|(G^{s',S'}_{\zeta_\tau, \psi_\tau})'(0)(\zero) - (G^{s,S}_{\zeta_\tau, \psi_\tau})'(0)(\zero)\right| = \left|(G^{s',s}_{\zeta_\tau, \psi_\tau})'(0)(\zero) + (G^{S, S'}_{\zeta_\tau, \psi_\tau})'(0)(\zero)\right| \lesssim_{L,C} \|\psi\|_\infty s + \|\psi\|_{W_\zeta(B)} S^{-1}.
  \end{equation}
  That is, $(G^{s,S}_{\zeta, \psi})'(\tau)(\zero) = (G^{s,S}_{\zeta_\tau, \psi_\tau})'(0)(\zero)$ is Cauchy as $s\to 0$ and $S\to \infty$, with bounds independent of $\tau$.
  Thus, $(G^{s,S}_{\zeta,\psi})'(\tau)$ converges uniformly as $s\to 0$ and $S\to \infty$. This lets us pass the derivative under the limit, so by Lemma~\ref{lem:riesz-converge},
  $$G_{\zeta, \psi}'(\tau) = \frac{\ud}{\ud \tau} \lim_{\substack{s\to 0 \\ S\to \infty}} G^{s,S}_{\zeta, \psi}(\tau) = \lim_{\substack{s\to 0 \\ S\to \infty}} (G^{s,S}_{\zeta, \psi})'(\tau).$$
  Finally, \eqref{eq:G'-unif-Cauchy} implies that
  $$|G_{\zeta,\psi}'(0)(p) - (G^{s,S}_{\zeta,\psi})'(0)(p)| \lesssim_{L,C} \|\psi\|_\infty s + \|\psi\|_{W_\zeta(B)},$$
  as desired.
\end{proof}

This implies Proposition~\ref{prop:perturb-formula} and thus part (\ref{it:first-upper}) of Lemma~\ref{lem:mainBounds}.

\section{Approximating by a planar singular integral}\label{sec:approx-by-linear}

For any vertical plane $P\subset \HH$ with nonzero slope, let $\lambda_P\from \HH\to \R$ be the affine function such that $\Gamma_{\lambda_P}=P$. For any bounded smooth function $\phi\from \HH \to \R$ which is constant on cosets of $\paramY$, let $H_{P,\phi}\from P\to \R$ be the function
$$H_{P,\phi}(p) := G'_{\lambda_P,\phi}(0)(p)$$
for any $p\in P$. By Lemma~\ref{lem:G'-formula-and-bound-rev} and Lemma~\ref{lem:grrab-small}, $H_{P,\phi}(p)$ exists and
$$H_{P,\phi}(p) = \lim_{\substack{r\to 0 \\ R\to \infty}} (G^{r,R}_{\lambda_P,\phi})'(0)(p) = \lim_{\substack{r\to 0 \\ R\to \infty}} \int_{p A_{r,R}} (\phi(q)-\phi(p)) \YL \widehat{K}(p^{-1} \Pi_P(q)) \ud q;$$
recall that we denote this limit by
\begin{equation}\label{eq:H-as-pv-v0}
  H_{P,\phi}(p) = \pv(p) \int_{p V_0} (\phi(q)-\phi(p))\YL \widehat{K}(p^{-1} \Pi_P(q)) \ud q.
\end{equation}
The functions which are constant on cosets of $\paramY$ are naturally identified with functions on $P$, so we can view $\phi\mapsto H_{P,\phi}$ as a singular integral operator acting on functions from $P$ to $\R$.
It is translation-invariant in the sense that if $P_0$ goes through $\zero$, $v_0\in P_0$, and $\hat{\phi}(v)=\phi(v + v_0)$ for all $v\in P_0$, then 
$$H_{P_0,\hat{\phi}}(v) = H_{P_0, \phi}(v + v_0)$$
for all $v\in P_0$.

In this section, we will show that when $P$ is tangent to $\Gamma_{f_i}$ at $p$, then $H_{P,\nu_i}$ approximates  $G_{f_i,\nu_i}'(0)$ in a neighborhood of $p$. We will use this to bound how quickly $G_{f_i,\nu_i}'(0)$ can vary, and in the next section, we will use this approximation to bound the correlation between $G_{f_i,\nu_i}'(0)$ and $G_{f_j,\nu_j}'(0)$ when $i\ne j$.

After rescaling $f_i$ and $\nu_i$ as in Section~\ref{subsec:rescaling}, it suffices to consider functions $\alpha$ and $\gamma$ that satisfy the conclusion of Lemma~\ref{lem:gamma-deriv-bounds}, i.e., satisfy \eqref{eq:gamma-jk-bounds} and \eqref{eq:alpha-jk-bounds} for some $c>0$. Many of the constants in the following bounds will depend on the value of $c$, so we omit $c$ from the subscripts for the rest of this section. We will prove the following lemmas.

\begin{lemma}[$H$ approximates $G'$]\label{lem:tildeH-approx-dupe}
  Let $\epsilon=\frac{1}{10}$. Let $\alpha$ and $\gamma$ satisfy Lemma~\ref{lem:gamma-deriv-bounds} for some sufficiently large $\rho$. Let $p\in \Gamma_\alpha$ and let $P$ be the tangent plane to $\Gamma_\alpha$ at $p$. For any $q\in P$ such that $d_\Kor(p,q)\le \rho^\epsilon$, 
  $$\left|G_{\alpha, \gamma}'(0)(q) - H_{P, \gamma}(q)\right| \lesssim \rho^{-\epsilon}.$$
  Furthermore, for any $0<r\le 1\le R$,
  \begin{equation}\label{eq:H-unif}
    \left|H_{P,\gamma}(q) - (G_{\lambda_P, \gamma}^{r,R})'(0)(q)\right| \lesssim A^{-1}(R^{-1} + r).
  \end{equation}
\end{lemma}

\begin{lemma}[Hölder bounds on $G'_{\alpha,\gamma}$]\label{lem:Fprime-Lipschitz-dupe}
  Let $\epsilon=\frac{1}{10}$. For all $p,q\in \Gamma_\alpha$,
  \begin{equation}\label{eq:Gprime-Holder}
    \left|G_{\alpha, \gamma}'(0)(p) - G_{\alpha, \gamma}'(0)(q)\right| \lesssim d_{\Kor}(p,q)^{\epsilon} + \rho^{-\frac{1}{2}}.
  \end{equation}
\end{lemma}

We apply Lemma~\ref{lem:Fprime-Lipschitz-dupe} to $G_{f_i,\nu_i}'$ by rescaling.
\begin{cor}[Hölder bounds on $G_{f_i,\nu_i}'$]\label{cor:Fprime-Holder}
  Let $\epsilon>0$ be as above. Let $i\ge 0$, let $r_i=A^{-1}\rho^{-i}$, and let $p,q\in \Gamma_{f_i}$. Then
  \begin{equation}\label{eq:Fprime-Holder}
    \left|G_{f_i,\nu_i}'(0)(p) - G_{f_i,\nu_i}'(0)(q)\right| \lesssim (r_i^{-1}d_{\Kor}(p,q))^{\epsilon} + \rho^{-\frac{1}{2}}.
  \end{equation}
\end{cor}
\begin{proof}
  Let $g=G_{f_i,\nu_i}'(0)$.
  Let $s_i=s^{-1}_{r_i}$. Let $\alpha(p) = r_i^{-1} f_i(s_i^{-1}(p))$ and $\gamma(p)=r_i^{-1} \nu_i(s_i^{-1}(p))$. These satisfy Lemma~\ref{lem:gamma-deriv-bounds} and $g(p)=G_{\alpha,\gamma}'(0)(s_i(p))$. If $p,q\in \Gamma_{f_i}$, then $s_i(p),s_i(q)\in \Gamma_\alpha$, so, by Lemma~\ref{lem:Fprime-Lipschitz-dupe},
  $$\left|g(p) - g(q)\right| = \left|G_{\alpha,\gamma}'(0)(s_i(p)) - G_{\alpha,\gamma}'(0)(s_i(q))\right| \lesssim (r_i^{-1}d_{\Kor}(p,q))^{\epsilon} + \rho^{-\frac{1}{2}}.$$
\end{proof}

The proofs of Lemmas~\ref{lem:tildeH-approx-dupe} and \ref{lem:Fprime-Lipschitz-dupe} are based on the following bounds.

\begin{lemma}\label{lem:horiz-gamma-bounds}
  Let $\alpha$ and $\gamma$ satisfy Lemma~\ref{lem:gamma-deriv-bounds}. Let $C>0$ and let $P$ be a vertical plane with $|\slope(P)|\le C$. Let $W=X+\slope(P) Y$.
  Let $\nabla_P = \nabla_{\lambda_P}$. Let $i,j\ge 0$ and $i+j\le 3$. Then for any $p \in P$,
  \begin{equation}\label{eq:horiz-gamma}
    |W^iZ^j \gamma(p)| = |\nabla_P^iZ^j \gamma(p)| \lesssim_C A^{-1}(1 + d_\Kor(p,\Gamma_\alpha))^{i}.
  \end{equation}
\end{lemma}


\begin{lemma}\label{lem:flat-pv-formula}
  Let $\alpha$ and $\gamma$ satisfy Lemma~\ref{lem:gamma-deriv-bounds}. Let $p\in \Gamma_\alpha$ and let $P$ be the tangent plane to $\Gamma_\alpha$ at $p$. When $\rho$ is sufficiently large,
  $$G_{\alpha,\gamma}'(0)(p) = H_{P,\gamma}(p) + O(\rho^{-\frac{1}{2}}).$$
\end{lemma}

\begin{lemma}\label{lem:changing-planes}
  Let $\alpha$ and $\gamma$ satisfy Lemma~\ref{lem:gamma-deriv-bounds}. Let $p\in \Gamma_\alpha$ and let $C>0$.  Then for any two planes $P$ and $Q$ through $p$ with slopes at most $C$,
  \begin{align*}
    |H_{P,\gamma}(p)-H_{Q,\gamma}(p)| \lesssim_C |\slope P-\slope Q|
  \end{align*}
\end{lemma}

\begin{lemma}\label{lem:moving-basepoint}
  Let $\alpha$ and $\gamma$ satisfy Lemma~\ref{lem:gamma-deriv-bounds}. Let $p\in \Gamma_\alpha$ and $q\in \HH$. Let $P$ be a plane through $\zero$ with $|\slope P| \le 1$, and suppose that $d_{\Kor}(p,q)\le 1$. Then
  $$|H_{p P,\gamma}(p) - H_{q P,\gamma}(q)| \lesssim d_\Kor(p,q)^{\frac{1}{5}}.$$
\end{lemma}

Given these lemmas, we prove Lemmas~\ref{lem:tildeH-approx-dupe} and \ref{lem:Fprime-Lipschitz-dupe} as follows.
\begin{proof}[Proof of Lemma~\ref{lem:tildeH-approx-dupe}]
  Let $p\in \Gamma_\alpha$ and let $P_p$ be the tangent plane to $\Gamma_\alpha$ at $p$. Let $\lambda$ be the affine function such that $\Gamma_\lambda = P_p$, and let $q\in P_p$ be such that $d_\Kor(p,q)\le \rho^{\epsilon}$.
  
  Let $\kappa=d_{\Kor}(p,q)$. By Lemma~\ref{lem:a-Taylor} and Lemma~\ref{lem:gamma-deriv-bounds}, $\overline{q} = q Y^t$, where $t = \alpha(q) - \lambda(q) = O(\rho^{-1} \kappa^2)$. We choose $\rho$ large enough that $|t|<1$. Let $P_{\overline{q}}$ be the tangent plane to $\Gamma_\alpha$ at $\overline{q}$ and let $Q=Y^{-t} P_{\overline{q}}$ be the plane through $q$ parallel to $P_{\overline{q}}$.
  Then by the triangle inequality,
  \begin{multline*}
    \big|G_{\alpha, \gamma}'(0)(q) - H_{P_p, \gamma}(q)\big| \le
    \big|G_{\alpha, \gamma}'(0)(q) - H_{P_{\overline{q}},\gamma}(\overline{q})\big| \\
    + \big|H_{P_{\overline{q}},\gamma}(\overline{q}) - H_{Q, \gamma}(q)\big| + \big|H_{Q,\gamma}(q) - H_{P_p,\gamma}(q)\big|.
  \end{multline*}
  By Lemma~\ref{lem:flat-pv-formula},
  $$\big|G_{\alpha, \gamma}'(0)(q) - H_{P_{\overline{q}},\gamma}(\overline{q})\big|\lesssim \rho^{-\frac{1}{2}}.$$
  Since $P_{\overline{q}}$ and $Q$ are parallel and $d_\Kor(q, \overline{q})\le t < 1$, Lemma~\ref{lem:moving-basepoint} implies that
  $$\big|H_{P_{\overline{q}},\gamma}(\overline{q}) - H_{Q, \gamma}(q)\big| \lesssim d_\Kor(q,\overline{q})^{\frac{1}{5}} \lesssim (\rho^{-1} \kappa^2)^{\frac{1}{5}}.$$
  Finally, by Lemma~\ref{lem:changing-planes} and Lemma~\ref{lem:a-Taylor},
  $$\big|H_{Q,\gamma}(q) - H_{P_p,\gamma}(q)\big| \lesssim |\nabla_\alpha \alpha(q) - \nabla_\alpha \alpha(p)| \lesssim \kappa \rho^{-1} + \kappa^2 \rho^{-1}.$$

  Since $\kappa \le \rho^{\frac{1}{10}}$, these bounds imply that
  $$\big|G_{\alpha, \gamma}'(0)(q) - H_{P_p, \gamma}(q)\big| \lesssim \rho^{-\frac{1}{2}} + (\rho^{-1} \kappa^2)^{\frac{1}{5}} + \kappa \rho^{-1} + \kappa^2 \rho^{-1} \lesssim \rho^{-\frac{1}{10}},$$
  as desired.
  
  To prove \eqref{eq:H-unif}, we apply Lemma~\ref{lem:G'-formula-and-bound-rev} with $\zeta=\lambda$, $\psi=\gamma$, and $r=1$. Let $q\in P_p$ such that $\kappa=d_\Kor(p,q)\le \rho^\epsilon$ as above. Note that 
  $$d_\Kor(q,\Gamma_\alpha)\le d_\Kor(q,\overline{q})\lesssim \rho^{-1}\kappa^2\lesssim 1.$$
  
  Since $|\slope(P_p)|\le 1$, we take $B=B(q,3)$. Since $\lambda$ is affine, $\|\lambda\|_{W'_{\lambda}(B)} = 0$. For any $v\in B$, we have $d_\Kor(v,\Gamma_\alpha)\lesssim 3 + d_\Kor(q,\Gamma_\alpha) \lesssim 1$, so Lemma~\ref{lem:horiz-gamma-bounds} implies that  $\|\gamma\|_{W_{\lambda}(B)} \lesssim A^{-1}$. By Lemma~\ref{lem:G'-formula-and-bound-rev} and Lemma~\ref{lem:grrab-small},
  $$\big|H_{P_p,\gamma}(q) - (G^{r,R}_{\lambda,\gamma})'(0)(q)\big|\lesssim \|\gamma\|_{W_{\lambda}(B)} r + \|\gamma\|_{\infty} R^{-1}\lesssim A^{-1}(r+R^{-1}),$$
  as desired.
\end{proof}

\begin{proof}[Proof of Lemma~\ref{lem:Fprime-Lipschitz-dupe}]
  We claim that there is an $\epsilon>0$ such that for all $p, q\in \Gamma_\alpha$,
  $$\big|G_{\alpha, \gamma}'(0)(p) - G_{\alpha, \gamma}'(0)(q)\big| \lesssim d_{\Kor}(p,q)^{\epsilon} + \rho^{-\frac{1}{2}}.$$
 
  Let $r=d_\Kor(p,q)$. By Proposition~\ref{prop:perturb-formula}, $\|G_{\alpha,\gamma}'(0)\|_\infty\lesssim 1$, so it suffices to consider the case that $r\le 1$. 
  
  By Lemma~\ref{lem:flat-pv-formula}, we have
  \begin{equation}\label{eq:Fprime-eq-1}
  \big|G_{\alpha, \gamma}'(0)(p) - G_{\alpha, \gamma}'(0)(q)\big| \lesssim \big|H_{P_p, \gamma}(0)(p) - H_{P_q, \gamma}(0)(q)\big| + \rho^{-\frac{1}{2}}.
  \end{equation}
  Let $Q$ be the plane parallel to $P_p$ that goes through $q$. Lemma~\ref{lem:changing-planes} and Lemma~\ref{lem:moving-basepoint} imply that there is an $\epsilon>0$ such that 
  \begin{align}
\notag    \big|H_{P_p, \gamma}(0)(p) - H_{P_q, \gamma}(0)(q)\big|   
    &\le \big|H_{P_p, \gamma}(0)(p) - H_{Q, \gamma}(0)(q)\big| + \big|H_{Q, \gamma}(0)(q) - H_{P_q, \gamma}(0)(q)\big| 
    \\ 
\notag    &\lesssim r^{\epsilon} + |\slope Q-\slope P_q| \\ 
\label{eq:Fprime-eq-2}    &= r^{\epsilon} + |\nabla_\alpha\alpha(p) - \nabla_\alpha\alpha(q)|.
  \end{align}
  By Lemma~\ref{lem:a-Taylor} with $a=\alpha$ and $m=\nabla_\alpha$ and by Lemma~\ref{lem:gamma-deriv-bounds},
  $$|\nabla_\alpha\alpha(p) - \nabla_\alpha\alpha(q)| \lesssim r \|\nabla^2_\alpha \alpha\|_{\infty} + r^2 \|\partial_z \nabla_\alpha \alpha\|_{\infty} \lesssim r \rho^{-1}\lesssim r.$$
  Combining this with \eqref{eq:Fprime-eq-1} and \eqref{eq:Fprime-eq-2}, we see that
  $$\big|G_{\alpha, \gamma}'(0)(p) - G_{\alpha, \gamma}'(0)(q)\big| \lesssim r^{\epsilon} + r + \rho^{-\frac{1}{2}}\lesssim r^\epsilon + \rho^{-\frac{1}{2}},$$
  as desired.
\end{proof}

\subsection{Proofs of Lemmas~\ref{lem:horiz-gamma-bounds}--\ref{lem:moving-basepoint}}

Now we prove the lemmas that we used in the proofs of Lemmas~\ref{lem:tildeH-approx-dupe} and \ref{lem:Fprime-Lipschitz-dupe}. First, we prove Lemma~\ref{lem:horiz-gamma-bounds}, which bounds derivatives of $\gamma$ near $\Gamma_\alpha$.

\begin{proof}[Proof of Lemma~\ref{lem:horiz-gamma-bounds}]
  Recall that $\Pi_P = \Pi_{\lambda_P}$ is the projection to $P$ and that $\nabla_P(v) = \nabla_{\lambda_P}(v) = X(v) + (y(v) - \lambda_P(v)) Z(v)$. Since $\nabla_P$ is constant on vertical lines, we have $[\nabla_P,Z]=0$. 
  
  Since $W=X+\slope(P) Y$ is horizontal, for any $u\in P$, the curve $g(w) = uW^w$ is a horizontal curve in $P$, so its projection $\Pi\circ g$ is an integral curve of $\nabla_P$. For any function $a$ which is constant on cosets of $\paramY$,
  $$W^i a(u) = (a\circ g)^{(i)}(0) = (a\circ \Pi \circ g)^{(i)}(0) = \nabla_P^ia(\Pi(u))=\nabla_P^i a(u).$$
  Therefore, for any $i$ and $j$ and any $p\in P$,
  $$W^i Z^j[\gamma](p) = \nabla_{P}^i Z^j[\gamma](p).$$
  This proves the first equality in \eqref{eq:horiz-gamma}.
  
  We claim that for any $p\in \HH$ and any $i\ge 0$, $j\ge 0$ with $i+j\le 3$, we have
  \begin{equation}\label{eq:Lambda-bound}
    |\nabla_P^i Z^j \gamma(p)|  \lesssim (1 + |\lambda_P(p)-\alpha(p)|)^i.
  \end{equation}
  Since $\alpha$ is intrinsic Lipschitz, $|\lambda_P(p)-\alpha(p)|\lesssim d_\Kor(p,\Gamma_\alpha)$ for all $p\in P$, so this will imply the lemma.


  
  
  Let $h\from \HH\to \R$ be a smooth function that is constant on cosets of $\paramY$. Let $\Lambda(v)=1+ |\lambda_P(v)-\alpha(v)|$. For $c>0$, $d\ge 0$ and $n\ge 0$, we say that $h$ has \emph{$(c, d,n)$--derivative growth} if for any word $E\in \{Z,\nabla_\alpha\}^*$ of length at most $d$ and any $q\in \HH$, we have
  $$|E h(q)| \le c A^{-1} \Lambda(q)^n.$$
  In particular, $|h(q)| \lesssim A^{-1} \Lambda(q)^n$. We claim that $\nabla_P^j Z^i\gamma$ has $(c_{j}, 3-i-j, j)$--derivative growth when $i+j\le 3$. This will imply \eqref{eq:Lambda-bound}.
  
  When $j=0$, Lemma~\ref{lem:gamma-deriv-bounds} implies that $\gamma$ has $(c_0, 3, 0)$--derivative growth and $Z^i\gamma$ has $(c_0, 3-i, 0)$--derivative growth for some $c_0\lesssim A^{-1}$. 
  
  We thus proceed by induction. Suppose that $h$ has $(c_n, d, n)$--derivative growth for some $d\le 3$ and consider $\nabla_P h$. Note that 
  $$\nabla_P h = \nabla_\alpha h + (\alpha - \lambda_P) Z h.$$
  For any $0\le l\le d-1$, any $E=E_1\dots E_l\in \{Z,\nabla_\alpha\}^*$, and any $q\in \HH$,
  \begin{multline*}
    |E \nabla_P h(q)| \le |E \nabla_\alpha h(q)| + |E[(\alpha - \lambda_P) Z h](q)|\\
    \le c_n \Lambda(q)^n + \sum_{S\subset \{1,\dots, l\}} |E_S[\alpha - \lambda_P](q) \cdot E_{S^c} Z h(q)|,
  \end{multline*}
  where $E_S=\prod_{i\in S} E_i$. By Lemma~\ref{lem:gamma-deriv-bounds}, $|E_S[\alpha-\lambda_P](v)|\lesssim_C 1$ unless $S=\emptyset$ and $E_S=\id$. Furthermore, $E_{S^c} Z$ is a word of length at most $d$, so $|E_{S^c} Z h(q)|\le c_n \Lambda(q)^n$. Therefore,
  \begin{equation*}
    |E \nabla_P h(q)| \lesssim_C c_n \Lambda(q)^n + 2^l c_n \Lambda(q)^n + |\lambda_P(q) - \alpha(q)| \cdot c_n \Lambda(q)^n
    \lesssim_{C,d} c_n \Lambda(q)^{n+1}.
  \end{equation*}
  That is, $\nabla_P h$ has $(c_{n+1}, d-1,n+1)$--derivative growth for some $c_{n+1}\lesssim_{C,d} c_n$. 
  
  For $0\le i\le 3$, $Z^i\gamma$ has $(c_0, 3-i, 0)$--derivative growth for some $c_0\lesssim A^{-1}$, so for $0\le j\le 3-i$, there are $c_j\lesssim_{C} A^{-1}$ such that $\nabla_P^j Z^i \gamma$ has $(c_j, 3-i-j,j)$--derivative growth. In particular, for all $p\in P$,
  $$\big|\nabla_P^j Z^i \gamma(p)\big| \le c_j \Lambda(p)^{j} \lesssim_C A^{-1}(1 + |\lambda_P(p)-\alpha(p)|)^{j},$$
  as desired.
\end{proof}

\begin{proof}[Proof of Lemma~\ref{lem:flat-pv-formula}]
  Let $p\in \Gamma_\alpha$ and let $P$ be the vertical tangent plane to $\Gamma_\alpha$ at $p$. 
  We claim that
  $$G_{\alpha,\gamma}'(0)(p) = \frac{\ud}{\ud t} F_{\alpha + t\gamma}(p) \big|_{t=0} = H_{P,\gamma}(p) + O(\rho^{-\frac{1}{2}}).$$
  After translating, we may suppose that $p=\zero$. Let $\gamma_0(w)=\gamma(w)-\gamma(\zero)$ for all $w\in \HH$. By Proposition~\ref{prop:perturb-formula} and \eqref{eq:H-as-pv-v0}, we can write
  \begin{multline}\label{eq:flat-pv-decomp}
    \left|G_{\alpha,\gamma}'(0)(\zero) - H_{P,\gamma}(\zero)\right|\le \limsup_{r\to 0} \left|\left(G^{r,\sqrt{\rho}}_{\alpha,\gamma}\right)'(0)(\zero) - \int_{A_{r,\sqrt{\rho}}} \gamma_0(w) \YL \widehat{K}(\Pi_P(w)) \ud w \right| \\
    + \limsup_{R\to \infty} \left|\left(G^{\sqrt{\rho},R}_{\alpha,\gamma}\right)'(0)(\zero) - \int_{A_{\sqrt{\rho},R}} \gamma_0(w) \YL \widehat{K}(\Pi_P(w)) \ud w \right|.
  \end{multline}

  We start by bounding the $R\to \infty$ term.
  Since $P$ has bounded slope, we have $\|\Pi_P(w)\|_\Kor \approx \|w\|_\Kor$.
  By the homogeneity of $\YL \widehat{K}$ and the boundedness of $\gamma$,
  $$\left|\int_{A_{\sqrt{\rho},R}} \gamma_0(w) \YL \widehat{K}(\Pi_P(w)) \ud w\right|\lesssim \left|\int_{A_{\sqrt{\rho},R}} \|w\|_\Kor^{-4}\ud w \right| \approx \int_{\sqrt{\rho}}^\infty \kappa^{-4}\kappa^2\ud \kappa = \rho^{-\frac{1}{2}},$$
  using Lemma~\ref{lem:polar} to change variables from $w$ to $\kappa$.
  By Lemma~\ref{lem:grrab-large}, for all $R>\sqrt{\rho}$,
  \begin{equation}\label{eq:flat-pv-decomp-1}
    \left|(G^{\sqrt{\rho},R}_{\alpha,\gamma})'(0)(\zero) - \int_{A_{\sqrt{\rho},R}} \gamma_0(w) \YL \widehat{K}(\Pi_P(w)) \ud w \right| \lesssim \rho^{-\frac{1}{2}} + \rho^{-\frac{1}{2}} \approx \rho^{-\frac{1}{2}}.
  \end{equation}

  Now we consider the $r\to 0$ term.
  By \eqref{eq:grrab-small}, letting $w=(x,0,z)$, for any $0<r<\sqrt{\rho}$,
  \begin{multline}\label{eq:flat-pv-approx}
    (G^{r,\sqrt{\rho}}_{\alpha,\gamma})'(0)(\zero) - \int_{A_{r,\sqrt{\rho}}} \gamma_0(w) \YL \widehat{K}(\Pi_P(w)) \ud w \\
    = \int_{A_{r,\sqrt{\rho}}} \gamma_0(w)\left(\YL \widehat{K}(\Psi_\alpha(w)) - \YL \widehat{K}(\Pi_P(w))\right) - \gamma(\zero) x \partial_z \alpha(w) \YL\widehat{K}(\Psi_\alpha(w))\ud w,
  \end{multline}
  and we will bound the terms in the integrand separately.

  We start with the first term in \eqref{eq:flat-pv-approx}. Let $w\in V_0$ and let $\kappa=\|w\|_\Kor$.
  Let $\lambda=\lambda_P$, so that $\Gamma_{\lambda}=P$.
  By Lemma~\ref{lem:a-Taylor},
  \begin{align}
    |\alpha(w) - \lambda(w)|\lesssim \rho^{-1} \kappa^2
    \label{e:small-bound}
  \end{align}
  By the Mean Value Theorem, there is some $t$ between $\lambda_P(w)$ and $\alpha(w)$ such that 
  \begin{align*}
    \left| \YL \widehat{K}(w Y^{\alpha(w)}) - \YL \widehat{K}(w Y^{\lambda_P(w)}) \right| = |\alpha(w) - \lambda_P(w)|\cdot |\YL^2 \widehat{K}(w Y^{t})|.
  \end{align*} 
  Furthermore, $\|w Y^t\|_\Kor \approx \|w\|_\Kor$, so by the $(-5)$--homogeneity of $\YL^2 \widehat{K}$,
  \begin{equation}\label{e:YK-linearization}
    \left| \YL \widehat{K}(w Y^{\alpha(w)}) - \YL \widehat{K}(w Y^{\lambda_P(w)}) \right| \overset{\eqref{e:small-bound}}{\lesssim} \rho^{-1}\kappa^2\cdot \kappa^{-5} = \rho^{-1}\kappa^{-3}.
  \end{equation}

  We apply Lemma~\ref{lem:a-Taylor} to $\gamma_0$ to get
  $$|\gamma_0(w)|\lesssim \kappa \|\nabla_\alpha \gamma_0\|_{\infty} + \kappa^2 \|\partial_z \gamma_0\|_{\infty}\lesssim \kappa + \kappa^2$$
  by Lemma~\ref{lem:gamma-deriv-bounds}. Since $\|\gamma_0\|_\infty\lesssim 1$, we have $|\gamma_0(w)| \lesssim \min\{\kappa+\kappa^2,1\}\lesssim \kappa$.
  Therefore,
  \begin{multline}\label{eq:flat-pv-large}
    \left|\int_{A_{r,\sqrt{\rho}}} \gamma_0(w)(\YL \widehat{K}(\Psi_\alpha(w)) - \YL \widehat{K}(\Pi_P(w))) \ud w\right| \\ \lesssim \int_{A_{r,\sqrt{\rho}}} \|w\|_\Kor \cdot \rho^{-1} \|w\|_\Kor^{-3}\ud w \lesssim \int_r^{\sqrt{\rho}}\kappa^{-2} \rho^{-1} \kappa^2 \ud \kappa 
    \le \rho^{-\frac{1}{2}},
  \end{multline}
  using Lemma~\ref{lem:polar} in the penultimate inequality.

  It remains to bound the second term in \eqref{eq:flat-pv-approx}.
  We write
  \begin{align*}
    \int_{A_{r,\sqrt{\rho}}} x \partial_z \alpha(w) \YL\widehat{K}(\Psi_\alpha(w))\ud w 
    &= \int_{A_{r,\sqrt{\rho}}} x (\partial_z \alpha(w)- \partial_z\alpha(\zero)) \YL\widehat{K}(\Psi_\alpha(w))\ud w \\
    & \qquad + \int_{A_{r,\sqrt{\rho}}} x \partial_z\alpha(\zero) (\YL\widehat{K}(\Psi_\alpha(w)) - \YL\widehat{K}(\Pi_P(w))) \ud w\\
    & \qquad + \int_{A_{r,\sqrt{\rho}}} x \partial_z\alpha(\zero) \cdot \YL\widehat{K}(\Pi_P(w)) \ud w \\
    &=: I_1 + I_2 + I_3.
  \end{align*}

  To bound $I_1$, let $m = \partial_z \alpha$. By Lemma~\ref{lem:gamma-deriv-bounds}, we have $\|m\|_\infty \lesssim \rho^{-1}$, $\|\nabla_\alpha m\|_\infty \lesssim \rho^{-1}$, and $\|\partial_z m\|_\infty \lesssim \rho^{-1}$, so by Lemma~\ref{lem:a-Taylor}, letting $\kappa=\|w\|_\Kor$ as above,
  \begin{align*}
    |m(w) - m(\zero)|\lesssim \min\{\rho^{-1} , \rho^{-1}(\kappa + \kappa^2)\} \lesssim \rho^{-1} \kappa.
  \end{align*}
  Therefore, using Lemma~\ref{lem:polar},
  \begin{multline*}
    |I_1| \lesssim \int_{A_{r,\sqrt{\rho}}}\left| x \rho^{-1}\|w\|_\Kor \YL\widehat{K}(\Psi_\alpha(w))\right|\ud w \\ \lesssim \rho^{-1} \int_{A_{r,\sqrt{\rho}}}\|w\|_\Kor^{-2} \ud w \lesssim \rho^{-1} \int_r^{\sqrt{\rho}}\kappa^{-2}\kappa^2\ud \kappa 
    \le \rho^{-\frac{1}{2}}.
  \end{multline*}
  
  By \eqref{e:YK-linearization}
  $$|I_2| \lesssim \int_{A_{r,\sqrt{\rho}}} \left|x \partial_z \alpha(\zero) \rho^{-1} \|w\|_\Kor^{-3}\right| \lesssim \int_{A_{r,\sqrt{\rho}}} \rho^{-2} \|w\|_\Kor^{-2}\ud w \lesssim \rho^{-2} \int_r^{\sqrt{\rho}}\ud \kappa \le \rho^{-\frac{3}{2}}.$$

  Finally, recall that $\theta(x,y,z)=(-x,-y,z)$ and let $h(w) = x(w) \partial_z\alpha(\zero) \cdot  \YL\widehat{K}(\Pi_P(w))$. The symmetry of $\YL\widehat{K}$ implies that
  $$h(\theta(w)) = x(\theta(w)) \partial_z\alpha(\zero)\cdot  \YL\widehat{K}(\Pi_P(\theta(w))) = -x(w) \partial_z\alpha(\zero)  \cdot \YL\widehat{K}(\Pi_P(w)) = -h(w).$$
  Since $\theta(A_{r,\sqrt{\rho}}) =A_{r,\sqrt{\rho}}$,
  $$I_3 = \int_{A_{r,\sqrt{\rho}}} x \partial_z\alpha(\zero)  \YL\widehat{K}(\Pi_P(w)) \ud w = 0,$$
  and
  \begin{equation}\label{eq:flat-pv-small}
    \int_{A_{r,\sqrt{\rho}}} x \partial_z \alpha(w) \YL\widehat{K}(\Psi_\alpha(w))\ud w = I_1 + I_2 + I_3 = O(\rho^{-\frac{1}{2}}).
  \end{equation}

  Combining these inequalities, we find that
  $$\left|G_{\alpha,\gamma}'(0)(\zero) - H_{P,\gamma}(\zero)\right|\lesssim \rho^{-\frac{1}{2}},$$
  as desired.
\end{proof}

\begin{proof}[Proof of Lemma~\ref{lem:changing-planes}]

  Without loss of generality, we may suppose that $p=\zero$. We claim  that
  $$\left|\pv \int_{V_0} (\gamma(q) - \gamma(\zero)) \left(\YL\widehat{K}(\Pi_P(q)) - \YL\widehat{K}(\Pi_Q(q))\right)\ud q\right| \lesssim_C |\slope P-\slope Q|.$$
  Let $M(q)=\YL\widehat{K}(\Pi_P(q)) - \YL\widehat{K}(\Pi_Q(q))$.
  By the smoothness and homogeneity of $\YL\widehat{K}$, for all $q\in V_0$, we have
  $$|M(q)| \lesssim_C |\slope P-\slope Q| \|q\|_\Kor^{-4}$$
  and $M(\theta(q)) = M(q)$. Let $\gamma_0(q)=\gamma(q) - \gamma(\zero)$, and let $\gamma_0^{\sE}(q) = \frac{1}{2}(\gamma_0(q) + \gamma_0(\theta(q)))$ so that
  $$\int_{A_{r,R}} (\gamma(q) - \gamma(\zero)) M(q) \ud q = \int_{A_{r,R}} \gamma_0^\sE(q) M(q) \ud q.$$

  On one hand, since $\|\gamma\|_\infty\lesssim 1$,  Lemma~\ref{lem:polar} implies that for any $R>1$,
  $$\left|\int_{A_{1,R}} \gamma_0^\sE(q) M(q) \ud q\right| \lesssim |\slope P-\slope Q| \|\gamma\|_\infty \int_1^R \kappa^{-4}\kappa^2\ud \kappa \lesssim |\slope P-\slope Q|.$$
  
  On the other hand, by Lemma~\ref{lem:gamma-deriv-bounds} and Lemma~\ref{lem:a-Taylor},
  $$\gamma_0(q) = x(q) \nabla_\alpha \gamma(q) + O(\|q\|_\Kor^2),$$
  so $|\gamma_0^\sE(q)|\lesssim \|q\|_\Kor^2$. Therefore, for any $r<1$, 
  $$\left|\int_{A_{r,1}} \gamma_0^\sE(q) M(q) \ud q\right| \lesssim |\slope P-\slope Q| \int_r^1 \kappa^{-2}\kappa^2\ud \kappa \lesssim |\slope P-\slope Q|.$$
  
  Combining these two inequalities, we have
  \begin{multline*}
    \left|\pv \int_{V_0} \gamma_0^\sE(q) M(q) \ud q\right| \le \lim_{r\to 0} \left|\int_{A_{r,1}} \gamma_0^\sE(q) M(q) \ud q\right| + \lim_{R\to \infty} \left|\int_{A_{1,R}} \gamma_0^\sE(q) M(q) \ud q\right| \\
    \lesssim_C |\slope P-\slope Q|.
  \end{multline*}
\end{proof}

Finally, we prove Lemma~\ref{lem:moving-basepoint}. We will need the following bound.
\begin{lemma} \label{l:delta-1-bound} 
  Let $C>0$, let $P$ be a plane through the origin with $|\slope(P)|\le C$ and let $\psi \from \HH\to \R$ be a bounded smooth function such that $\psi(\zero)=0$ and $\psi$ is constant on cosets of $\paramY$. Let $r>0$, and let
  $$\epsilon=\|\partial_z \psi\|_{L_\infty(B(\zero,r))} + \|X^2 \psi\|_{L_\infty(V_0\cap B(\zero,r))}.$$
  Then 
  \begin{align*}
    \left| \pv \int_{V_0} \psi(v) \YL \widehat{K}(\Pi_P(v)) \ud v \right| \lesssim_C \epsilon r + \|\psi\|_\infty r^{-1}.
  \end{align*}
\end{lemma}

\begin{proof}
  Let $v\in V_0\cap B(\zero,r)$ and let $\kappa=\|v\|_\Kor$.
  By Taylor approximation, there is a $c>0$ such that 
  \begin{align*}
    |\psi(v) - X \psi(\zero) x(v)|\le c \epsilon \kappa^2.
  \end{align*}
  Let $D=V_0\cap B(\zero, r)$. Then
  \begin{align*}
    &\left| \pv \int_{D} \psi(v) \YL \widehat{K}(\Pi_P(v)) \ud v \right| \\
    &\qquad \leq \left| \pv \int_{D} X \psi(\zero) x(v) \YL \widehat{K}(\Pi_P(v)) \ud v \right| + \pv \int_{D} \left| c \epsilon  \|v\|_\Kor^2 \YL \widehat{K}(\Pi_P(v)) \right| \ud v.
  \end{align*}
  The first term is $0$ by symmetry, and since $\|v\|_\Kor^2 \YL \widehat{K}(v)$ is $(-2)$--homogeneous, Lemma~\ref{lem:polar} implies that
  $$\left| \pv \int_{D} \psi(v) \YL \widehat{K}(\Pi_P(v)) \ud v \right|\lesssim_C c\epsilon \int_0^{r} \kappa^{-2} \kappa^2 \ud \kappa \lesssim \epsilon r.$$

  Additionally, by the homogeneity of $\YR K$ and Lemma~\ref{lem:polar},
  \begin{align*}
    \left| \pv \int_{V_0\setminus D} \psi(v) \YL \widehat{K}(\Pi_P(v)) \ud v \right| \lesssim \int_{r}^\infty \|\psi\|_\infty \kappa^{-4} \kappa^2\ud \kappa \lesssim \|\psi\|_\infty r^{-1}.
  \end{align*}
  Summing these two inequalities proves the lemma.
\end{proof}

We now prove Lemma~\ref{lem:moving-basepoint}.

\begin{proof}[Proof of Lemma~\ref{lem:moving-basepoint}]
  Recall that $p\in \Gamma_\alpha$, $q\in \HH$, and $d_\Kor(p,q)\le 1$. After a translation, we may suppose that $p=\zero$.  
  Then for any $h\in \HH$,
  \begin{align*}
    H_{h P,\gamma}(h) 
    & = \pv(\zero) \int_{V_0} (\gamma(h)-\gamma(hv))\YL \widehat{K}(\Pi_P(v))\ud \mu,
  \end{align*}
  so, letting 
  $$\nu(v) = \gamma(\zero) - \gamma(v) - \gamma(q) + \gamma(qv),$$
  \begin{align*}
    H_{P,\gamma}(\zero) - H_{q P, \gamma}(q) 
    & = \pv(\zero) \int_{V_0} \nu(v) \YL \widehat{K}(\Pi_P(v))\ud \mu.
  \end{align*}
  Then $\|\nu\|_{\infty}\le 4\|\gamma\|_{\infty}\lesssim 1$.
  We abbreviate partial derivatives of $\nu(x,y,z)$ and $\gamma(x,y,z)$ as $\gamma_x=\partial_x \gamma$, $\gamma_{xz}=\partial_x\partial_z \gamma$, etc. 
  
  Let $\kappa=d_\Kor(\zero,q)\le 1$ and let $r=\kappa^{-\frac{1}{5}}$. We claim that
  $|\nu_z(v)|\lesssim \kappa^{\frac{4}{5}}$ and $|\nu_{xx}(v)|\lesssim \kappa^{\frac{2}{5}}$ for all $v\in V_0\cap B(\zero,r)$. This will let us apply Lemma~\ref{l:delta-1-bound} to $\nu$.
  
  We write $q=X^{x_0}Z^{z_0}Y^t$; note that $|x_0|\le \kappa$, $|t| \le \kappa$, and $|z_0|\le \kappa^2$.
  Since $\gamma$ is constant on cosets of $\paramY$,
  \begin{multline*}
    \nu(x,0,z) = \nu(X^x Z^z) = \gamma(\zero) - \gamma(X^x Z^z) - \gamma(q Y^{-t}) + \gamma(q Y^{-t} \cdot Y^{t} X^{x} Y^{-t} Z^z)\\
    = \gamma(0, 0, 0) - \gamma(x, 0, z) - \gamma(x_0, 0, z_0) + \gamma(x_0 + x, 0, z_0 - t x + z).
  \end{multline*}
    
  Suppose that $v=(x,0,z)\in V_0\cap B(\zero,r)$. Then
  \begin{align*}
    \nu_z(v) & = \gamma_z(v') - \gamma_z(v) \\
    \nu_{xx}(v) &= \gamma_{xx}(v') - 2 t \gamma_{xz}(v') + t^2 \gamma_{zz}(v') - \gamma_{xx}(v),
  \end{align*}
  where $v'=(x_0 + x, 0, z_0 - t x + z)$. 
  
  Note that $|x(v')| \le r + \kappa \lesssim r$ and $|z(v')| \le \kappa^2 + r^2 + \kappa r \lesssim r^2$, so there is a $c>0$ such that $v, v'\in B(\zero, cr)$. Let $S=V_0\cap B(\zero,cr)$. By Lemma~\ref{lem:horiz-gamma-bounds}, for $i+j\le 3$,
  \begin{equation}\label{eq:h-gamma-app}
    \|\partial_x^i \partial_z^j \gamma\|_{L_\infty(S)} \lesssim 1+r^i.
  \end{equation}
  In particular, $|t \gamma_{xz}(v')|\lesssim \kappa(1 + r)\lesssim \kappa^{\frac{4}{5}}$ and $|t^2 \gamma_{zz}(v')| \lesssim \kappa^2$.
  
  It remains to bound $|\gamma_{xx}(v') - \gamma_{xx}(v)|$ and $|\gamma_{z}(v') - \gamma_{z}(v)|$. Since $v^{-1}v'=(x_0,0,z_0-tx)$, the Mean Value Theorem and \eqref{eq:h-gamma-app} imply that
  \begin{multline*}
    |\gamma_{xx}(v') - \gamma_{xx}(v)| \le |x_0| \|\gamma_{xxx}\|_{L_\infty(S)} + |z_0-tx|\|\gamma_{xxz}\|_{L_\infty(S)}\\ 
    \lesssim \kappa(1+r^3) + (\kappa^2 + \kappa r)(1+r^2) \lesssim \kappa r^3 = \kappa^{\frac{2}{5}}.
  \end{multline*}
  Likewise,
  \begin{multline*}
    |\gamma_{z}(v') - \gamma_{z}(v)| \le |x_0| \|\gamma_{xz}\|_{L_\infty(S)} + |z_0-tx|\|\gamma_{zz}\|_{L_\infty(S)}\\ 
    \lesssim \kappa(1+r) + (\kappa^2 + \kappa r) \lesssim \kappa r = \kappa^{\frac{4}{5}}.
  \end{multline*}
  
  Combining these inequalities, we obtain $|\nu_z(v)|\lesssim \kappa^{\frac{4}{5}}$ and $|\nu_{xx}(v)|\lesssim \kappa^{\frac{2}{5}}$ for all $v\in V_0\cap B(\zero,r)$. By Lemma~\ref{l:delta-1-bound}, this implies
  \begin{multline*}
    |H_{P,\gamma}(\zero) - H_{q P, \gamma}(q)| 
    = \left|\pv(\zero) \int_{P} \nu(v) \YL \widehat{K}(\Pi_P(v))\ud \mu\right| \\
    \lesssim_C r(\kappa^{\frac{4}{5}} + \kappa^{\frac{2}{5}}) + r^{-1}\|\nu\|_\infty
    \lesssim \kappa^{\frac{1}{5}} = d_{\Kor}(\zero,q)^{\frac{1}{5}},
  \end{multline*}
  as desired.
\end{proof}

\section{Lower bounds on the first derivative}\label{sec:first-lower}

Now we use the approximations in the previous section to prove lower bounds on $G'_{\alpha,\gamma}(0)$. Our main estimate is the following lemma, which shows that we can estimate $G'_{\alpha,\gamma}(0)(p)$ in terms of the restriction of $\gamma$ to the vertical line $p\langle Z\rangle$. 

Specifically, let $P$ be the vertical tangent plane to $\Gamma_\alpha$ at $p$, i.e., $P = p \langle W, Z\rangle$, where $W=X+\nabla_\alpha \alpha(p) Y$. Let $\Pi_P\from \HH\to P$, $\Pi_P(p W^wZ^zY^y)=pW^wZ^z$ be the projection to $P$ along cosets of $Y$. Let $\pi_p\from \HH\to p\langle Z\rangle$, $\pi_p(p W^wZ^zY^y)=pZ^z$. This map is constant along cosets of $\paramY$ and projects $P$ to $p\langle Z\rangle$ along cosets of $\langle W\rangle$. We will show the following bound.

\begin{lemma}\label{lem:first-z}
  Let $A>1$. When $\rho$ is sufficiently large (depending only on $A$), the following holds.
  Let $\alpha$ and $\gamma$ be functions satisfying Lemma~\ref{lem:gamma-deriv-bounds}.
  Let $p\in \Gamma_\alpha$ and let $P$ be the vertical tangent plane to $\Gamma_\alpha$ at $p$. Let $\pi_p$ be as above. Then
  \begin{multline}\label{eq:first-z} 
    G'_{\alpha,\gamma}(0)(p) = H_{P,\gamma\circ \pi_p}(p) + O(A^{-2}) \\ 
    =\pv(p) \int_{pV_0} (\gamma(\pi_p(q)) - \gamma(p)) \YL \widehat{K}(p^{-1}\Pi_P(q)) \ud q + O(A^{-2}).
  \end{multline}
\end{lemma}
\begin{proof}
  The second equality in \eqref{eq:first-z} is \eqref{eq:H-as-pv-v0}, so it suffices to prove the first equality. After a left-translation, we may suppose that $p=\zero$. 

  By Lemma~\ref{lem:flat-pv-formula}, 
  \begin{multline}\label{eq:first-z-goal}
    |G'_{\alpha,\gamma}(0)(\zero)-H_{P,\gamma\circ \pi_p}(\zero)| \le |G'_{\alpha,\gamma}(0)(\zero)-H_{P,\gamma}(\zero)| + |H_{P,\gamma}(\zero)-H_{P,\gamma\circ \pi_p}(\zero)| \\
    \lesssim \rho^{-\frac{1}{2}} + |H_{P,\gamma-\gamma\circ \pi_p}(\zero)|.
  \end{multline}
  We thus consider $H_{P,\gamma-\gamma\circ \pi_p}(\zero)$.
  Recall that by \eqref{eq:H-as-pv-v0},
  $$H_{P,\gamma-\gamma\circ \pi_p}(\zero) = \pv(\zero)\int_{V_0} (\gamma(q)-\gamma\circ \pi_p(q))\YL\widehat{K}(\Pi_P(q)) \ud q.$$

  Since $\|\gamma\|_\infty \lesssim A^{-1}$, Lemma~\ref{lem:polar} implies that
  \begin{equation}\label{eq:first-z-large}
    \int_{V_0 \setminus B(\zero,A)} |(\gamma(q)-\gamma\circ \pi_p(q))\YL \widehat{K}(\Pi_P(q))| \ud q \lesssim \int_A^\infty A^{-1} \kappa^{-4} \kappa^2\ud \kappa \lesssim A^{-2}.
  \end{equation}

  Let $D=P\cap B(\zero,3A)$. We claim that if $v=W^wZ^z\in D$, then
  $$\gamma(v) = \gamma(Z^z) + w \cdot W \gamma(Z^z) + O(A^{-3}w^2).$$

  Let $\sigma=\nabla_\alpha \alpha(\zero)$ and let $\lambda\from \HH\to \R$, $\lambda(x,y,z)=\sigma x$ so that $\Gamma_\lambda=P$.
  Recall that for all $q\in \HH$,  we have $(\nabla_\alpha)_q = X_q + (y(q) - \alpha(q)) Z_q$. If $v\in P$, then $y(v)=\lambda(v)$, so
  $$W_v = (\nabla_\alpha)_v -(\lambda(v) - \alpha(v)) Z_v + \sigma Y_v.$$
  Let $m\from \HH\to \R$ be a smooth function which is constant on cosets of $\paramY$. Then $Ym=0$, so for $v\in P$, 
  $$Wm(v) = Xm(v) = \nabla_\alpha m(v) - (\lambda(v) - \alpha(v)) Zm(v) =(\nabla_\alpha - (\lambda - \alpha) Z)[m](v).$$
  We can apply this identity to $\alpha$, $\gamma$, and their derivatives with respect to $\nabla_\alpha$ and $Z$, which are all constant on cosets of $\paramY$.
  

  Note that by Lemma~\ref{lem:gamma-deriv-bounds} and Lemma~\ref{lem:a-Taylor},
  \begin{equation}\label{eq:misc-bounds}
    \lim_{\rho\to \infty} \max\{\|\lambda - \alpha\|_{L_\infty(D)}, \|\nabla_\alpha \alpha - \nabla_\alpha \alpha(\zero)\|_{L_\infty(D)}, \|Z\alpha\|_{\infty}\} = 0.
  \end{equation}
  One consequence of \eqref{eq:misc-bounds} is that for all $v\in D$,
  $$W\gamma(v) = \nabla_\alpha \gamma(v) + O\left(\|\lambda-\alpha\|_{L_\infty(D)} \|Z \gamma\|_\infty\right) \overset{\eqref{eq:misc-bounds}}{=} \nabla_\alpha \gamma(v) + o_\rho(1),$$
  where $o_\rho(1)$ is little--$o$ notation denoting an error term bounded by a function of $\rho$ that goes to zero as $\rho\to \infty$.

  We can bound the second derivative similarly. Evaluating all functions at $v\in D$,
  \begin{align*}
    W^2 \gamma & = W(\nabla_\alpha - (\lambda - \alpha) Z)[\gamma] = \nabla^2_\alpha \gamma - (\lambda - \alpha) Z \nabla_\alpha \gamma - W[\lambda - \alpha]\cdot Z\gamma - (\lambda - \alpha) W Z\gamma.
  \end{align*}
  By Lemma~\ref{lem:gamma-deriv-bounds}, $\|Z \nabla_\alpha \gamma\|_\infty \lesssim A^{-2}$, so by \eqref{eq:misc-bounds}, $(\lambda - \alpha) Z \nabla_\alpha \gamma=o_\rho(1)$. Likewise, $\|Z\gamma\|_\infty \lesssim A^{-1}$. Since $\nabla_\alpha \lambda(v) = \nabla_\alpha \alpha(\zero),$
  $$W[\lambda- \alpha](v) = \nabla_\alpha \alpha(\zero) - \nabla_\alpha \alpha(v) - (\lambda(v) - \alpha(v)) \cdot Z \alpha \overset{\eqref{eq:misc-bounds}}{=} o_\rho(1).$$
  Finally, 
  $$(\lambda(v)-\alpha(v)) W Z \gamma(v) = (\lambda(v)-\alpha(v)) \nabla_\alpha Z \gamma(v) - (\lambda(v)-\alpha(v))^2 Z^2\gamma(v) \overset{\eqref{eq:misc-bounds}}{=} o_\rho(1),$$
  so $W^2 \gamma(v) = \nabla_\alpha^2\gamma(v)+o_\rho(1)$. Thus, by Lemma~\ref{lem:gamma-deriv-bounds}, we can choose $\rho$ large enough that $|W \gamma(v)|\lesssim A^{-2}$ and $|W^2 \gamma(v)|\lesssim A^{-3}$ for all $v\in D$.
  
  By Taylor's theorem, if $v=W^wZ^z\in D$, then
  $$\gamma(v) = \gamma(Z^z) + w W\gamma(Z^z) + O(A^{-3}w^2).$$
  Let $\theta(x,y,z)=(-x,-y,z)$ and let $\gamma^\sE(v) = \frac{1}{2}(\gamma(v) + \gamma(\theta(v)))$. Then
  $$\gamma^\sE(v) = \gamma(Z^z) + O(A^{-3} w^2).$$
  If $q\in B(\zero,A)$, then $\Pi_P(q) = \pi_p(q)W^{x(q)}$ and $\Pi_P(q)\in D$, so $\gamma^\sE(q) = \gamma(\pi_p(q)) + O(A^{-3} x(q)^2).$
  By \eqref{eq:first-z-large},
  $$H_{P,\gamma-\gamma\circ \pi_p}(\zero) = \pv(\zero)\int_{V_0\cap B(\zero,A)} (\gamma(q)-\gamma(\pi_p(q)))\YL\widehat{K}(\Pi_P(q))\ud q + O(A^{-2}).$$
  Let $0<r<A$ and let $U=V_0\cap (B(\zero,A)\setminus B(\zero,r))$. Then by symmetry and the $(-4)$--homogeneity of $\YL\widehat{K}$,
  \begin{multline*}
    \left|\int_{U} (\gamma(q)-\gamma(\pi_p(q)))\YL\widehat{K}(\Pi_P(q))\ud q\right| = \left|\int_{U} (\gamma^\sE(q)-\gamma(\pi_p(q))) \YL\widehat{K}(\Pi_P(q))\ud q\right| \\
     \lesssim  \int_{U} A^{-3}x(q)^2\|q\|_\Kor^{-4}\ud q 
     \lesssim  A^{-3} \int_r^A \kappa^2 \kappa^{-4} \kappa^2\ud \kappa \le A^{-2},
  \end{multline*}
  where we use Lemma~\ref{lem:polar} to change variables from $q$ to $\kappa$.

  This holds for any $r$, so 
  $|H_{P,\gamma-\gamma\circ \pi_p}(\zero)|\lesssim A^{-2}$. The lemma then follows from \eqref{eq:first-z-goal}.
\end{proof}

Furthermore, we can write $H_{P,\gamma\circ \pi_p}(p)$ as a one-dimensional singular integral.

\begin{lemma}\label{lem:first-z-convert}
  With notation as above, for $z\ne 0$ and $a\in \R$, let
  \begin{equation}\label{eq:def-L}
    L(z) = L_{a}(z) := \int_{-\infty}^\infty \YL \widehat{K}((X+aY)^{x} Z^{z}) \ud x.
  \end{equation}
  For $p\in \Gamma_\alpha$ and $a=\nabla_\alpha\alpha(p)$,
  \begin{equation}\label{eq:first-z-convert-goal}
    H_{P,\gamma\circ \pi_p}(p) = \lim_{r\to 0} \int_{\R\setminus (-r,r)} (\gamma(p Z^{z}) - \gamma(\zero)) L_a(z) \ud z.
  \end{equation}
\end{lemma}
\begin{proof}
  After a left translation, we may suppose that $p=\zero$.
  Let $W=X+aY$ so that $P=\langle W, Z\rangle$. Since $\YL \widehat{K}$ is $(-4)$--homogeneous, the integral in \eqref{eq:def-L} converges absolutely. Note that for any $z\in \R$ and any $t>0$, 
  \begin{multline}\label{eq:L-homog}
    L(t^2z) = \int_{-\infty}^\infty \YL \widehat{K}(W^{w} Z^{t^2 z}) \ud w = \int_{-\infty}^\infty t \YL \widehat{K}(W^{tw} Z^{t^2 z}) \ud w\\ = \int_{-\infty}^\infty t^{-3} \YL \widehat{K}(W^{w} Z^z) \ud w
    = t^{-3} L(z).
  \end{multline}
  
  We first write both sides of \eqref{eq:first-z-convert-goal} in terms of integrals over subsets of $P$. 
  On one hand, for $r>0$, let
  $$I_r :=\int_{\R\setminus (-r^2,r^2)} (\gamma(Z^{z}) - \gamma(\zero)) L(z) \ud z.$$
  By \eqref{eq:L-homog}, since $\gamma$ is bounded, this integral converges absolutely. The right side of \eqref{eq:first-z-convert-goal} is equal to $\lim_{r\to 0} I_r$. For $q=W^wZ^z\in P$, we have $\pi_p(q)=Z^z=Z^{z(q)}$, and we define $\gamma_0\from P\to \R$, 
  $$\gamma_0(q) = \gamma(\pi_p(q)) - \gamma(\zero) = \gamma(Z^{z(q)}) - \gamma(\zero).$$
  Then by Fubini's Theorem, 
  $$I_r = \int_{\R\setminus (-r^2,r^2)} \gamma_0(Z^{z}) \int_{-\infty}^{\infty} \YL \widehat{K}(W^w Z^z) \ud w \ud z = \int_{E_r} \gamma_0(q) \YL \widehat{K}(q) \ud q,$$
  where $E_r = \{W^wZ^z\in P\mid |z|\ge r^2\}$.
  
  On the other hand,  let
  \begin{multline*}
    J_r:=\int_{V_0\setminus B(\zero,r)} (\gamma(\pi_p(q)) -\gamma(\zero)) \YL \widehat{K}(\Pi_P(q)) \ud q \\= \int_{V_0\setminus B(\zero,r)} \gamma_0(\Pi_P(q)) \YL \widehat{K}(\Pi_P(q)) \ud q 
    = \int_{\Pi_P(V_0\setminus B(\zero,r))} \gamma_0(q) \YL \widehat{K}(q) \ud q.
  \end{multline*}
  This integral likewise converges absolutely, and $\lim_{r\to 0} J_r$ is equal to the left side of \eqref{eq:first-z-convert-goal}.
  Let $F_r = \Pi_P(V_0\setminus B(\zero,r))\subset P$.
  
  Since $|\slope P|\le 1$, we have $\|v\|_\Kor \approx \|\Pi_P(v)\|_\Kor$ for all $v\in V_0$. In particular, if $v\in V_0\cap B(\zero,r)$, then $\|\Pi_P(v)\|_\Kor \le 2\|v\|_\Kor\le 2r$ and $|z(\Pi_P(v))| \le r^2$. Therefore, $E_r\subset F_r$.
  We thus consider the difference
  $$J_r - I_r  = \int_{F_r\setminus E_r}\gamma_0(q) \YL \widehat{K}(q) \ud q.$$
  
  For $0<s<S$, let $A_{s,S} = V_0\cap (B(\zero,S)\setminus B(\zero,s))$, and for $i\ge 0$, let 
  $$D_i = \Pi_P(A_{2^i r,2^{i+1} r})\setminus E_r,$$ so that up to a measure-zero set, $F_r\setminus E_r = \bigcup_{i=0}^\infty D_i$. For $q\in D_i$, we have $|z(q)|\le r^2$, $|x(q)|\le 2^{i+1}r$, and $|\YL\widehat{K}(q)|\lesssim (2^i r)^{-4}$. Furthermore, by the Mean Value Theorem 
  $$|\gamma_0(q)| = |\gamma(Z^{z(q)}) - \gamma(\zero)| \le \|\partial_z\gamma\|_\infty r^2\lesssim r^2.$$
  Therefore,
  $$|J_r-I_r| \lesssim \sum_{i=0}^\infty \mu(D_i)\cdot r^2\cdot (2^i r)^{-4} \lesssim \sum_{i=0}^\infty 2^i r^5 (2^i r)^{-4} = \sum_{i=0}^\infty 2^{-3i} r \le 2r.$$
  It follows that $\lim_{r\to 0} J_r = \lim_{r\to 0} I_r$, which implies \eqref{eq:first-z-convert-goal}.
\end{proof}
For $g\from \R\to \R$, we write
$$\pv \int_{\R} g(t) \ud t = \lim_{r\to 0} \int_{\R\setminus (-r,r)} g(t) \ud t,$$
as long as the limit on the right exists.

For the rest of this section, we restrict to the special case that $K$ is the Riesz kernel
\begin{align*}
  K(x,y,z) & :=\left( \frac{2 x (x^2+y^2) - 8 yz}{((x^2+y^2)^2+16 z^2)^{3/2}}, \frac{2 y(x^2+y^2)+ 8xz}{((x^2+y^2)^2+16 z^2)^{3/2}} \right)\\
  &= r^{-6} \left(2x (x^2+y^2)-8yz, 2 y(x^2+y^2) + 8 xz \right),
\end{align*}
where $r=\|(x,y,z)\|_{\Kor}$. One can calculate that
$$\YL \widehat{K}(X^x Z^z) = r^{-10}\left(64 z^3-20 x^4 z,80 x^2 z^2-x^6\right).$$

Integrating this using Mathematica, we find that
\begin{align*}
  L_{0}(1) = \int_{-\infty}^\infty \frac{\left(64 - 20 x^4,80 x^2-x^6\right)}{(x^4+16)^{\frac{5}{2}}}\ud x 
  = \left(0,  \frac{\Gamma(\frac{3}{4})\Gamma(\frac{7}{4})}{3\sqrt{\pi}}\right) \approx (0, 0.212\dots),
\end{align*}
where $\Gamma$ is the Euler gamma function. By the symmetry of $\widehat{K}$, we have $L_{0}(1) = L_{0}(-1)$. Let $\xi=y(L_{0}(1))$. By \eqref{eq:L-homog}, $L_{0}(z)=(0,|z|^{-\frac{3}{2}}\xi)$ for all $z\ge 0$.

A similar calculation shows
$$\XL \widehat{K}(X^x Z^z) = r^{-10}\left(48 x^2 z^2-3 x^6,20 x^4 z-64 z^3\right)$$
$$M_{0}(\pm 1) := \int_{-\infty}^\infty \XL \widehat{K}(X^x Z^{\pm 1}) \ud x = \int_{-\infty}^\infty \frac{\left(48 x^2 - 3x^6,\pm(20 x^4 - 64)\right)}{(x^4+16)^{\frac{5}{2}}}\ud x = (0,0).$$

These calculations imply that for any left-invariant horizontal vector field $F = a\XL + b\YL$ and any $z\ne 0$, the integral
$$\int_{-\infty}^\infty F \widehat{K}(X^x Z^{z}) \ud x = aM_0(z) + bL_0(z)$$
is normal to $V_0$ and is zero only if $F$ is a multiple of $X$. The orthogonal symmetry of $K$ implies that
$$\int_{-\infty}^\infty F \widehat{K}(W^w Z^{z}) \ud w$$
is likewise normal to $P$ and is zero only if $F$ is a multiple of $W$. In particular, $L_a(z) = |z|^{-\frac{3}{2}} L_a(1)$ is nonzero and normal to $X+aY$.


We use this formula to prove a lower bound for $G_{f_i,\nu_i}'(0)$. Let $\kappa\from V_0\to \R$ be as in Section~\ref{sec:construction}. That is, $\kappa$ is a bump function supported on $U=[0,1]\times \{0\}\times [0,1]$. Let $m>0$ be such that $\kappa(x,0,z)\ge m$ whenever $x,z\in [\frac{1}{4},\frac{3}{4}]$.
Recall that $r_i=A^{-1}\rho^{-i}$ and that we defined a set of pseudoquads $\mathcal{Q}_i=\{Q_{i,1}, \dots, Q_{i,k_i}\}$ that partition $U$, parametrizations $R_{i,j}\from [0,Ar_i] \times [0,r_i^2] \to Q_{i,j}$, bump functions $\kappa_{i,j}$
$$\kappa_{i,j}(R_{i,j}(s,t)) = A^{-1} r_i \kappa(A^{-1}r_i^{-1} s, r_i^{-2} t),$$
and a set $J_i\subset \{1,\dots, k_i\}$ such that $\nu_i=\sum_{j\in J_i} \kappa_{i,j}$. 

By Lemma~\ref{l:R-vert-bound}, there are functions $g_{i,j}$ such that
\begin{equation}\label{eq:def-gij}
  R_{i,j}(s,t) = R_{i,j}(0,0) + (s, g_{i,j}(s,t)),
\end{equation}
where $\partial_tg_{i,j}(s,t) \in [\frac{3}{4},\frac{4}{3}]$ for all $s$ and $t$. 

As in Section~\ref{subsec:rescaling}, we can rescale $f_i$ and $\nu_i$ by a factor of $r_i^{-1}$ to get functions $\alpha$ and $\gamma$ that satisfy Lemma~\ref{lem:gamma-deriv-bounds}. By Lemmas~\ref{lem:first-z} and \ref{lem:first-z-convert} and the scale invariance of the Riesz transform, for any $p\in \Gamma_{f_i}$,
\begin{multline}\label{eq:G'-single}
  G'_{f_i,\nu_i}(0)(p) = \pv \int_\R (\nu_i(pZ^z) - \nu_i(p)) L_a(z)\ud z + O(A^{-2}) \\ = L_a(1)\cdot \pv  \int_\R (\nu_i(pZ^z) - \nu_i(p)) |z|^{-\frac{3}{2}}\ud z + O(A^{-2}),
\end{multline}
where $a=\nabla_{f_i}f_i(p)$. 

This lets us prove the following bound.
\begin{lemma}\label{lem:point-lower}
  Suppose that $j\in J_i$ and $s\in [\frac{1}{4}Ar_i,\frac{3}{4}Ar_i]$ and let $p=R_{i,j}(s,0)$. There is a $c>0$ such that if $A$ is sufficiently large, then $|G'_{f_i,\nu_i}(0)(p)| \ge cA^{-1}$.
\end{lemma}
\begin{proof}
  Since $j\in J_i$, we have 
  $$\nu_i(R_{i,j}(s,t)) = \kappa_{i,j}(R_{i,j}(s,t)) = A^{-1} r_i \kappa(A^{-1}r_i^{-1} s, r_i^{-2} t)$$
  for all $t$. In particular, $\nu_i(p)=0$. Since $\nu_i$ is nonnegative,
  $$\pv \int_\R (\nu_i(pZ^z) - \nu_i(p)) |z|^{-\frac{3}{2}}\ud z = \int_\R \nu_i(pZ^z) |z|^{-\frac{3}{2}}\ud z \ge 0.$$
  
  Let $z_0 = z(p) = g_{i,j}(s,0)$. Then for $t\in [0,r_i^2]$, we have $R_{i,j}(s,t) = pZ^{g_{i,j}(s,t) - z_0}$. We thus substitute $z = g_{i,j}(s,t) - z_0$. Since $\partial_z g_{i,j}\in [\frac{3}{4},\frac{4}{3}]$, we have $\ud z \approx \ud t$ and $z \approx t$, so
  \begin{multline*}
    \int_\R \nu_i(pZ^z) |z|^{-\frac{3}{2}}\ud z \ge \int_{0}^{g_{i,j}(s,r_i^2) - z_0} \nu_i(pZ^z) |z|^{-\frac{3}{2}}\ud z \\
     \approx \int_{0}^{r_i^2} \nu_i(R_{i,j}(s,t))  |t|^{-\frac{3}{2}}\ud t
     \approx \int_{0}^{r_i^2} A^{-1}r_i \kappa(A^{-1}r_i^{-1} s, r_i^2 t) t^{-\frac{3}{2}}\ud t.
  \end{multline*}
  Let $\hat{s}=A^{-1}r_i^{-1}s$ and note that $\hat{s}\in [\frac{1}{4},\frac{3}{4}]$. We substitute $u=r_i^2 t$ and use the fact that $\kappa(\hat{s},u)\ge m$ for all $u\in [\frac{1}{4},\frac{3}{4}]$ to obtain
  \begin{equation*}
    \int_\R \nu_i(pZ^z) |z|^{-\frac{3}{2}}\ud z \gtrsim A^{-1} \int_{0}^{1} r_i \kappa(\hat{s}, u) (r_i^{-2} u)^{-\frac{3}{2}} r_i^{-2} \ud u \ge A^{-1} \int_{\frac{1}{4}}^{\frac{3}{4}} m u^{-\frac{3}{2}} \ud u \gtrsim \frac{A^{-1}}{2}.
  \end{equation*}
  By \eqref{eq:G'-single}, there is a $c_0>0$ such that $|G'_{f_i,\nu_i}(0)(p)| \gtrsim c_0 A^{-1} |L_a(1)| + O(A^{-2})$. Let $c=\frac{c_0}{2}\min_{a\in [-1,1]} |L_a(1)|$. When $A$ is sufficiently large, $|G'_{f_i,\nu_i}(0)(p)| \ge cA^{-1}$, as desired.
\end{proof}

Now we use Lemma~\ref{lem:Fprime-Lipschitz-dupe} to prove part (\ref{it:first-lower}) of Lemma~\ref{lem:mainBounds}.

\begin{proof}[{Proof of part (\ref{it:first-lower}) of Lemma~\ref{lem:mainBounds}}]
  Recall that $S_i=\bigcup_{j\not \in J_i} Q_{i,j}$. By Proposition~\ref{p:stopping-time}, there is an $\epsilon>0$ such that if $i\le \epsilon A^4$, then $|S_i|\le \frac{1}{2}$ and thus
  \begin{equation}\label{eq:Ji-lower}
    \sum_{j\in J_i}|Q_{i,j}| = 1-|S_i|\ge \frac{1}{2}.
  \end{equation}
  
  Let $c$ be as in Lemma~\ref{lem:point-lower}, so that 
  $$|G'_{f_i,\nu_i}(0)(R_{i,j}(s,0))| \ge cA^{-1}$$
  for all $j\in J_i$ and $s\in [\frac{1}{4}Ar_i,\frac{3}{4}Ar_i]$.
  Let $t\in [0,r_i^2]$, let $p=R_{i,j}(s,0)$, and let $q=R_{i,j}(s,t)$. 
  Since $\partial_t g_{i,j}(s,t) \in [\frac{3}{4},\frac{4}{3}]$ for all $s$ and $t$, we have 
  $$d_\Kor(p,q)\approx \sqrt{g_{i,j}(s,t)-g_{i,j}(s,0)} \approx \sqrt{t}.$$
  By Corollary~\ref{cor:Fprime-Holder}, there is an $a>0$ such that
  $$\left|G_{f_i,\nu_i}'(0)(p) - G_{f_i,\nu_i}'(0)(q)\right| \lesssim (r_i^{-1} \sqrt{t})^a + \rho^{-\frac{1}{2}}.$$
  We can thus choose $\rho_0, \delta>0$ depending only on $\kappa$ such that if $\rho>\rho_0$ and $t\in [0,\delta r_i^2]$, then 
  $$|G_{f_i,\nu_i}'(0)(R_{i,j}(s,t))|\ge \frac{c}{2}A^{-1}.$$

  Then
  \begin{multline*}
    \int_{Q_{i,j}} |G_{f_i,\nu_i}'(0)(q)|\ud q = 
    \int_0^{r_i^2} \int_0^{Ar_i} |G_{f_i,\nu_i}'(0)(R_{i,j}(s,t))| \partial_z g_{i,j}(s,t) \ud s \ud t \\
    \gtrsim \int_0^{\delta r_i^2} \int_{\frac{1}{4}Ar_i}^{\frac{3}{4}Ar_i} |G_{f_i,\nu_i}'(0)(R_{i,j}(s,t))| \ud s \ud t \gtrsim \delta A^{-1} |Q_{i,j}| \gtrsim A^{-1}|Q_{i,j}|.
  \end{multline*}
  By \eqref{eq:Ji-lower},
  $$\int_{U} |G_{f_i,\nu_i}'(0)(q)|\ud q \gtrsim \sum_{j\in J_i} \int_{Q_{i,j}} |G_{f_i,\nu_i}'(0)(q)|\ud q \gtrsim \sum_{j\in J_i} A^{-1}|Q_{i,j}| \ge \frac{1}{2}A^{-1},$$
  as desired.
\end{proof}

\section{Quasi-orthogonality}\label{sec:quasi-orthog}
In this section, we prove part (\ref{it:quasi-ortho}) of Lemma~\ref{lem:mainBounds}, which claims that there is an $\epsilon>0$ such that $|\langle F_i'(0), F_j'(0)\rangle| \lesssim \rho^{-\epsilon}$ for all $0\le i < j$.

Recall that $\nu_i$ oscillates with wavelength roughly $r_i = A^{-1}\rho^{-i}$, so we expect that $F_i'(0)$ also oscillates with wavelength roughly $r_i$. Since $r_j < \rho^{-1} r_i$, $F_j'(0)$ has higher frequency than $F_i'(0)$. We thus bound $\langle F_i'(0), F_j'(0)\rangle$ by partitioning $\Psi_{f_j}(U)\subset \Gamma_{f_j}$ into sets of diameter on the order of $r_j \rho^{\delta}$ for some small $\delta>0$. Let $Q$ be such a set. Since $j>i$, Lemma~\ref{lem:Fprime-Lipschitz-dupe} implies that $F_i'(0)$ is nearly constant on $Q$. We claim that the average of $F_j'(0)$ on $Q$ is small and thus $\int_Q F_i'(0)(q) F_j'(0)(q)\ud q$ is small.

We start by bounding the average of $F_j'(0)$ on rectangles (Section~\ref{sec:rectangle-averages}). We will then bound the average of $F_j'(0)$ on pseudoquads (Section~\ref{sec:quad-averages}) and complete the proof of Lemma~\ref{lem:mainBounds}.(\ref{it:quasi-ortho}) by tiling $U$ by pseudoquads (Section~\ref{sec:quasi-orthog-proof}).



\subsection{Averaging over rectangles}\label{sec:rectangle-averages}
We begin the proof of Lemma~\ref{lem:mainBounds}.(\ref{it:quasi-ortho}) by bounding the average of $F_j'(0)$ on rectangles of scale roughly $r_i\rho^\epsilon$. 

Let $P\subset \HH$ be a vertical plane of slope $a$ and let $W=X+aY$ so that $P=\langle W,Z\rangle$. For $v\in P$ and $r>0$, we define
$$E(v,r;P):=\{v W^w Z^z \mid |w|\le r, |z|\le r^2\}.$$
We call $E(v,r;P)$ a \emph{rectangle} in $P$.
In this section, we prove the following lemma.
\begin{lemma}\label{lem:rectangle-averages-rescaled}
  There is an $\epsilon>0$ such that
  when $\rho$ is sufficiently large, the following property holds. 
  Let $f_i$, $\nu_i$, and $\Sigma_i=\Gamma_{f_i}$ be as in Section~\ref{sec:construction}. Let $i\ge 0$,  let $p_0 \in \Sigma_i$, and let $P$ be the tangent plane to $\Sigma_i$ at $p_0$. Let $p\in P\cap B(p_0, \rho^\epsilon r_i)$ and $2r_i<R<\rho^{\epsilon} r_i$. Then
  \begin{equation}\label{eq:rect-avg-fi}
    \frac{1}{R^3}\left|\int_{E(p,R;P)} F_i'(0)(q) \ud q \right| \lesssim \rho^{-\epsilon} +  \frac{\log (R r_i^{-1})}{R r_i^{-1}}.
  \end{equation}
\end{lemma}

After a rescaling and translation, it suffices to consider the case that $\alpha$ and $\gamma$ satisfy Lemma~\ref{lem:gamma-deriv-bounds} and $p_0=\zero\in \Gamma_\alpha$.
Let $P$ be the tangent plane to $\Gamma_\alpha$ at $p_0$ and $W=X+\nabla_\alpha \alpha(p_0)Y$. Let $\epsilon>0$ be as in Lemma~\ref{lem:tildeH-approx-dupe}. It suffices to show that for $p\in B(p_0,\rho^{\frac{\epsilon}{2}};P)$ and $2<R<\rho^{\frac{\epsilon}{2}}$,
\begin{equation}\label{eq:rect-avg-gamma}
  \frac{1}{R^3}\left|\int_{E(p,R;P)} G_{\alpha,\gamma}'(0)(q) \ud q \right| \lesssim \rho^{-\epsilon} +  \frac{\log(R)}{R}.
\end{equation}  
By Lemma~\ref{lem:tildeH-approx-dupe}, when $\rho$ is sufficiently large,
\begin{equation}\label{eq:rect-avg-h-gamma}
  \frac{1}{R^3}\left|\int_{E(p,R;P)} G_{\alpha,\gamma}'(0)(q) \ud q - \int_{E(p,R;P)} H_{P,\gamma}(q) \ud q \right| \lesssim \rho^{-\epsilon}.
\end{equation}
Then Lemma~\ref{lem:rectangle-averages-rescaled} is a consequence of the following bound.
\begin{lemma}\label{lem:rectangle-H-average}
  Let $P$ and $\gamma$ be as above. Let $p\in P\cap B(p_0,\sqrt{\rho})$ and let $2 < R < \sqrt{\rho}$. Then
  $$\frac{1}{R^3} \left|\int_{E(p,R;P)} H_{P,\gamma}(q) \ud q \right| \lesssim \frac{\log(R)}{R}.$$
\end{lemma}

We first reduce Lemma~\ref{lem:rectangle-H-average} to a question about a singular integral on $P\times P$. For $r>0$, let
$$\Delta_r=\{(v,w)\in P\times P \mid d_\Kor(v,w)<r\}$$
and for $U\subset P\times P$, define
$$\pv \int_{U} M(v,w) \ud (v ,w) := \lim_{r\to 0}  \int_{U\setminus \Delta_r} M(v,w) \ud (v, w).$$
Recall that for $v\in P$, we defined
$$H_{P,\phi}(v) = \pv(v) \int_{p V_0} (\phi(w)-\phi(v))\YL \widehat{K}(v^{-1} \Pi_P(w)) \ud w.$$

\begin{lemma}\label{lem:to-singular-P}
  Let $\lambda$ be a bounded smooth function which is constant on cosets of $\paramY$. Let $p\in \HH$ and let $P$ be a vertical plane through $p$ with finite slope. Then 
  \begin{equation} \label{eq:to-singular-P-1}
    H_{P,\lambda}(p) = \pv(p) \int_{P}(\lambda(q)-\lambda(p)) \YL \widehat{K}(p^{-1} q) \ud q
  \end{equation}
  and 
  \begin{equation} \label{eq:to-singular-P-2}
    \int_{E(p,R;P)} H_{P,\lambda}(q) \ud q = \pv \int_{E(p,R;P) \times P}(\lambda(w)-\lambda(v)) \YL \widehat{K} (v^{-1}w) \ud(v,w).
    \end{equation}
\end{lemma}
\begin{proof}
  Without loss of generality, we suppose that $p=\zero$. Let $D_r:=V_0\cap B(\zero,r)$ and $D^P_r := P \cap B(\zero, r)$, and let $A_{r,R} := D_R\setminus D_r$ and $A^P_{r,R} := D^P_R\setminus D^P_r$.
  Let $\lambda_0(q)=\lambda(q)-\lambda(\zero)$ so that $H_{P,\lambda}(\zero) = \lim_{\substack{r\to 0\\R\to \infty}} L_{r,R}$, where
  $$L_{r,R} := \int_{D_R\setminus D_r} \lambda_0(q) \YL \widehat{K}(\Pi_{P}(q)) \ud q.$$
  Likewise, since $\lambda$ is constant on cosets of $\paramY$, we can write the right side of \eqref{eq:to-singular-P-1} as $\lim_{\substack{r\to 0\\R\to \infty}} M_{r,R}$, where
  $$M_{r,R} := \int_{D^P_R\setminus D^P_r} \lambda_0(q) \YL \widehat{K}(q) \ud q = \int_{\Pi(D^P_R)\setminus \Pi(D^P_r)} \lambda_0(q) \YL \widehat{K}(\Pi_{P}(q)) \ud q.$$
  
  Then
  \begin{align*}
    L_{r,R} - M_{r,R} & = \int_{V_0} \left(\one_{D^P_R} - \one_{D^P_r} - \one_{\Pi(D^P_R)} + \one_{\Pi(D^P_r)}\right)(q) \lambda_0(q) \YL \widehat{K}(\Pi_{P}(q)) \ud q \\
     & = \int_{V_0} (\one_{D_{R}} - \one_{\Pi(D^P_R)})(q) \lambda_0(q) \YL \widehat{K}(\Pi_P(q)) \ud q \\
    & \qquad - \int_{V_0} (\one_{D_r} - \one_{\Pi(D^P_r)})(q) \lambda_0(q) \YL \widehat{K}(\Pi_P(q)) \ud q \\
    & =: I_R - I_r.
  \end{align*}
  Note that there is a $c>1$ depending on the slope of $P$ such that 
  $$\supp(\one_{D_s} - \one_{\Pi(D^P_s)}) \subset A_{c^{-1}s,cs}$$ 
  for all $s>0$.
  
  Since $\lambda$ is bounded and $\YL \widehat{K}$ is $(-4)$--homogeneous,
  $$|I_R| \lesssim |A_{c^{-1}R,cR}| \cdot \|\lambda\|_\infty (c^{-1} R)^{-4} \lesssim_{\lambda, P} R^{-1},$$
  so $|I_R|\to 0$ as $R\to \infty$. 
  
  Let $\theta(x,y,z)=(-x,-y,z)$ and let $\lambda_0^\sE(q) = \frac{1}{2}(\lambda_0(q) + \lambda_0(\theta(q)))$ be the even part of $\lambda_0$. Then $\one_{D_r} - \one_{\Pi(D^P_r)}$ is even, so
  $$I_r = \int_{V_0} (\one_{D_{r}}(q) - \one_{\Pi(D^V_r)}(q)) \lambda_0^\sE(q) \YL \widehat{K}(\Pi_P(q)) \ud q.$$
  Since $\lambda_0$ is smooth and $\lambda_0(\zero)=0$, we have $|\lambda_0^\sE(q)|\lesssim_{\lambda} \|q\|_\Kor^2$ when $\|q\|_\Kor$ is sufficiently small, so when $r$ is sufficiently small,
  $$|I_r| \lesssim_\lambda |A_{c^{-1}r,cr}| \cdot r^2 r^{-4} \lesssim_{\lambda,P} r.$$
  Therefore, $|L_{r,R} - M_{r,R}|\lesssim_{\lambda,P} r + R^{-1}$, which implies \eqref{eq:to-singular-P-1}.
  
  Now let $E=E(p,R;P)$. Since $\YL \widehat{K}$ is $(-4)$--homogeneous and $\lambda$ is bounded, for any $r>0$ and $v\in P$,
  $$\int_{v A^P_{r,\infty}} (\lambda(w)-\lambda(v)) \YL \widehat{K}(v^{-1}w) \ud w$$
  converges absolutely. Furthermore, by the bounds above and Lemma~\ref{lem:G'-formula-and-bound-rev}, there is a $c>0$ depending only on $\lambda$ and $P$ such that for $v\in E$, 
  $$\left|H_{P,\lambda}(v) - \int_{v A^P_{r,\infty}} (\lambda(w) - \lambda(v)) \YL \widehat{K}(v^{-1}w) \ud w\right| \lesssim cr.$$
  Therefore, using uniform convergence to exchange the integral and the limit,
  \begin{align*}
    \int_{E}  H_{P,\lambda}(q) \ud q
      & = \int_{E} \lim_{r\to 0} \int_{v A^P_{r,\infty}} (\lambda(w) - \lambda(v)) \YL \widehat{K}(v^{-1}w) \ud w \ud v \\ 
      & = \lim_{r\to 0} \int_{E} \int_{v A^P_{r,\infty}} (\lambda(w)-\lambda(v)) \YL \widehat{K}(v^{-1}w) \ud w \ud v \\ & = \pv \int_{E\times P} (\lambda(w)-\lambda(v)) \YL \widehat{K}(v^{-1}w) \ud(v,w).
  \end{align*}
  This proves \eqref{eq:to-singular-P-2}.
\end{proof}

Now we prove Lemma~\ref{lem:rectangle-H-average}.
\begin{proof}[Proof of Lemma~\ref{lem:rectangle-H-average}]
  Let $E=E(p,R;P)$. By Lemma~\ref{lem:to-singular-P}
  \begin{multline*}
    \int_{E} H_{P,\gamma}(v)\ud v 
    = \pv \int_{E\times P\setminus E}(\gamma(w)-\gamma(v)) \YL \widehat{K}(v^{-1}w) \ud(v,w) \\
    + \pv \int_{E\times E}(\gamma(w)-\gamma(v)) \YL \widehat{K}(v^{-1}w) \ud(v,w) =: J_1 + J_2.
  \end{multline*}
  We claim that $|J_1|\lesssim R^2 \log R$ and $|J_2|\lesssim R^2$.
  
  We first consider $J_1$.
  First, we claim that $|\gamma(u)-\gamma(v)|\lesssim A^{-1} d_\Kor(u,v)$ for all $v\in E$ and $u\in P$. Let $\kappa:=d_\Kor(u,v)$.
  On one hand, if $\kappa \ge 1$, then
  $$|\gamma(u)-\gamma(v)|\lesssim \|\gamma\|_\infty \lesssim A^{-1}\le A^{-1} \kappa,$$
  so we consider the case that $\kappa\le 1$. Then $d_\Kor(u,p_0) \le d_\Kor(u,v) + d_\Kor(v,p_0) \le 2 R$. Let $W=X+\slope(P) Y$ and write $u=vW^wZ^z$ for some $|w|\le \kappa$ and $|z|\le \kappa^2$. Since $P$ is tangent to $\Gamma_\alpha$ at $p_0$, Lemma~\ref{lem:a-Taylor} and Lemma~\ref{lem:gamma-deriv-bounds} imply that for all $q\in P\cap B(p_0,2R)$, 
  $$d_{\Kor}(q, \Gamma_\alpha) \lesssim \rho^{-1} d_{\Kor}(q, p_0)^{2} \lesssim 1.$$
  Lemma~\ref{lem:horiz-gamma-bounds} implies that $|W\gamma(q)|\lesssim A^{-1}$ and $|Z\gamma(q)|\lesssim A^{-1}$. By the Mean Value Theorem, $|\gamma(u)-\gamma(v)|\lesssim (\kappa+\kappa^2)A^{-1}\lesssim \kappa A^{-1}$, as desired.
  
  Now let $v\in E$ and $\epsilon=d_\Kor(v,\partial E)$; suppose $\epsilon>0$. 
  Let 
  $$j_1(v)=\int_{P\setminus E}\left|(\gamma(w)-\gamma(v))\YL \widehat{K}(v^{-1}w)\right| \ud w.$$
  Since $\YL \widehat{K}$ is $(-4)$--homogeneous, by Lemma~\ref{lem:polar},
  \begin{multline*}
    j_1(v) \lesssim \int_{P\setminus B(v,\epsilon)} |\gamma(v)-\gamma(w)|d_\Kor(v,w)^{-4} \ud w \\
      \lesssim \int_{P\setminus B(v,\epsilon)} \min\{\|\gamma\|_\infty , A^{-1} d_\Kor(v,w)\} d_\Kor(v,w)^{-4} \ud w 
      \lesssim \int_\epsilon^\infty \min\{r^{-4},r^{-3}\}\cdot r^2 \ud r,
  \end{multline*}
  so there is a $C>0$ such that
  $$j_1(v) \le \begin{cases} 
    C |\log \epsilon| + C & 0<\epsilon<1 \\
    C \epsilon^{-1} & \epsilon \ge 1.
  \end{cases}$$
  For any $\epsilon_0>0$,
  $$\bigl|\{v\in E\mid d_\Kor(v,\partial E) < \epsilon_0\}\bigr| \lesssim \min\{R^3, R^2 \epsilon_0\},$$
  so for any $t>0$,
  $$\bigl|\{v\in E\mid j_1(v)>t\}\bigr| \lesssim \begin{cases}
      R^3 & 0\le t < \frac{C}{R} \\
      \frac{C}{t}R^2 & \frac{C}{R} \le t< C \\
      R^2\exp\big(\frac{C-t}{C}\big) & C \le t.
    \end{cases}$$
  Therefore,
  \begin{multline*}
    \int_{E} j_1(v) \ud v \lesssim \frac{C}{R} R^3 + \int_{CR^{-1}}^{C} \frac{C}{t} R^2 \ud t + \int_{C}^\infty R^2\exp\big(1-\frac{t}{C}\big) \ud t \\
    \le CR^2 + C R^2 \log R + C R^2 \lesssim R^2 \log R.
  \end{multline*}
  By Fubini's Theorem and dominated convergence,
  $$|J_1|=\left|\lim_{r\to 0} \int_{E} \int_{P\setminus (E\cup B(v,r))}(\gamma(w)-\gamma(v)) \YL \widehat{K}(v^{-1}w) \ud(v,w)\right| \le \int_{E} j_1(v)\ud v \lesssim R^2\log R.$$
  
  Now we consider $J_2$. We have
  \begin{align*}
    J_2&=\pv \int_{E\times E}(\gamma(w)-\gamma(v)) \YL \widehat{K}(v^{-1}w) \ud(v,w) \\
    & = \pv \int_{E\times E}\gamma(w) \YL \widehat{K}(v^{-1}w) \ud(v,w)
    - \pv \int_{E\times E}\gamma(v) \YL \widehat{K}(v^{-1}w) \ud(v,w).
  \end{align*}
  Exchanging $v$ and $w$ in the first term, we get
  $$J_2=\pv \int_{E\times E} \gamma(v)(\YL \widehat{K}(w^{-1}v)-\YL \widehat{K}(v^{-1}w))\ud(v,w) = \pv \int_{E\times E} \gamma(v) M(v^{-1}w) \ud(v,w),$$
  where $M(p) = \YL \widehat{K}(p^{-1})-\YL \widehat{K}(p)$. 
  We use the following lemma to show that $M$ is vertically antisymmetric, i.e., $M(W^wZ^z) = - M(W^wZ^{-z})$ for all $w,z\in \R$.
  \begin{lemma}\label{lem:quadrupole}
   Let $N\from \HH\to \R^2$ be an orthogonal kernel. For any horizontal vector $W$, let $\WR$ be the corresponding left-invariant vector field. Let $\theta(x,y,z)=(-x,-y,z)$ be the homomorphism that rotates $\HH$ around the $z$--axis by $\pi$.
    Then for any $g\in \HH$ and any horizontal vector $W$,
    $$q_W(g) := \WR N(g) - \WR N(g^{-1}) - \WR N(\theta(g)) + \WR N(\theta(g^{-1})) = 0.$$
  \end{lemma}
  We defer the proof until after the proof of Lemma~\ref{lem:rectangle-H-average}. Let $h\from \HH\to \HH$, $h(x,y,z) = (x,y,-z)$, so that $h(q)=\theta(q^{-1})$ for any $q\in \HH$. Then  $\widehat{K}$ is an orthogonal kernel, so for $q=W^wZ^z\in P$, 
  $$M(W^wZ^z) + M(W^wZ^{-z}) = \YL \widehat{K}(q^{-1}) - \YL \widehat{K}(q) + \YL \widehat{K}(\theta(q)) - \YL \widehat{K}(\theta(q^{-1})) = 0.$$
  
  By Fubini's theorem, 
  \begin{align*}
    J_2 & = \lim_{r\to 0} \int_{E\times E \setminus \Delta_r} \gamma(v) M(v^{-1}w) \ud(v,w) \\
    & = \lim_{r\to 0} \int_E \gamma(v) \int_{E \setminus B(v,r)} M(v^{-1}w) \ud w \ud v \\
    & = \lim_{r\to 0} \int_E \gamma(v) \int_{v^{-1} E \setminus B(\zero,r)} M(w) \ud w \ud v.
  \end{align*}
  For $S\subset P$ and $r>0$, let $k_r(S) = \int_{S \setminus B(\zero,r)} M(w) \ud w$. Then 
  $$S = (S\setminus h(S)) \sqcup (S\cap h(S)).$$
  The symmetry of $M$ implies that $k_r(S\cap h(S)) = 0$ and thus $k_r(S) = k_r(S\setminus h(S))$.
  
  Let $w_0\in (-R,R)$ and $z_0\in [0, R^2)$, so that $v = W^{w_0}Z^{z_0}$ lies in the top half of $E$. Let $\delta = R^2-z_0\in (0, R^2)$. Then 
  $$v^{-1} E = \{W^w Z^z : |w + w_0| \le R, z \in (\delta - 2R^2, \delta) \}$$
  and
  $$v^{-1} E \setminus h(v^{-1}E) = \{W^w Z^z : |w + w_0| \le R, z\in (\delta-2R^2, -\delta]\}.$$
  That is, $v^{-1} E \setminus h(v^{-1}E)\subset A^P_{2\sqrt{\delta},\infty}$. Therefore, for all $r > 0$, we can use Lemma~\ref{lem:polar} and $(-4)$--homogeneity of $M$ to show that
  $$|k_r(v^{-1} E)| = |k_r(v^{-1} E \setminus h(v^{-1}E)| \le \int_{A^P_{2\sqrt{\delta},\infty}} |M(w)| \ud w \lesssim \int_{2\sqrt{\delta}}^\infty \rho^{-2}\ud \rho \lesssim \delta^{-\frac{1}{2}}.$$
  
  More generally, letting $\delta(v)=\bigl|R^2 - |z(v)|\bigr|$, we have $|k_r(v^{-1}E)| \lesssim \delta(v)^{-\frac{1}{2}}$ for all $r>0$ and all $v\in E$. By dominated convergence,
  \begin{multline*}
    |J_2| = \left|\lim_{r\to 0} \int_E \gamma(v) k_r(v^{-1}E)\ud v\right| 
     \lesssim \|\gamma\|_\infty \int_E \delta(v)^{-\frac{1}{2}} \ud v \\
     \lesssim 2R \int_{-R^2}^{R^2} \bigl|R^2 - |z|\bigr|^{-\frac{1}{2}} \ud z 
     = 4R \int_{0}^{R^2} z^{-\frac{1}{2}} \ud z = 8 R^2.
  \end{multline*}
    
  Therefore, 
  $$\left|\int_{E} H_{P,\gamma}(v)\ud v\right|\le |J_1| + |J_2| \lesssim R^2\log R,$$
  as desired.
\end{proof}
We used Lemma~\ref{lem:quadrupole} in the proof of Lemma~\ref{lem:rectangle-H-average}, and we prove it now.
\begin{proof}[Proof of Lemma~\ref{lem:quadrupole}]
  We first consider the case that $g=(x,0,z)\in V_0$. Any left-invariant horizontal field can be written as a linear combination of $\XL$ and $\YL$, so it suffices to consider $W=X$ or $W=Y$.

  Let $I\from \HH\to \HH$ be an involutory linear isometry of $\HH$ that fixes $\zero$. Then $N(I(h)) = I(N(h))$ for all $h\in \HH$, so by the chain rule,
  $$\WR N(I(g)) = I_*(\WR)[N\circ I](g) = I_*(\WR)[I\circ N](g) = I(I_*(\WR)N(g)).$$
  Let $\phi(x,y,z)=(x,-y,-z)$ and $\psi(x,y,z)=(-x,y,-z)$ so that
  $$q_W(g) = \WR N(g) - \WR N(\psi(g)) - \WR N(\theta(g)) + \WR N(\phi(g)).$$
  Let $\XL N(g)=(a,b)$. Since $\phi$, $\psi$, and $\theta$ are involutory isometries of $\HH$,
  \begin{align*}
    q_X(g) & = \XL Ng) - \psi(\psi_*(\XL) N(g)) - \theta(\theta_*(\XL) N(g)) + \phi(\phi_*(\XL) N(g)) \\
    & = \XL N(g) + \psi(\XL N(g)) + \theta(\XL N(g)) + \phi(\XL N(g))\\
    & = (a,b)+(-a,b)+(-a,-b)+(a,-b)=0.
  \end{align*}
  Let $\YL N(g)=(c,d)$. Then
  \begin{align*}
    q_Y(g) & = \YL N(g) - \psi(\psi_*(\YL) N(g)) - \theta(\theta_*(\YL) N(g)) + \phi(\phi_*(\YL) N(g)) \\
  & = \YL N(g) - \psi(\YL N(g)) + \theta(\YL N(g)) - \phi(\YL N(g)) \\
  & = (c,d)-(-c,d)+(-c,-d)-(c,-d)=0.
\end{align*}
Thus the lemma holds for $g\in V_0$.

Let $g\in V_0$ and let $R\from \HH\to\HH$ be a rotation around the $z$--axis. Let $W'=R(W)$ so that $R_*(\WR)=\WR'$. Then, as above, $\WR N(R(g)) = R(R_*(\WR)N(g)),$ and since $R$ commutes with $\theta$,
\begin{align*}
  q_W(R(g))&=\WR N(R(g)) - \WR N(R(g^{-1})) - \WR N(R(\theta(g))) + \WR N(R(\theta(g^{-1})))\\
  &= R\left( \WR' N(g) - \WR' N(g^{-1}) - \WR' N(\theta(g)) + \WR' N(\theta(g^{-1}))\right)\\
  &= R(q_{W'}(g))= 0.
\end{align*}
Any point in $\HH$ can be written as $R(g)$ for some rotation $R$ and some $g\in V_0$, so $q_W(h)=0$ for all $h\in \HH$.
\end{proof}

\subsection{Averaging over pseudoquads}\label{sec:quad-averages}

In the previous section, we bounded the average of $F_i'(0)$ on rectangles of the form $E(p,r;P)$, where $P$ is tangent to $\Sigma_i$ at $p_0$ and $d_\Kor(p_0,p)\le r_i \rho^{\epsilon}$. The projections of these rectangles do not tile $V_0$, because $P$ depends on $p_0$, so in this section, we will bound the average of $F_i'(0)$ on pseudoquads for $\Sigma_i$.

We will need the following bound on the size of a pseudoquad of given height and width.
\begin{lemma}\label{lem:pq-variation}
Let $\psi$ be a $\lambda$--intrinsic Lipschitz function for some $\lambda\in (0,1)$. Let $\delta_x, \delta_z\ge 0$ and let $g_1,g_2\in C^1(\R)$ be functions such that for all $x$, $g_1(x)\le g_2(x)$, $g_i'(x)=-\psi(x,0,g_i(x))$, and $\delta_z=g_2(0)-g_1(0)$. 
Then for any $x\in \R$,
$$|g_1(x)-g_2(x)| \lesssim_\lambda \delta_z + x^2.$$
Let 
$$Q=\{(x,0,z)\mid x\in [0,\delta_x], z\in [g_1(x),g_2(x)]\}.$$
Then 
$$\diam \Psi_{\psi}(Q) \lesssim_\lambda \delta_x + \sqrt{\delta_z}$$
and
$$|Q|\lesssim_\lambda \delta_x \delta_z + \delta_x^3.$$
\end{lemma}
\begin{proof}
  By \eqref{eq:ilg-deriv-bound}, since $\psi$ is $\lambda$--intrinsic Lipschitz, the $g_i$ satisfy
  \begin{equation}\label{eq:conv-Lip}
    |g_i'(x) - g_i'(x')| = |\psi(x, 0, g_i(x)) - \psi(x', 0, g_i(x'))| \le \|\nabla_\psi \psi\|_\infty |x-x'| \le L |x-x'|,
  \end{equation}
  for all $x,x'\in \R$, where $L=\lambda(1-\lambda^2)^{-\frac{1}{2}}$. By Lemma~\ref{lem:ilg-line-distance}, 
  \begin{equation}\label{eq:left-edge}
    |\psi(0, 0, z_1) - \psi(0, 0, z_2)| \le \frac{2}{1-\lambda} d_\Kor((0,0,z_1), (0, 0, z_2)) = \frac{4}{1-\lambda} \sqrt{|z_2-z_1|}.
  \end{equation}
  Then $|g_1'(0)-g_2'(0)| = |\psi(0, 0, g_1(0)) - \psi(0, 0, g_2(0))| \le \frac{4}{1-\lambda} \sqrt{\delta_z}$, and
  $$|g_1'(x)-g_2'(x)| \le \frac{4}{1-\lambda} \sqrt{\delta_z} + 2L |x|$$
  for all $x$. Integrating this inequality and using the definition of $\delta_z$, we find
  $$|g_1(x)-g_2(x)|\le \delta_z + \frac{4}{1-\lambda} |x| \sqrt{\delta_z} + L x^2\lesssim_\lambda \delta_z + x^2.$$
  In particular, 
  $$|Q|=\int_0^{\delta_x} |g_2(x)-g_1(x)| \ud x \le \delta_x \cdot \left(\delta_z + \frac{4}{1-\lambda} \delta_x \sqrt{\delta_z} + L \delta_x^2\right) \lesssim_\lambda \delta_x \delta_z + \delta_x^3.$$
  
  Finally, let
  $$E=\{\Psi_\psi(0,0,z)\mid z\in [g_1(0),g_2(0)]\}$$
  be the left edge of $\Psi_\psi(Q)$. By 
  \eqref{eq:left-edge}, $\diam(E)\lesssim \sqrt{\delta_z}$.
  Every point $p\in \Psi_\psi(Q)$ lies on a horizontal curve in $\Psi_\psi(Q)$ that intersects $E$, and we can parametrize this curve as $\gamma=(\gamma_x,\gamma_y,\gamma_z)\from [0,\delta_x]\to Q$ where $x(\gamma(t))=t$. By \eqref{eq:conv-Lip}, $|\gamma_y'(t)|\le L$, so
  $$\ell(\gamma)=\int_0^{\delta_x} \sqrt{1+\gamma_y'(t)^2}\ud t \le \delta_x \sqrt{1+L^2},$$
  and 
  $$d_\Kor(\Psi_\psi(0, 0, g_1(0)), p)\le \diam E + \ell(\gamma)\lesssim_\lambda \sqrt{\delta_z}+\delta_x,$$
  as desired.
\end{proof}

Now we bound the integral of $F_i'(0)$ on a pseudoquad.
\begin{lemma}\label{lem:quad-averages}
  Let $r_i=A^{-1}\rho^{-i}$ and let $f_i$, $\nu_i$, and $\Sigma_i=\Gamma_{f_i}$ be as in Section~\ref{sec:construction}. There is a $\delta>0$ such that if $\rho$ is sufficiently large, then for any pseudoquad $Q$ for $\Sigma_i$ with $\delta_x(Q) \le r_i \rho^{\delta}$ and $\delta_z(Q) \le r_i^{2} \rho^{2\delta}$, we have
  $$\left|\int_{Q} F_i'(0)(q) \ud q\right| \lesssim
  r_i^3 \rho^{3 \delta - \frac{\delta}{3}}.$$
\end{lemma}

\begin{proof}
  Let $\epsilon>0$ be as in Lemma~\ref{lem:rectangle-averages-rescaled} and let $\delta=\frac{\epsilon}{2}$.
  After a left-translation, we may suppose that $f_i(\zero)=0$ and that the lower left corner of $\Psi_{f_i}(Q)$ is $\zero$. That is,
  $$Q=\{(x,0,z)\mid x\in [0,\delta_x(Q)], z\in [g_1(x),g_2(x)]\}$$
  where $g_1,g_2\from [0, \delta_x(Q)]\to \R$ are functions with characteristic graphs such that $g_1(0)=0$, $g_2(0)=\delta_z(Q)$. 
  By Lemma~\ref{lem:pq-variation}, $|g_2(x)-g_1(x)|\lesssim r_i^2\rho^{2\delta}$ for all $x\in [0,\delta_x(Q)]$,
  $\diam(\Psi_{f_i}(Q))\lesssim r_i \rho^\delta$,
  and $|Q| \lesssim r_i^3 \rho^{3\delta}$.
  In particular, for any $q\in Q$, the intrinsic Lipschitz condition implies 
  \begin{equation}\label{eq:fiQ-bound}
    |f_i(q)|\lesssim |f_i(\zero)| + \diam(Q) \lesssim r_i \rho^\delta.
  \end{equation}
  
  Let $P$ be the tangent plane to $\Sigma_i$ at $\zero$, and let $\sigma=\nabla_{f_i}f_i(\zero)$ be the slope of $P$. Note that $|\sigma|\le 1$.
  Let $W=X+\sigma Y$ so that $P=\langle W,Z\rangle$.
  We will cover $\Pi_P(Q)$ by rectangles. Recall that 
  $$\Pi_P(x,0,z)=(x,0,z)\cdot Y^{\sigma x} = \left(x,\sigma x, z+\frac{\sigma}{2}x^2\right)=W^xZ^{z+\frac{\sigma}{2}x^2}.$$
  Let $\bar{g}_j(x)=g_j(x)+\frac{\sigma}{2}x^2$, so that 
  $$\Pi_P(X^x Z^{g_j(x)}) = X^x Z^{g_j(x)} Y^{\sigma x} = W^x Z^{\bar{g}_j(x)}$$
  and
  $$\Pi_P(Q)=\left\{W^w Z^z \mid w\in [0,\delta_x(Q)], z\in [\bar{g}_1(w),\bar{g}_2(w)]\right\}.$$
  Note that $d_\Kor(q,\Pi_P(q)) \le |x(q)|$, so $\diam(\Pi_P(Q))\lesssim \diam(Q) \lesssim r_i\rho^{\delta}$.
  
  Since $g_j$ has a characteristic graph, it satisfies 
  $$g'_j(x) = -f_i(x,0,g_j(x)) = -f_i(X^x Z^{g_j(x)});$$
  it follows that $\bar{g}_j$ satisfies 
  \begin{equation}\label{eq:barg-characteristic}
    \bar{g}'_j(x) = -f_i(X^x Z^{g_j(x)}) + \sigma x = \sigma x - f_i(W^x Z^{\bar{g}_j(x)}).
  \end{equation}
  In particular, for $x\in [0,\delta_x(Q)]$,
  \begin{equation}\label{eq:barg-deriv}
    |\bar{g}_j'(x)| \le |\sigma| \delta_x + \left|f_i(W^x Z^{\bar{g}_j'(x)})\right| \lesssim r_i \rho^{\delta}.
  \end{equation}
  
  Let $R=r_i \rho^{\frac{\delta}{2}}$ and let 
  $$D=\{W^wZ^z\mid w\in [0, 2R], z\in [0, 2R^2]\}$$
  This is a translate of $E(\zero, R;P)$. Let
  $$\cT=\{W^{2jR} Z^{2kR^2}D\mid j,k\in \Z\}$$
  be a tiling of $P$ by translates of $D$, let 
  $$\cS_0=\{E\in \cT\mid \interior E \subset \Pi_P(Q)\},$$
  and let 
  $$\cS_1=\{E\in \cT\mid \interior E \cap \partial \Pi_P(Q)\ne \emptyset\}.$$
  The rectangles in $\cS_0$ and $\cS_1$ cover $\Pi_P(Q)$, and
  \begin{equation}\label{eq:avg-break}
    \left|\int_Q F'_i(0)(q)\ud q\right| \le \sum_{E\in \cS_0} \left|\int_E F'_i(0)(q)\ud q\right| + \sum_{E\in \cS_1} \int_E |F'_i(0)(q)|\ud q.
  \end{equation}
  
  Since each rectangle in $\cS_0$ has measure $4R^3$, we have $4R^3 \cdot \#\cS_0 \le |Q|\lesssim r_i^3\rho^{3\delta}$. When $\rho$ is sufficiently large, we have $E \subset B(\zero, r_i \rho^\epsilon)$ for every $E\in \cS_0\cup \cS_1$, so Lemma~\ref{lem:rectangle-averages-rescaled} implies that
  $$\left|\int_{E} F'_i(0)(q)\ud q\right| \lesssim R^3 \rho^{-\epsilon} + R^3 \frac{\log (R r_i^{-1})}{R r_i^{-1}} = R^3 \rho^{-\epsilon} + R^3 \frac{\log \rho^{\frac{\delta}{2}}}{\rho^{\frac{\delta}{2}}} \lesssim R^3 \rho^{-\frac{\delta}{3}}$$
  for any $E\in \cS_0$. Then
  \begin{equation}\label{eq:s0-bound}
    \sum_{E\in \cS_0} \left|\int_E F'_i(0)(q)\ud q\right| \lesssim \#\cS_0 \cdot R^3 \rho^{-\frac{\delta}{3}} \lesssim r_i^3\rho^{3\delta - \frac{\delta}{3}}. 
  \end{equation}
  
  Now we consider the $\cS_1$ term. We first bound the number of elements of $\cS_1$. 
  If $E\in \cS_1$, then $E$ intersects one of the edges of $\Pi_P(Q)$. Let $\cS^{\text{lr}}_1\subset \cS_1$ be the set of rectangles that intersect the left or right edge and let $\cS^{\text{tb}}_1\subset \cS_1$ be the set that intersect the top or bottom edge.

  By Lemma~\ref{lem:pq-variation}, there is a $C>1$ such that the left and right edges of $\Pi_P(Q)$ are vertical segments of height at most $C r_i^2 \rho^{2\delta}$. Since each $E\in \cT$ is a rectangle of height $2R^2$,
  $$\#\cS_1^{\text{lr}} \le \frac{C r_i^2 \rho^{2\delta}}{R^2} + 2 \lesssim C \rho^\delta.$$

  The top and bottom edges of $\Pi_P(Q)$ are the curves 
  $$\gamma_j=\{W^w Z^{\bar{g}_j(w)}\mid w\in [0, \delta_x(Q)]\}.$$
  We can partition $\cT$ into strips of rectangles with the same $x$--coordinates, i.e.
  $$\cT_k=\{E\in \cT\mid x(E)=[2kR,2(k+1)R]\}.$$
  Then for each $0\le k \le \frac{\delta_x(Q)}{2R}$, 
  $$|\{E\in \cT_k \mid E\cap \gamma_j\ne \emptyset\}| \le \frac{1}{2R^2}\int_{2kR}^{2(k+1)R} |\bar{g}_j'(x)|\ud x + 2 \stackrel{\eqref{eq:barg-deriv}}{\lesssim} \frac{1}{R} r_i \rho^{\delta} + 2 \le \rho^{\frac{\delta}{2}}$$
  and
  $$\# \cS_1^{\text{tb}} \lesssim \frac{\delta_x(Q)}{R} \rho^{\frac{\delta}{2}}\le\rho^\delta.$$

  Therefore, $\#\cS_1\lesssim \rho^\delta$. By part (\ref{it:first-upper}) of Lemma~\ref{lem:mainBounds}, $\|F_i'(0)\|_\infty \lesssim A^{-1}$, so
  \begin{equation}\label{eq:s1-bound}
    \sum_{E\in \cS_1} \int_E |F'_i(0)(q)|\ud q \lesssim \rho^\delta R^3 A^{-1} \le r_i^3\rho^{\frac{5}{2}\delta}. 
  \end{equation}
  By \eqref{eq:avg-break}, \eqref{eq:s0-bound}, and \eqref{eq:s1-bound},
  $$\left|\int_Q F'_i(0)(q)\ud q\right| \lesssim r_i^3\rho^{3\delta - \frac{\delta}{3}} + r_i^3\rho^{3\delta - \frac{\delta}{2}}\lesssim r_i^3\rho^{3\delta - \frac{\delta}{3}},$$
  as desired.
\end{proof}

\subsection{Proof of Lemma~\ref{lem:mainBounds}.(\ref{it:quasi-ortho})}\label{sec:quasi-orthog-proof} 
Let $0\le i < j \le N$.
Let $\delta>0$ be as in Lemma~\ref{lem:quad-averages}; note that we can take $\delta<\frac{1}{2}$. Let $\epsilon>0$ be as in Lemma~\ref{lem:Fprime-Lipschitz-dupe}; we take $\epsilon<1$. We claim that 
$$|\langle F_i'(0), F_j'(0)\rangle|\lesssim \rho^{-\min\{\frac{\epsilon}{4}, \frac{\delta}{3}\}}.$$

Recall that $f_j$ is supported on the unit square $U=[0,1]\times \{0\} \times [0,1]$, so that the top and bottom boundaries of $U$ are  characteristic curves of $\Sigma_j$. 

Let
$w\in [\frac{1}{2}r_j \rho^{\delta},r_j \rho^\delta]$ and $h\in [\frac{1}{2}r_j^2 \rho^{2\delta},r_j^2\rho^{2\delta}]$ be such that $N_x := w^{-1}$ and $N_z := h^{-1}$ are integers. For $m=0, \dots, N_x$ and $k=0,\dots, N_z$, let $v_{m,k} = (m w, 0, k h) \in V_0$ and let $g_{m,k}\from [m w, (m+1)w] \to \R$ be the function such that the graph $z=g_{m,k}(x)$ is a segment of the characteristic curve of $\Sigma_j$ through $v_{m,k}$. For $m=0, \dots, N_x-1$ and $k=0,\dots, N_z-1$, let $Q_{m,k}$ be the pseudoquad
$$Q_{m,k}:= \{(x,0,z)\mid x\in [m w, (m+1)w], z\in [g_{m,k}(x),g_{m,k+1}(x)]\};$$
this is the pseudoquad of $\Sigma_j$ with lower-left corner $v_{m,k}$, $\delta_x(Q_{m,k})=w$, and $\delta_z(Q_{m,k})=h$.
The pseudoquads $Q_{m,k}$ then have disjoint interiors and cover $U$. 

By Lemma~\ref{lem:quad-averages}, for every $m$ and $k$,
\begin{equation}\label{eq:quad-avg-bds}
  \left|\int_{Q_{m,k}} F_j'(0)(q) \ud q\right| \lesssim r_j^3 \rho^{3\delta - \frac{\delta}{3}}.
\end{equation}

Suppose that $p,q\in Q_{m,k}$. We claim that $\Psi_{f_i}(p)$ is close to $\Psi_{f_i}(q)$ and thus $|F_i'(0)(p)-F_i'(0)(q)|$ is small.
Let $p_n=\Psi_{f_n}(p)$ and $q_n=\Psi_{f_n}(q)$.
By Lemma~\ref{lem:pq-variation}, 
$$d_\Kor(p_j, q_j) \le \diam \Psi_{f_j}(Q_{m,k})\lesssim r_j\rho^\delta.$$
Since $i<j$,
$$\|f_j - f_i\|_\infty \le \sum_{n=i}^{j-1} \|\nu_n\|_\infty \lesssim \sum_{n=i}^{j-1} A^{-1}r_n \lesssim r_{i}.$$
Let $a=f_i(p)-f_j(p)$ and let $b=x(q)-x(p)$. Then $p_i=p_jY^a$, and
\begin{multline*}
  d_\Kor(p_i, q_i\paramY)\le d_\Kor(p_i, q_j Y^a) = \|Y^{-a} p_j^{-1} q_j Y^a\|_\Kor \\
  = \|p_j^{-1} q_j Z^{ab}\|_\Kor \lesssim d_\Kor(p_j,q_j) + \sqrt{ab} \lesssim r_j\rho^\delta + \sqrt{r_i r_j \rho^\delta}.
\end{multline*}
Since $j>i$, we have $r_j \le \rho^{-1}r_i$, so 
$$d_\Kor(p_i, q_i\paramY) \lesssim r_i \rho^{-1+\delta} + r_i \rho^{\frac{-1+\delta}{2}} \lesssim r_i \rho^{-\frac{1}{4}}.$$

Let $m=d_\Kor(p_i, q_i\paramY)$ and let $c\in q_i\paramY$ satisfy $d_\Kor(p_i,c)=m$. Note that $p_i = \Psi_{f_i}(p_i)$ and so $y(p_i) = f_i(p_i)$.  As the $y$ function is 1-Lipschitz we get that $|f_i(p_i) - y(c)| \lesssim m$. By Lemma~\ref{lem:ilg-line-distance}, $|f_i(p_i)-f_i(q_i)| \lesssim d_\Kor(p_i, q_i\paramY)$.  Thus,
$$d_\Kor(p_i, q_i) \le d_\Kor(p_i, c) + |y(c)-f_i(p_i)| + |f_i(p_i)-f_i(q_i)| \lesssim m \lesssim r_i \rho^{-\frac{1}{4}}.$$
Therefore, by Corollary~\ref{cor:Fprime-Holder}, 
\begin{equation}\label{eq:quad-centers-diff}
  |F_i'(0)(p)-F_i'(0)(q)| = |F_i'(0)(p_i)-F_i'(0)(q_i)| \lesssim \rho^{-\frac{\epsilon}{4}}.
\end{equation}

Then
\begin{align*}
  \Bigg|\int_U &F_i'(0)(q)F_j'(0)(q) \ud q\Bigg| 
  \le \sum_{m,k} \left|\int_{Q_{m,k}} F_i'(0)(q)F_j'(0)(q) \ud q\right| \\
  & \le \sum_{m,k} \left|\int_{Q_{m,k}} F_j'(0)(q)\left[\left(F_i'(0)(q)-F_i'(0)(v_{m,k})\right) + F_i'(0)(v_{m,k})\right] \ud q\right| \\
  & \le \sum_{m,k} \int_{Q_{m,k}} \left|F_i'(0)(q)-F_i'(0)(v_{m,k})\right| |F_j'(0)(q)| \ud q \\
  &\qquad + \sum_{m,k}|F_i'(0)(v_{m,k})| \left|\int_{Q_{m,k}} F_j'(0)(q) \ud q \right|,
\end{align*}
where the sums are all taken over $0\le m < N_x$ and $0\le k < N_z$.
Lemma~\ref{lem:mainBounds}.(\ref{it:first-upper}) implies that $\|F_n'(0)\|_\infty \lesssim A^{-1}\lesssim 1$, so by \eqref{eq:quad-centers-diff},
$$\sum_{m,k}\int_{Q_{m,k}} \left|F_i'(0)(q)-F_i'(0)(v_{m,k})\right| |F_j'(0)(q)| \ud q \lesssim |U| \rho^{-\frac{\epsilon}{4}} \|F_j'(0)\|_\infty \lesssim \rho^{-\frac{\epsilon}{4}}.$$
Likewise, by \eqref{eq:quad-avg-bds},
$$\sum_{m,k} |F_i'(0)(v_{m,k})| \left|\int_{Q_{m,k}} F_j'(0)(q) \ud q \right| \lesssim N_x N_z r_j^3 \rho^{3\delta-\frac{\delta}{3}}
\lesssim \rho^{-\frac{\delta}{3}}.$$
Therefore,
$$\big|\langle F_i'(0), F_j'(0)\rangle\big|\lesssim \rho^{-\frac{\delta}{3}} + \rho^{-\frac{\epsilon}{4}},$$
as desired.

\section{Second derivative bounds}\label{sec:second-deriv}

In this section, we will prove the following lemma.
\begin{lemma}\label{lem:second-bounds}
  For any $A>1$ and any $C>0$, if $\rho$ is sufficiently large, then the following bounds hold. Let $\alpha, \gamma\from \HH\to \R$ be functions that satisfy Lemma~\ref{lem:gamma-deriv-bounds}.
  Then
  $$\left\|G_{\alpha,\gamma}''(t)\right\|_\infty \lesssim_C A^{-3}$$
  for all $t\in [0,1]$.
\end{lemma}

We first set some notation that we will use in the rest of this section. Similarly to Section~\ref{sec:perturbations}, given functions $\alpha$ and $\gamma$ that satisfy Lemma~\ref{lem:gamma-deriv-bounds}, for $\tau\in \R$, we define
$$a_\tau(w) := (\alpha+\tau\gamma)(Y^{(\alpha+\tau\gamma)(\zero)} w) - (\alpha+\tau \gamma)(\zero)$$
(as in \eqref{eq:def-zt}) and
$$b_\tau(w) := \gamma(Y^{(\alpha+\tau\gamma)(\zero)} w).$$
By Lemma \ref{lem:ilg-translations}, these are translates of $\alpha$ and $\gamma$ in the sense that $a_\tau(\zero) = 0$,
$$\Gamma_{a_\tau} = Y^{-(\alpha+\tau\gamma)(\zero)} \Gamma_{\alpha + \tau \gamma},$$
and
$$\Gamma_{a_\tau + t b_\tau} = Y^{-(\alpha+\tau\gamma)(\zero)} \Gamma_{\alpha + (\tau + t) \gamma}.$$
By the left-invariance of the Riesz transform, for any $\tau, t\in \R$,
\begin{equation}\label{eq:G-forms}
  G_{\alpha,\gamma}(\tau+t)(\zero) = G_{a_\tau,b_\tau}(t)(\zero) = F_{a_{\tau+t}}(\zero) = \pv(\zero)\int_{V_0} \widehat{K}(\Psi_{a_{\tau+t}}(v))\ud v.
\end{equation}

We will use \eqref{eq:G-forms} to decompose $G_{\alpha,\gamma}(\tau)(\zero)$ and differentiate the decomposition.
We fix some $\tau\in [0,1]$ and abbreviate $a=a_\tau$ and $b=b_\tau$. For $w\in \HH$, let $\overline{w} = \Psi_a(w)$ and $w_t = Y^{t b(\zero)} w Y^{- t b(\zero)} = w Z^{-b(\zero)x(w)}$. Then 
\begin{equation} \label{eq:a-tau-t}
  \Psi_{a_{\tau + t}}(w) = Y^{- t b(\zero)} \Psi_{a + t b}(w_t) = Y^{- t b(\zero)} \overline{w_t} Y^{t b(w_t)}.
\end{equation}
For $0<r<R$, let $D_r=B(\zero,r)\cap V_0$ and $A_{r,R}=D_R\setminus D_r$. Then we can decompose $G_{\alpha,\gamma}(\tau + t)(\zero) = G_{a,b}(t)(\zero)$ as follows:
\begin{align*}
  G_{a,b}(t)(\zero) & = \pv(\zero)\int_{D_1} \widehat{K}(\Psi_{a_{\tau + t}}(v)) \ud v + \pv(\zero)\int_{V_0\setminus D_1} \widehat{K}(\Psi_{a_{\tau + t}}(v)) \ud v \\
  & = \pv(\zero)\int_{D_1} \widehat{K}(\Psi_{a_{\tau + t}}(v)) \ud v + \pv(\zero)\int_{V_0\setminus D_1} \widehat{K}(Y^{- t b(\zero)} \overline{w_t} Y^{t b(w_t)}) \ud w \\
  & = \pv(\zero)\int_{D_1} \widehat{K}(\Psi_{a_{\tau + t}}(v)) \ud v + \lim_{R\to \infty} \int_{A_{1,R}^{t}} \widehat{K}(Y^{- t b(\zero)} \overline{w} Y^{t b(w)}) \ud w \\
  & =: \Gsm_{\tau}(t) + \Glg_{\tau}(t),
\end{align*}
where 
$$A_{1,R}^{t} := Y^{tb(\zero)} A_{1,R} Y^{-tb(\zero)}.$$
For $0<r<1<R$, let
$$\Gsm_{\tau,r}(t) := \int_{A_{r,1}} \widehat{K}(\Psi_{a_{\tau + t}}(v)) \ud v$$
$$\Glg_{\tau,R}(t) := \int_{A_{1,R}^{t}} \widehat{K}(Y^{- t b(\zero)} \overline{w} Y^{t b(w)}) \ud w.$$
We will show the following bounds.
\begin{lemma}\label{lem:unif-sec-lg}
  For any $\tau\in [0,1]$ and any $1\le R<R'$,
  \begin{equation}\label{eq:unif-sec-lg}
    \left|(\Glg_{\tau,R'})''(0) - (\Glg_{\tau,R})''(0)\right| \lesssim A^{-3}R^{-1}.
  \end{equation}
\end{lemma}
\begin{lemma}\label{lem:unif-sec-sm}
  For any $\tau\in [0,1]$ and any $0<r'<r \le 1$
  $$\left|(\Gsm_{\tau,r})''(0) - (\Gsm_{\tau,r'})''(0)\right| \lesssim A^{-3}r.$$
\end{lemma}

\begin{proof}[Proof of Lemma \ref{lem:second-bounds}]
These lemmas show that the functions $\tau\mapsto (\Gsm_{\tau,r})''(0)$ and $\tau\mapsto (\Glg_{\tau,R})''(0)$ are uniformly Cauchy on the interval $\tau\in [0,1]$ as $r \to 0$ and $R \to \infty$. Let $f^\sml(\tau) = \lim_{r\to 0} (\Gsm_{\tau,r})''(0)$ and $f^\lrg(\tau) = \lim_{R\to \infty} (\Glg_{\tau,R})''(0)$. Then $(\Gsm_{\tau})''(0) = f^\sml(\tau)$ and $(\Glg_\tau)''(0) = f^\lrg(\tau)$.
Moreover, as $(\Gsm_{\tau,1})''(0) = 0 = (\Glg_{\tau,1})''(0)$ for all $\tau$,
\begin{align*}
  |(\Gsm_\tau)''(0)| &= \lim_{r \to 0} |(\Gsm_{\tau,r})''(0) - (\Gsm_{\tau,1})''(0)| \lesssim A^{-3}, \\
  |(\Glg_\tau)''(0)| &= \lim_{R \to \infty} |(\Glg_{\tau,R})''(0) - (\Glg_{\tau,1})''(0)| \lesssim A^{-3}.
\end{align*}
Thus
\begin{align*}
  \left|G_{a,b}''(\tau)(\zero)\right| = \left| (\Glg_\tau)''(0) + (\Gsm_\tau)''(0) \right| \lesssim A^{-3},
\end{align*}
as desired.
\end{proof}

Before we prove Lemmas~\ref{lem:unif-sec-lg} and \ref{lem:unif-sec-sm}, we will need some lemmas. The first proves bounds on the derivatives of $a$ and $b$ which follow from Lemma~\ref{lem:gamma-deriv-bounds}.
\begin{lemma}\label{lem:ab-deriv-bounds}
  There is a $c>0$ such that for any $k\le 3$ and any word $F \in \{\nabla_a, Z\}^k$
  \begin{equation}\label{eq:b-jk-bounds}
      \|F b\|_{\infty} \le c A^{-\# \nabla_a(F)-1},
  \end{equation}
  If $F \notin \{\id, \nabla_\alpha\}$, then
  \begin{equation}\label{eq:a-jk-bounds} 
    \|F a\|_{\infty} \le c A^{-\# \nabla_a(F)-1}.
  \end{equation}
\end{lemma}
\begin{proof}
  Let $m = \alpha(\zero) + \tau \gamma(\zero)$ and let $\lambda\from \HH \to \HH$, $\lambda(p) = Y^{m}p$. Then $\nabla_a = \lambda^*(\nabla_{\alpha + \tau \gamma})$, and by the Chain Rule, if $F \in \{\nabla_a, Z\}^k$, then 
  $$F a(p) = F'[\alpha+\tau \gamma - m](\lambda(p)),$$
  where $F'$ is obtained from $F$ by replacing $\nabla_a$ by $\nabla_{\alpha + \tau \gamma}$. 
  
  Let $\hat{\nabla} = A \nabla_{\alpha + \tau \gamma}$. It suffices to prove that for any $k\le 3$ and any $E \in \{\hat{\nabla}, Z\}^k$,
  $$\|E \gamma\|_{\infty} \le c A^{-1}$$
  and that if $E \notin \{\id, \hat{\nabla}\}$, 
  $$\|E \alpha\|_{\infty} \le c A^{-1}.$$
  
  Let $g = - A \tau \gamma$ and $\nabla = A\nabla_\alpha$ so that $\hat{\nabla} = \nabla + g Z$.  
  Suppose by induction that for any $k\le d$, we can write any $E\in \{\hat{\nabla}, Z\}^k$ as
  \begin{equation}\label{eq:E-decomp}
    E = \sum_i D_{i,1}[g]\dots D_{i,k_i}[g] \cdot C_i
  \end{equation}
  where $C_i, D_{i,j}\in \{\nabla, Z\}^*$ and $\ell(C_i) + \sum_j \ell(D_{i,j}) = k$. We call the $C_i$'s the \emph{monomials} of $E$. For instance, $Z$ is trivially of the form \eqref{eq:E-decomp}, we can write $\hat{\nabla} = \nabla + g Z$, and
  $$\hat{\nabla}^2 = \nabla^2 + \nabla g \cdot Z + g \cdot \nabla Z + g \cdot Z \nabla + g\cdot Zg \cdot Z + g^2\cdot Z^2.$$
  By the product rule, if $E$ can be written in this form, then so can $ZE$ and $\hat{\nabla} E$, and each monomial of $ZE$ or $\hat{\nabla}E$ is a monomial of $E$ or a monomial of $E$ with one additional letter. 
  
  If $E\not \in \{\id, \hat{\nabla}\}$, then $E$ ends in either $Z$, $\hat{\nabla}^2$, or $Z\hat{\nabla}$. Since $\id$ and $\nabla$ are not monomials of $Z$, $\hat{\nabla}^2$, or $Z\hat{\nabla}$, they cannot be monomials of $E$.
  
  By Lemma~\ref{lem:gamma-deriv-bounds}, if $k\le 3$ and $C\in \{\nabla, Z\}^k$, then $\|C\gamma\|_\infty \lesssim A^{-1}$, and if $C\not \in \{\id, \nabla\}$, then $\|C\alpha\|_\infty \lesssim A^{-1}$. If $E\in \{\hat{\nabla}, Z\}^k$ is as in \eqref{eq:E-decomp}, then
  $$\|E \gamma\|_{\infty} \lesssim \sum_i \|D_{i,1}[A\gamma]\|_\infty \cdots \|D_{i,k_i}[A\gamma]\|_\infty \|C_i \gamma\|_\infty \lesssim A^{-1}.$$
  Moreover, if $E\not \in \{\id, \hat{\nabla}\}$, then $C\not \in \{\id, \nabla\}$ for all $i$, so
  $$\|E \alpha\|_{\infty} \lesssim \sum_i \|D_{i,1}[A\gamma]\|_\infty \cdots \|D_{i,k_i}[A\gamma]\|_\infty \|C_i \alpha\|_\infty \lesssim A^{-1}$$
  as well.
\end{proof}

As a consequence, $Z a$, $b$, and $Z b$ are close to even. Recall that for a function $f\from \HH\to \R$, we define the even and odd parts of $f$ by
\begin{align*}
f^\sE(v) & = \frac{f(v) + f(\theta(v))}{2}
& f^\sO(v) & = \frac{f(v) - f(\theta(v))}{2}
\end{align*}
so that $f=f^\sE+f^\sO$. Furthermore, if $g\from \HH\to \R$, then
\begin{align}\label{eq:eo-product}
  (fg)^\sE & = [(f^\sE + f^\sO)(g^\sE + g^\sO)]^\sE = f^\sE g^\sE + f^\sO g^\sO\\
\notag  (fg)^\sO & = f^\sE g^\sO + f^\sO g^\sE.
\end{align}

\begin{lemma}\label{lem:horizontality}
  Let $a$ be as above and let $m\from \HH \to \R$ be a smooth function that is constant on cosets of $\paramY$. Let $v\in B(\zero, \sqrt{\rho})$. Then
  $$|m(v) - m(\theta(v))| \lesssim \|v\|_\Kor \|\nabla_a m\|_\infty + A^{-1} \|v\|_\Kor \|Z m\|_\infty.$$
  
  If $\rho>A^2$ and $m=Z a$, $m=b$, or $m=Z b$, then for any $v\in \HH$, we have $|m^\sE(v)|\le A^{-1}$ and 
  $$|m^\sO(v)| = \frac{1}{2}|m(v) - m(\theta(v))| \lesssim A^{-2} \|v\|_\Kor.$$
\end{lemma}
\begin{proof}
  Let $p = \Psi_\alpha(\zero)$ and let $\sigma = \nabla_\alpha \alpha(\zero)$.   Let $W=X+\sigma Y$ and $P=\langle W,Z\rangle$. We first bound the distance from $\Gamma_a$ to $P$. By Lemma~\ref{lem:a-Taylor} and Lemma~\ref{lem:gamma-deriv-bounds}, for $q\in \HH$,
  $$|\alpha(q) - \alpha(\zero) - \sigma x(q)| \lesssim A^{-1} \rho^{-1} d_\Kor(p,q)^2.$$
  
  Let $u \in B(\zero,3\sqrt{\rho})$ and let $u' = Y^{(\alpha + \tau\gamma)(\zero)} u$. Then 
  $$d_\Kor(p,u') = d_\Kor(Y^{\alpha(\zero)}, Y^{(\alpha + \tau\gamma)(\zero)} u) \le |\tau\gamma(\zero)| + \|u\|_\Kor \lesssim A^{-1} + \sqrt{\rho}\lesssim \sqrt{\rho}.$$
  We have
  $$a(u) = \alpha(u') - \alpha(\zero) + \tau(\gamma(u') - \gamma(\zero)),$$ 
  so
  \begin{multline}\label{eq:plane-dist}
    |a(u) - \sigma x(u)| = |\alpha(u') - \alpha(\zero) - \sigma x(u)| + 2|\tau|\|\gamma\|_\infty \\ \lesssim A^{-1} \rho^{-1} d_\Kor(p, u')^2 + A^{-1} \lesssim A^{-1}.
  \end{multline}
  
  Recall that for all $q\in \HH$, we have $(\nabla_a)_q = X_q + (y(q) - a(q)) Z_q$. If $u\in P$, then $y(u)=\sigma x(u)$, so
  $$W_u = (\nabla_a)_u - (\sigma x(u) - a(u)) Z_u + \sigma Y_u.$$
  Let $m\from \HH\to \R$ be a smooth function which is constant on cosets of $\paramY$. Then $Y m=0$, so for $u\in P$,
  $$Wm(u) = \nabla_a m(u) - (\sigma x(u) - a(u)) Z m(u).$$
  By \eqref{eq:plane-dist},
  \begin{equation} \label{eq:wmu}
    |W m(u)| \lesssim \|\nabla_a m\|_\infty + A^{-1} \|Z m\|_\infty.
  \end{equation}
  
  Let $w,z$ be such that $\Pi_P(v)= W^w Z^z$; note that $|w|\le \|v\|_\Kor \le \sqrt{\rho}$, so
  $$\|\Pi_P(v)\|_\Kor \le \|v\|_\Kor +  |y(v) - \sigma w| \le 3 \sqrt{\rho}.$$
  Then, by the Mean Value Theorem and \eqref{eq:wmu}, 
  $$|m(v) - m(\theta(v))| = |m(W^w Z^z) - m(W^{-w}Z^z)| \lesssim \|v\|_\Kor (\|\nabla_a m\|_\infty + A^{-1} \|Z m\|_\infty),$$
  as desired.
  
  Finally, if $m=Z a$, $m=b$, or $m=Z b$ and $\rho > A^2$ then $\|m\|_\infty\lesssim A^{-1}$ and $\|\nabla_a m\|_\infty + A^{-1} \|Z m\|_\infty\lesssim A^{-2}$. Therefore, 
  $$|m(v) - m(\theta(v))| \lesssim A^{-2} \|v\|_\Kor$$
  for all $v\in B(\zero,A)$ and 
  $$|m(v) - m(\theta(v))| \lesssim \|m\|_\infty \lesssim A^{-2} \|v\|_\Kor$$
  for all $v\not\in B(\zero,A)$.
\end{proof}

Finally, we bound functions of the form $N_{T}(v) := T(\overline{v})$ when $T$ is a homogeneous kernel.
\begin{lemma}\label{lem:ker-on-graph}
  Let $v\in \HH$, $v\ne \zero$. Let $T$ be a smooth $k$--homogeneous kernel on $\HH$. Then $|N_T(v)|\lesssim_T \|v\|_\Kor^k$ and 
  $$|N_T(\theta(v)) - (-1)^{k} N_T(v)|\lesssim_T A^{-1} \|v\|_\Kor^{k+1}.$$
\end{lemma}
\begin{proof}
  Let $\lambda(v) = \nabla_a a(\zero) \cdot x(v)$ be the function whose graph is the vertical plane tangent to $\Gamma_a$ at $\zero$. By Lemma~\ref{lem:ab-deriv-bounds} and Lemma~\ref{lem:a-Taylor}, 
  $a(v) = \lambda(v) + O(A^{-1}\|v\|_\Kor^2)$, and $|a(v) + a(\theta(v))|\lesssim A^{-1}\|v\|_\Kor^2$. 
  
  As in the proof of Lemma~\ref{lem:riesz-converge},
  $$\theta(\overline{v}) = \overline{\theta(v)} Y^{- a(v) - a(\theta(v))},$$
  and any point $w$ on the segment from $\theta(\overline{v})$ to $\overline{\theta(v)}$ satisfies $\|w\|_\Kor\approx \|v\|_\Kor$. The mean value theorem and the $(k-1)$--homogeneity of $\YL T$ imply that
  \begin{multline*}
    |N_T(\theta(v)) - (-1)^{k} N_T(v)| = \big|T(\overline{\theta(v)}) - T(\theta(\overline{v}))\big| \\ \lesssim_T |a(v) + a(\theta(v))| \|v\|_\Kor^{k-1} \lesssim A^{-1}\|v\|_\Kor^{k+1},
  \end{multline*}
  as desired.
\end{proof}

Now we prove Lemma~\ref{lem:unif-sec-lg}.
\begin{proof}[Proof of Lemma \ref{lem:unif-sec-lg}]
  As above, we let $a = a_\tau$ and $b = b_\tau$. Let
  $$\phi_t(w) = Y^{-tb(\zero)} \overline{w} Y^{tb(w)},$$
  so that  $\Glg_{\tau,R}(t) = \int_{A^t_{1,R}} \widehat{K}(\phi_t(w)) \ud w.$
  
  As in the proof of Lemma~\ref{lem:grrab-large}, we define $M_r(x) = \frac{1}{4} \sqrt{r^4 - x^4}$ and 
  \begin{multline*} 
    A_{r,R}(x,t) = [-M_{R}(x) - t b(\zero) x, M_{R}(x) - t b(\zero) x] \\ \setminus (-M_r(x) - t b(\zero) x, M_r(x) - t b(\zero) x)
  \end{multline*}
  so that $A_{r,R}^{t} = \{(x,0,z) : z\in A_{r,R}(x,t)\}$ and \begin{align*}
    \Glg_{\tau,R}(t) = \int_{-R}^R \int_{A_{1,R}(x,0,t)} \widehat{K}(\phi_t(x,0,z)) \ud z \ud x.
  \end{align*}

  Taking the derivative with respect to $t$ gives
  \begin{align*}
    \left(\Glg_{\tau,R}\right)'(t) &= \int_{-R}^R \int_{A_{1,R}(x,t)} \partial_t \left[\widehat{K}(\phi_t(x,0,z))\right] \ud z \ud x \\
    &\qquad - \int_{-R}^R b(\zero) x  \widehat{K}(\phi_t(x,0,u)) \big|_{u=-M_R(x) - b(\zero) tx}^{M_R(x) - b(\zero) tx} \ud x \\
    &\qquad + \int_{-1}^1 b(\zero) x  \widehat{K}(\phi_t(x,0,u)) \big|_{u=-M_1(x) - b(\zero) tx}^{M_1(x) - b(\zero) tx} \ud x =: J_1^R - J_2^R + J_2^1,
   \end{align*} 
  where $f(u)\big|_{u=a}^b$ or $[f(u)]_{u=a}^b$ denotes $f(b)-f(a)$. We have
  \begin{multline*}
    \frac{\partial J_1^R}{\partial t}\bigg|_{t=0}
     = \int_{A_{1,R}} \partial_t^2 \left[\widehat{K}(\phi_t(w))\right]_{t = 0} \ud w \\
     - \left[\int_{-r}^r b(\zero)x \partial_t \left[\widehat{K}(\phi_t(x,0,u))\right]_{t=0} \big|_{u=-M_r(x)}^{M_r(x)} \ud x\right]_{r=1}^{R} =:
    I_1^R - I^R_2 + I^1_2.
  \end{multline*}
  By \eqref{eq:a-tau-t}, if $w=(x, 0, M)\in V_0$, then
  $$\phi_t(x, 0, M - b(\zero) t x) = \phi_t(w_t) = \Psi_{a_{\tau+t}}(w),$$
  so 
  $$J_2^r = \int_{-r}^r b(\zero) x \widehat{K}(\Psi_{a_{\tau+t}} (x,0,z))\big|_{z=-M_r(x)}^{M_r(x)} \ud x,$$
  and
  $$\frac{\partial J_2^r}{\partial t} \bigg|_{t=0} =  \int_{-r}^r b(\zero) x \partial_t \left[\widehat{K}(\Psi_{a_{\tau+t}} (x,0,z))\right]_{t=0} \big|_{z=-M_r(x)}^{M_r(x)} \ud x =: I_3^r.$$

  Then
  $$\left(\Glg_{\tau,R}\right)''(0) = I_1^R - I^R_2 + I^1_2 + I_3^R - I_3^1.$$
  To prove \eqref{eq:unif-sec-lg}, it suffices to show that $|I_j^R - I_j^{R'}| \lesssim A^{-3} R^{-1}$ for all $1\le R < R'$.
  
  The following calculations will be helpful. Let $w\in V_0$. We have $\phi_t(w) = Y^{-tb(\zero)} w Y^{a(w) + tb(w)}$ and $\phi_0(w)=\overline{w}$, so  
  \begin{align}
    \partial_t[\widehat{K}(\phi_t(w))] = -b(\zero) \YR \widehat{K}(\phi_t(w)) + b(w) \YL \widehat{K}(\phi_t(w)). \label{eq:one-deriv}
  \end{align}
  Taking a second derivative gives
  \begin{multline}
    \partial_t^2 [\widehat{K}(\phi_t(w))]_{t=0} = b(\zero)^2 \YR^2 \widehat{K}(\overline{w}) - b(\zero)b(w) (\YR \YL + \YL \YR) \widehat{K}(\overline{w}) \\
    \qquad + b(w)^2 \YL^2 \widehat{K}(\overline{w}). \label{eq:two-deriv}
  \end{multline}
  
  By \eqref{eq:a-tau-t},
  $$a_{\tau + t}(w) = a(w_t) + tb(w_t) - tb(\zero),$$
  so
  \begin{equation}\label{eq:pt-a-tau-t}
    \partial_t [a_{\tau + t}(w)] = -b(\zero) x \partial_z[a + t b](w_t) + b(w_t) - b(\zero)
  \end{equation}
  and 
  \begin{equation}
    \partial_t [\widehat{K}(\Psi_{a_{\tau + t}}(w)]]_{t=0} 
    =\YL \widehat{K}(\overline{w})(b(w) - b(\zero) - b(\zero) x  \partial_z a(w)).    
    \label{e:boundary-deriv}
  \end{equation}

  \smallskip

  \textbf{Bounding $I_1^R$}: By \eqref{eq:two-deriv}, we have that
  \begin{align*}
    I_{1}^R &= \int_{A_{1,R}} b(\zero)^2 \YR^2 \widehat{K}(\overline{w}) \ud w  + \int_{A_{1,R}} b(w)^2 \YL^2 \widehat{K}(\overline{w}) \ud w \\
    &\qquad- \int_{A_{1,R}} b(\zero)b(w) (\YR \YL + \YL \YR) \widehat{K}(\overline{w}) \ud w \\
    &= \sum_i \int_{A_{1,R}} b_i(w) N_{T_i}(w) \ud w,
  \end{align*}
  where $b_1(w) = b(\zero)^2$, $b_2(w) =b(w)^2$, $b_3(w)=b(\zero)b(w)$, and the $T_i$'s are smooth $(-5)$--homogeneous kernels. By \eqref{eq:eo-product} and Lemma~\ref{lem:horizontality}, for any $w\in \HH$, we have
  \begin{align*}
    |b^\sE(w)| & \le \|b\|_\infty \lesssim A^{-1} & |b^\sO(w)| & \lesssim A^{-2} \|w\|_\Kor\\
    |b_i^\sE(w)| & \le \|b\|_\infty^2 \lesssim A^{-2} & |b_i^\sO(w)| & \lesssim \|b\|_\infty |b^\sO(w)| \lesssim A^{-3} \|w\|_\Kor.
  \end{align*}
  By Lemma~\ref{lem:ker-on-graph},
  \begin{align*}
    |N_{T_i}^\sE(w)| & \le A^{-1} \|w\|_\Kor^{-4} & |N_{T_i}^\sO(w)| & \lesssim \|w\|_\Kor^{-5}.
  \end{align*}
  Therefore, by the symmetry of $A_{1,R}$ and Lemma~\ref{lem:polar},
  \begin{multline*}
    |I_{1,R}|\le \sum_i \left|\int_{A_{1,R}} (b_i N_{T_i})^\sE \ud w\right| = \sum_i \left|\int_{A_{1,R}} b_i^\sE N_{T_i}^\sE + b_i^\sO N_{T_i}^\sO \ud w\right| \\
    \lesssim \int_{A_{1,R}} A^{-3} \|w\|_\Kor^{-4}\ud w \lesssim \int_1^R A^{-3} \kappa^{-2} \ud \kappa \lesssim A^{-3}R^{-1}.
  \end{multline*}

  \smallskip
  
  \textbf{Bounding $I_2^r$}: Let $w_+(x) = (x,0,M_r(x))$ and $w_-(x) =  (x,0,M_r(x))$. By \eqref{eq:one-deriv}, we have $I_2^r = \int_{-r}^r h_+(x) - h_-(x)\ud x$, where
  $$h_\pm(x) = b(\zero) x \big(b N_{\YL \widehat{K}} - b(\zero) N_{\YR \widehat{K}}\big)(w_\pm(x))$$
  $$h_\pm(x) = b(\zero) x \big(b(w_\pm(x)) N_{\YL \widehat{K}}(w_\pm(x)) - b(\zero) N_{\YR \widehat{K}}(w_\pm(x))\big)$$
  Let $h^\sE_\pm(x) = \frac{1}{2} (h_\pm(x) + h_\pm(-x))$; then $I_2^r = \int_{-r}^r h^\sE(x)\ud x$. Since $\|w_\pm(x)\|_\Kor = r$,
  \begin{align*}
    |h_\pm^\sE(x)| & = \left|b(\zero) x \big(b N_{\YL \widehat{K}} - b(\zero) N_{\YR \widehat{K}}\big)^\sO(w_\pm(x))\right| \\
    & = \left|b(\zero) x \big(b^\sE N^\sO_{\YL \widehat{K}} + b^\sO N^\sE_{\YL \widehat{K}} - b(\zero) N^\sO_{\YR \widehat{K}}\big)(w_\pm(x))\right|\\
    & \lesssim A^{-1}r \left(A^{-1} \cdot A^{-1} r^{-3} + A^{-2} r \cdot r^{-4} + A^{-1} \cdot A^{-1} r^{-3}\right) \\
    & \lesssim A^{-3} r^{-2},
  \end{align*}
  and $|I_2^r|\lesssim A^{-3} r^{-1}$.
  
  \smallskip
  
  \textbf{Bounding $I_3^r$}:
  By \eqref{e:boundary-deriv}, $I_3^r=\int_{-r}^r  k_+(x)-k_-(x) + l_+(x) - l_-(x)\ud x$, where
  \begin{align*}
    k_{\pm}(x) & = b(\zero) x N_{\YL \widehat{K}}(w_\pm(x)) \big(b(w_\pm(x)) - b(\zero)\big) \\
    l_{\pm}(x) & = b(\zero)^2 x^2 \partial_z a(w_\pm(x)) N_{\YL \widehat{K}}(w_\pm(x))
  \end{align*}
  
  Since $\|w_\pm(x)\|_\Kor = r\ge 1$ and $\YL \widehat{K}$ is $(-4)$--homogeneous, 
  $$|l_{\pm}(x)| \lesssim A^{-2} r^2 A^{-1} r^{-4} = A^{-3} r^{-2}.$$
  By Lemma~\ref{lem:horizontality}, for any $w\in \HH$, we have $|(\partial_z a)^\sE(w)| \lesssim A^{-1}$ and $|\partial_z a^\sO(w)| \lesssim A^{-2} \|w\|_\Kor$. 
  Therefore,
  $$k^\sE_{\pm}(x) = b(\zero) x\left(N^\sE_{\YL \widehat{K}}(w_\pm(x)) b^\sO(w_\pm(x)) + N^\sO_{\YL \widehat{K}}(w_\pm(x)) \big(b^\sE(w_\pm(x)) - b(\zero)\big)\right)$$
  and
  $$|k^\sO_{\pm}(x)|\lesssim A^{-1} r \cdot r^{-4} \cdot A^{-2} r + A^{-1} r \cdot A^{-1} r^{-3} \cdot A^{-1} \lesssim A^{-3}r^{-2}.$$
  Therefore, 
  $$|I_3^r|= \left|\int_{-r}^r k^\sE_+(x)-k^\sE_-(x) + l_+(x) - l_-(x)\ud x\right| \lesssim A^{-3} r^{-1}.$$
  
  Thus, for any $R'>R\ge 1$, we have 
  $$\left|\left(\Glg_{\tau,R'}\right)''(0) - \left(\Glg_{\tau,R}\right)''(0)\right| \le \left|I_1^{R'} - I_1^R\right| + \left|I_2^{R'} - I_2^R\right| + \left|I_3^{R'} - I_3^R\right| \lesssim A^{-3}R^{-1}.$$
  This proves the lemma.
\end{proof}

\begin{proof}[Proof of Lemma~\ref{lem:unif-sec-sm}]
  Let $0<r<1$ and recall that
  $$\Gsm_{\tau,r}(t) := \int_{A_{r,1}} \widehat{K}(\Psi_{a_{\tau + t}}(v)) \ud v.$$
  We claim that 
  $$\left|(\Gsm_{\tau,r})''(0) - (\Gsm_{\tau,r'})''(0)\right| \lesssim A^{-3}r$$
  for all $0<r'<r<1$.
  
  As above, we abbreviate $a=a_\tau$ and $b=b_\tau$. Recall that $a(\zero)=0$.
  For $v=(x,y,z)$, let
  $$q_1(v) := \partial_t [a_{\tau + t}(v)]_{t=0} \stackrel{\eqref{eq:pt-a-tau-t}}{=} -b(\zero) x \partial_z a (v) + b(v) - b(\zero)$$
  and
  $$q_2(v) := \partial^2_t [a_{\tau + t}(v)]_{t=0} = b(\zero)^2x^2\partial^2_z a(v)-2 b(\zero) x \partial_z b(v).$$
  Then
  \begin{multline*}
    \frac{\ud^2}{\ud t^2} \widehat{K}(vY^{a_{\tau + t}(v)})\Big|_{t=0} = \frac{\ud}{\ud t} \partial_t [a_{\tau + t}(v)] \YL \widehat{K}(vY^{a_{\tau + t}(v)})\Big|_{t=0} \\
    = \partial^2_t [a_{\tau + t}(v)]_{t=0} \YL \widehat{K}(\overline{v}) + \big(\partial_t [a_{\tau + t}(v)]_{t=0}\big)^2 \YL^2 \widehat{K}(\overline{v}),
  \end{multline*}
  and
  $$(\Gsm_{\tau,r})''(0) = \int_{A_{r,1}} q_2(v) N_{\YL \widehat{K}}(v) + q_1^2(v) N_{\YL^2\widehat{K}}(v) \ud v.$$

  As above, we decompose these into odd and even terms. Let $b_0(v) = b(v) - b(\zero)$ and let $\kappa=\|v\|_\Kor$. By Lemma~\ref{lem:ab-deriv-bounds} and Lemma~\ref{lem:a-Taylor},
  \begin{equation}\label{eq:b0-approx-small}
    b_0(v) = x(v)\nabla_a b(\zero) + O(A^{-1}\kappa^2 ) = O(A^{-2}\kappa  + A^{-1}\kappa^2 ) = O(A^{-1}\kappa ),
  \end{equation}
  so $|b_0^\sE(v)|\lesssim A^{-1}\kappa^2 $ and $|b_0^\sO(v)|\lesssim A^{-1}\kappa $. Therefore, by Lemma~\ref{lem:horizontality},
  $$|q_1^\sE(v)|\le b(\zero) x |\partial_z a^\sO(v)| + |b_0^\sE(v)| \lesssim A^{-1} \kappa \cdot A^{-2}\kappa + A^{-1}\kappa^2\lesssim A^{-1}\kappa^2$$
  and
  $$|q_1^\sO(v)|\le b(\zero) x |\partial_z a^\sE(v)| + |b_0^\sO(v)| \lesssim A^{-1} \kappa \cdot A^{-1} + A^{-2}\kappa\lesssim A^{-2}\kappa.$$
  Therefore, $|(q_1^2)^\sO(v)| = 2|q_1^\sE(v)q_1^\sO(v)|\lesssim A^{-3}\kappa^3$ and
  $$|(q_1^2)^\sE(v)| = |q_1^\sE(v)^2+ q_1^\sO(v)^2|\lesssim A^{-2}\kappa^2.$$
  
  For $q_2$, on one hand,
  $$|q_2(v)| \le b(\zero)^2 x^2\cdot\|\partial^2_z a\|_\infty + 2|b(\zero)x| \cdot \|\partial_zb\|_\infty \lesssim A^{-2} \kappa^2 \cdot A^{-1} + A^{-1} \kappa \cdot A^{-1} \lesssim A^{-2}\kappa,$$
  so $|q_2^\sO(v)| \lesssim A^{-2}\kappa$. On the other hand, by Lemma~\ref{lem:horizontality},
  $$|q_2^\sE(v)|\le b(\zero)^2 x^2\|\partial^2_z a\|_\infty + |2 b(\zero) x (\partial_z b)^\sO(v)|\lesssim A^{-2} \kappa^2\cdot A^{-1} + A^{-1}\kappa \cdot A^{-2}\kappa\lesssim A^{-3}\kappa^2.$$
  
  Since $\YL\widehat{K}$ is $(-4)$--homogeneous and $\YL^2\widehat{K}$ is $(-5)$--homogeneous, we can use Lemma~\ref{lem:ker-on-graph} to bound $N_{\YL\widehat{K}}$ and $N_{\YL^2\widehat{K}}$. Thus
  \begin{align*}
    |(\Gsm_{\tau,r})''(0)&-(\Gsm_{\tau,r'})''(0)| = \left|\int_{A_{r',r}} (q_2 N_{\YL \widehat{K}})^\sE + (q_1^2 N_{\YL^2\widehat{K}})^\sE \ud v\right| \\
    & \le \left|\int_{A_{r',r}} q_2^\sE N_{\YL \widehat{K}}^\sE + q_2^\sO N_{\YL \widehat{K}}^\sO + (q_1^2)^\sE N_{\YL^2\widehat{K}}^\sE + (q_1^2)^\sO N_{\YL^2\widehat{K}}^\sO \ud v\right| \\
    & \le \int_{A_{r',r}} A^{-3}\kappa^2\cdot \kappa^{-4} + A^{-2} \kappa\cdot A^{-1}\kappa^{-3} + A^{-2}\kappa^{2} \cdot A^{-1}\kappa^{-4} + A^{-3}\kappa^{3}\cdot \kappa^{-5}\ud v\\
    & \approx \int_{r'}^r A^{-3}\kappa^{-2} \cdot \kappa^2 \ud \kappa \le A^{-3} r,
  \end{align*}
  where we used Lemma~\ref{lem:polar} to replace $\ud v$ by $\kappa^2\ud \kappa$. This proves the lemma.
\end{proof}

\section{Proof of Theorem \ref{mainthmintro}}
\label{sec:proofofmainthm}
In this section we will finally finish the proof of Theorem~\ref{mainthmintro}. First, recall that for any intrinsic Lipschitz function $\phi\from \HH\to \R$ and any function $g\from \HH\to \R$ which is constant on cosets of $\paramY$, we have  defined a parametric version of the Riesz transform by
$$\Rz_{\phi} g(p) := \tilde{T}^\mathsf{R}_\phi g(p) = \pv(\Psi_\phi(p)) \int_{\Psi_\phi(p) V_0} \widehat{\mathsf{R}}(\Psi_\phi(p)^{-1}\Psi_{\phi}(v)) g(v)\ud v.$$
Note that when $\phi$ is smooth and bounded and $p\in \Gamma_\phi$, Proposition~\ref{prop:annulus-change} implies that
$$\Rz_{\phi} g = T^\mathsf{R}[g\ud \eta_\phi].$$

In order to bound the $L_2$ norm of $\Rz$ on $\Gamma_{\phi_A}$ (where $\phi_A$ is as in Proposition \ref{prop:mainBounds-tilde-T}) we will need the following lemma, which allows us to replace $\one$ by an $L_2$ function.

\begin{lemma} \label{l:set-pv}
  Let $\phi$ satisfy the hypotheses of Proposition~\ref{prop:annulus-change}. Let $E \subset F$ be two bounded subsets of $V_0$ so that $d(E,F^c) > 0$. Then there is a $C>0$ depending only on $\phi$, $E$, and $F$ such that for every $p \in \Psi_\phi(E)$, the principal value $\tilde{T}_\phi \one_{F\paramY}(p)$ exists and satisfies
  $$\left| \tilde{T}_\phi\one_{F\paramY} (p) - \tilde{T}_\phi \one (p) \right| < C.$$
\end{lemma}
\begin{proof}
  Since we know $\tilde{T}_\phi \one (p)$ exists by Lemma~\ref{lem:riesz-converge}, it suffices to show that there is a $C>0$ such that for all $p\in E$,  $\tilde{T}_\phi[\one_{F\paramY}-\one](p)$ exists and 
  $$\left|\tilde{T}_\phi\left[\one - \one_{F\paramY}\right](p)\right| = \left|\tilde{T}_\phi \one_{\HH\setminus F\paramY} (p)\right| < C.$$
  
  Let $\chi = \one_{\HH\setminus F\paramY}$.
  By compactness and the boundedness of $F$, there is an $0<\epsilon<1$ such that for all $p\in E$ we have 
  $$\Pi(B_\epsilon(p) \cap pV_0) \subset \Pi(F) \subset \Pi(B_{\epsilon^{-1}}(p) \cap p V_0).$$
  Therefore, for $r<\epsilon$ and $R>\epsilon^{-1}$, we have $\chi=0$ on $A_{r,\epsilon}(p)$ and $\chi=1$ on $A_{\epsilon^{-1},R}(p)$, so   
  $$\tilde{T}_{\phi}[\chi](p) = \lim_{\substack{r\to 0\\R\to \infty}}\tilde{T}_{\phi;r,R}[\chi](p)  = \tilde{T}_{\phi;\epsilon,\epsilon^{-1}}[\chi](p) + \lim_{R\to \infty} \tilde{T}_{\phi;\epsilon^{-1},R} [\one](p).$$

  By Lemma~\ref{lem:riesz-converge}, this limit exists and satisfies $|\tilde{T}_{\phi}[\chi](p)| \lesssim \log \epsilon + \epsilon$ for all $p\in E$. This proves the lemma.
\end{proof}

Given a set $E \subseteq \HH$, we define the $L_2$ norm $\|\cdot\|_E := \|\cdot\|_{L_2(\mu|_{E})}$.

Let $W = [-1,2] \times \{0\} \times [-1,2] \subset V_0$ and let $\overline{W}=W\paramY$. Then $U$ and $W$ satisfy Lemma~\ref{l:set-pv}, so Proposition~\ref{prop:mainBounds-tilde-T} implies that
\begin{equation}\label{eq:set-pv-bound}
  \|\Rz_{\phi_A}{\bf 1}_{\overline{W}}\|_U \geq \|\Rz_{\phi_A} {\bf 1}\|_U - \|\Rz_{\phi_A} {\bf 1} - \Rz_{\phi_A}{\bf 1}_{\overline{W}}\|_U \geq cA - C \gtrsim A
\end{equation}
when $A$ is sufficiently large, and thus that
$$\|\Rz_{\phi_A} {\bf 1}_{\overline{W}}\|_{L_2(\eta_{\phi_A})}  \ge \|\Rz_{\phi_A}{\bf 1}_{\overline{W}}\|_U \gtrsim A.$$
Since ${\bf 1}_{\overline{W}}\in L_2(\eta_{\phi_A})$, the operator norm of $\Rz_{\phi_A}$ on $L_2(\eta_\phi)$ goes to infinity with $A$. By gluing together graphs with different values of $A$, we can construct a single intrinsic Lipschitz function $\phi$ such that $\Rz_{\phi}$ is unbounded on $L_2(\eta_\phi)$.

\begin{proof}[Proof of Theorem~\ref{mainthmintro}]
  For $x\in V_0$, $r>0$, let $\tau_{x,r}\from \HH\to \HH$ be the affine transform $\tau_{x,r}(v) = x s_r(v)$. Note that $\tau_{x,r}(V_0)=V_0$.
  Let $x_1,x_2,\dots \in V_0$ and let $r_1,r_2,\dots>0$ so that the subsets $W_n=\tau_{x_n,r_n}(W)$ are disjoint subsets of $W$. Let $U_n = \tau_{x_n,r_n}(U)$ and let $\overline{W}_n=W_n\paramY$.
  
  Let $\phi_n$ be as in Proposition~\ref{prop:mainBounds-tilde-T} and let 
  $$\hat{\phi}_n(u) = r_n \phi_n(\tau_{x_n,r_n}^{-1}(u))$$
  so that $\Gamma_{\hat{\phi}_n} = \tau_{x_n,r_n}(\Gamma_{\phi_n})$. Note that $\hat{\phi}_n(v)=0$ for $v\not \in U_n \paramY$.
  Let $f \from V_0 \to \R$ be the function 
  $$f(v)=\begin{cases}
    \hat{\phi}_n(v) & v \in U_n\paramY \\
    0 & \text{otherwise.}
    \end{cases}
  $$
  Then $f$ is an intrinsic Lipschitz function supported in $\overline{W}$. 
  
  Since $f|_{\overline{W}_n} = \hat{\phi}_n|_{\overline{W}_n}$, we have $\Rz_f {\bf 1}_{\overline{W}_n}(v) = \Rz_{\hat{\phi}_n} {\bf 1}_{\overline{W}_n}(v)$ for all $v\in \overline{W}_n$. By the translation- and scale-invariance of the Riesz kernel, for all $v\in V_0$,
  \begin{align}
    \Rz_{\hat{\phi}_n} \one_{\overline{W}_n}(v) = \Rz_{\phi_n} \one_{\overline{W}}(\tau_{x_n,r_n}^{-1}(v)). \label{e:riesz-affine-inv}
  \end{align}
  Since $(\tau_{x,r})_*(\mu)=r_n^{-3}\mu$, this implies
  \begin{equation*}
      \begin{split}
            \|\Rz_f {\bf 1}_{\overline{W}_n}(v)\|_{L_2(\eta_f)} &\ge \|\Rz_f {\bf 1}_{\overline{W}_n}(v)\|_{U_n} = r_n^{-\frac{3}{2}} \|\Rz_{\phi_n} {\bf 1}_{\overline{W}}(v)\|_U \\ &\stackrel{\eqref{eq:set-pv-bound}}{\gtrsim} r_n^{-\frac{3}{2}} n\approx n\|{\bf 1}_{\overline{W}_n}(v)\|_{L_2(\eta_f)}
      \end{split}
  \end{equation*}
  for all sufficiently large $n$. Thus, $\Rz_{f}$ is unbounded on $L_2(\eta_f)$. 
\end{proof}

\clearpage
\appendix

\section{Intrinsic Sobolev spaces and derivative bounds on $\Gamma_{f_i}$}
\label{ap:A}

In this section, we bound the derivatives of the functions constructed in Section~\ref{sec:construction}. 
We first introduce some Sobolev spaces $W_{i,d}$ and $W'_{i,d}$ that consist of functions on $\HH$.

Recall that for vector fields $V_1,\dots, V_k$, we let $\{V_1,\dots,V_k\}^l$ denote the set of words of length $l$ and let $\{V_1,\dots,V_k\}^*$ denote the set of words of any length. We let $\#V_i(D)$ denote the number of occurrences of $V_i$ in $D$. For $i\ge 0$, let $r_i = A^{-1} \rho^{-i}$, and let $f_i$ and $\eta_i$ be as in Section~\ref{sec:construction}. Let $\partial_i = \nabla_{f_i}$ and let
\begin{align}
  \hat{Z}_i &= r_i^{2} Z &
  \hat{\partial}_i &= A r_i \partial_i &
  \hat{\nu}_i &= A r_i^{-1} \nu_i &
  \hat{f}_i &= A r_i^{-1} f_i.\label{eq:def-rescaled}
\end{align}
The first two scaling factors correspond to the width and height of the pseudoquads in the construction of $\nu_i$; the third and fourth normalize $\nu_i$ so that $\|\hat{\nu}_i\|_\infty\approx 1$.

For any smooth function $g\from \HH\to \R$ which is constant on cosets of $\paramY$, let
$$\|g\|_{W_{i,d}} = \max_{\substack{D\in \{\hat{\partial}_i, \hat{Z}\}^* \\ \ell(D)\le d}} \|D g\|_\infty$$
and
$$\|g\|_{W'_{i,d}} = \max_{\substack{D\in \{\hat{\partial}_i, \hat{Z}\}^* \\ \ell(D)\le d \\ D\not \in \{\id, \hat{\partial}_i\}}} \|D g\|_\infty.$$

In this section, we will prove the following proposition, which is equivalent to Lemma~\ref{lem:word-sup-bound}.
\begin{prop}\label{prop:deriv-bounds}
  For any $d>0$, if $\rho>1$ is sufficiently large, then for all $i$, 
  $$\|\nu_i\|_{W_{i,d}} \lesssim_d A^{-1} r_i$$
  and
  $$\|f_i\|_{W_{i,d}'} \lesssim_d A^{-1} \rho^{-1} r_i.$$
\end{prop}
Equivalently, $\|\hat{\nu}_i\|_{W_{i,d}} \lesssim_d 1$ and $\|\hat{f}_i\|_{W_{i,d}} \lesssim_d \rho^{-1}$.

The proof of Proposition~\ref{prop:deriv-bounds} can be broken into two parts. In the first part, we bound $D\nu_i$ in the case that $D=\hat{Z}_i^k\hat{\partial}_i^j$.
\begin{lemma}\label{lem:standard-form-bounds}
  Given $d \geq 2$, there exists $\rho_0 > 0$ so that if $\rho \geq \rho_0$, then
  \begin{align}
    \| \hat{Z}_i^k \hat{f}_i \|_\infty &\le 2 \rho^{-1} & & \forall i \geq 0, 1\le k \le d \label{e:f-sobolev-bound} \\
    \| \hat{Z}_i^k \hat{\partial}_i^j \hat{\nu}_i \|_\infty &\lesssim_j 1  & & \forall i \geq 0, 0\le j \le d, 0\le k \le d. \label{e:eta-sobolev-bound}
  \end{align}
\end{lemma}

In the second part, we use Lemma~\ref{lem:standard-form-bounds} as part of an inductive argument. First, we bound $\hat{\nu}_i$ in terms of $\|\hat{f}_i\|_{W'_{i,d}}$.
\begin{lemma}\label{lem:induct-eta}
  For any $d>0$, there is a $c_1>1$ such that for any $i\ge 0$, if $\rho>1$ is sufficiently large and $\|\hat{f}_i\|_{W'_{i,d}} < 1$, then 
  \begin{equation}\label{eq:induct-eta}
    \|\hat{\nu}_i\|_{W_{i,d}} \le c_1.
  \end{equation}
\end{lemma}

To bound $\|\hat{f}_{i+1}\|_{W_{i+1,d}'}$, we compare $\|\cdot\|_{W_{i+1,d}'}$ and $\|\cdot\|_{W_{i,d}}$.
\begin{lemma}\label{lem:induct-f}
  For any $d>0$, if $\rho$ is sufficiently large, then for any $i$ and any $g\in W_{i,d}'$,
  \begin{equation}\label{eq:induct-f-1}
    \|g\|_{W_{i+1,d}'} \lesssim_d (1 + \|\hat{\nu}_i\|_{W_{i,d}})^d \rho^{-2} \|g\|_{W_{i,d}'}.
  \end{equation}
  In particular, if $\hat{\nu}_i$ satisfies \eqref{eq:induct-eta}, then there is a $c_2>0$ depending only on $d$ such that $\|g\|_{W_{i+1,d}'} \le c_2 \rho^{-2} \|g\|_{W_{i,d}'}.$
\end{lemma}

Given these lemmas, we prove Proposition~\ref{prop:deriv-bounds} by induction on $i$.
\begin{proof}[Proof of Proposition~\ref{prop:deriv-bounds}]
  Let $c_1>1$ and $c_2$ be as in Lemma~\ref{lem:induct-eta} and Lemma~\ref{lem:induct-f}. Let $c_3 = 2 c_1 c_2$. We claim that if $\rho$ is sufficiently large, then $\|\hat{\nu}_i\|_{W_{i,d}} \le c_1$ and $\|\hat{f}_i\|_{W_{i,d}'} \le c_3 \rho^{-1}$ for all $i\ge 0$. We suppose that $\rho > c_3$ so that this bound implies $\|\hat{f}_i\|_{W_{i,d}'} < 1$.

  We proceed by induction on $i$. When $i=0$, we have $f_0=0$, so $\|\hat{f}_0 \|_{W'_{0,d}}=0$. 
  Suppose that $i\ge 0$ and $\|\hat{f}_i\|_{W_{i,d}'} \le c_3 \rho^{-1} < 1$. Lemma~\ref{lem:induct-eta} implies that $\|\hat{\nu}_{i}\|_{W_{i,d}} \le c_1$. By Lemma~\ref{lem:induct-f},
  \begin{multline*}
    \|\hat{f}_{i+1}\|_{W_{i+1,d}'} \le A r_{i+1}^{-1} c_2 \rho^{-2} \|f_i+\nu_i\|_{W_{i,d}'} \\ \le c_2 \rho^{-1}(\|\hat{f}_i\|_{W_{i,d}'} + \|\hat{\nu}_i\|_{W_{i,d}}) \le c_2 c_3 \rho^{-2} + c_1 c_2 \rho^{-1} \le 2c_1 c_2 \rho^{-1}.
  \end{multline*}
  That is, $\|\hat{f}_{i+1}\|_{W_{i+1,d}'}\le c_3 \rho^{-1}$. By induction, $\|\hat{\nu}_i\|_{W_{i,d}} \le c_1$ and $\|\hat{f}_i\|_{W_{i,d}'} \le c_3 \rho^{-1}$ for all $i\ge 0$.
\end{proof}

In the following subsections, we will prove these lemmas.

\subsection{Proof of Lemma~\ref{lem:standard-form-bounds}}
Recall that in Section~\ref{sec:construction}, we defined $\nu_i$ using maps $R_{i,j}\from [0,A r_i] \times [0,r_i^2] \to Q_{i,j}$ for $j=1,\dots, A^{-1} r_i^{-3}$ and a subset $J_i\subset \{1,\dots, A^{-1} r_i^{-3}\}$. For each $j\in J_i$, $Q_{i,j}$ is a pseudoquad for $\Gamma_{f_{i}}$ and $R_{i,j}$ is the parametrization of $Q_{i,j}$ that satisfies $\partial_s[R_{i,j}(s,t)] = \nabla_{f_i}$ and $\partial_t[R_{i,j}(0,t)] = Z$. The $Q_{i,j}$'s have disjoint interiors and their union is the unit square $U=[0,1]\times\{0\}\times [0,1]$. 

We then define $\nu_i$ and $f_i$ by letting $\kappa\from [0,1]^2\to \R$ be a bump function, 
\begin{equation}\label{eq:def-kappa-2}
  \kappa_{i,j}(R_{i,j}(s,t)) = A^{-1}r_i \kappa(A^{-1}r_i s, r_i^{-2} t),
\end{equation}
$\nu_i=\sum_{j\in J_i} \kappa_{i,j}$, and $f_{i+1}=f_i+\nu_i$. Let $S_{i} = \bigcup_{j\not\in J_i} Q_{i,j}$, so that $f_i|_{S_i} = f_{i+1}|_{S_i}$. Recall that $S_0=\emptyset$ and that $S_i\subset S_{i+1}$ for all $i$.

  We prove Lemma~\ref{lem:standard-form-bounds} by induction on $i$.  We will show that if
  \eqref{e:f-sobolev-bound} holds for some $i\ge 0$, then \eqref{e:eta-sobolev-bound} holds for $i$ and \eqref{e:f-sobolev-bound} holds for $i+1$. Since $\nu_i=0$ and $f_{i+1}=f_i$ on $S_{i}$, it suffices to prove that \eqref{e:eta-sobolev-bound} and \eqref{e:f-sobolev-bound} hold on  $Q_{i,j}$ for all $j\in J_i$. 
 
  First, we restate \eqref{e:eta-sobolev-bound} and \eqref{e:f-sobolev-bound} in terms of \emph{flow coordinates} on $Q_{i,j}$. Let $j\in J_i$ and define coordinates $(s,t)$ on $Q_{i,j}$ by letting $(s(v), t(v)) = R_{i,j}^{-1}(v)$ for all $v\in Q_{i,j}$. Then $\frac{\partial}{\partial s} = \nabla_{f_i} = \partial_i$ on $Q_{i,j}$.
  We define rescaled coordinate systems on $Q_{i,j}$ by $(\hat{s}, \hat{t}) = (A^{-1} r_i^{-1} s, r_i^{-2} t)$ and $(\hat{x}_i, \hat{z}_i) = (A^{-1} r_i^{-1} x, r_i^{-2} z)$, so that $0\le \hat{s}\le 1$ and $0\le \hat{t}\le 1$ on $Q_{i,j}$. 
  Let $\hat{f}_i=A r_i^{-1} f_i$ and $\hat{\nu}_i=A r_i^{-1} \nu_i$ as in \eqref{eq:def-rescaled}.
  Let $\hat{S} := \frac{\partial}{\partial \hat{s}} = \hat{\partial}_i$ and $\hat{T} := \frac{\partial}{\partial \hat{t}}$. 
  By \eqref{eq:def-kappa-2}, we have 
  $$\hat{\nu}_i(R_{i,j}(A r_i \hat{s}, r_i^2\hat{t})) = A r_i^{-1}\kappa_{i,j_0}(R_{i,j}(s,t)) = \kappa(\hat{s}, \hat{t}),$$
  so $\|\hat{T}^k \hat{S}^l \hat{\nu}_i\|_\infty = \|\partial_{lk} \kappa \|_\infty \lesssim_{k,l} 1,$
  where $\partial_{lk} \kappa$ is the appropriate partial derivative of $\kappa$.
  With this notation, \eqref{e:eta-sobolev-bound} and \eqref{e:f-sobolev-bound} can be stated as
  \begin{align}\tag{F$_{i}$}
    \|\hat{Z}_{i}^k \hat{f}_{i}\|_{L_\infty(Q_{i,j})} & \le 2\rho^{-1} &&\text{for any $j\in J_i$ and $1\le k \le d$,}\\
    \tag{H$_{i}$}
    \|\hat{Z}_{i}^k \hat{S}^m \hat{\nu}_i\|_{L_\infty(Q_{i,j})} &\lesssim_{k,m} 1 &&\text{for any $j\in J_i$, $0\le k \le d$, and $0\le m\le d$.}
  \end{align}
  
  To prove (F$_i$) and (H$_i$), we will need some bounds from \cite{NY2}. The bounds in \cite{NY2} apply to $\psi_i$ rather than $f_i$, but for each $v\in U$, there is an $i'\le i$ such that $f_i$ and $\psi_{i'}$ agree on a neighborhood of $v$. Therefore, by Lemma~3.10 of \cite{NY2}, 
  $$\left\|\frac{\partial f_{i}}{\partial z}\right\|_{\infty} \le \max_{i'\le i} \left\|\frac{\partial \psi_{i'}}{\partial z}\right\|_\infty \le \max_{i'\le i} 2\rho^{i'-1} \le 2\rho^{i-1},$$
  and
  \begin{equation}\label{eq:z-t-bound}
    \frac{\partial \hat{z}_i}{\partial \hat{t}} = \left(\frac{\partial \hat{t}}{\partial \hat{z}_i}\right)^{-1} = \frac{\partial z}{\partial t} \in \left[\frac{3}{4},\frac{4}{3}\right].
  \end{equation}
  Thus
  \begin{equation}\label{eq:first-deriv-f}
    \left\|\hat{Z}_i \hat{f}_i\right\|_{\infty} = \frac{A r_i^{-1}}{r_i^{-2}} \left\|\frac{\partial f_i}{\partial z}\right\|_\infty \le 2 r_i A \rho^{i-1} = 2\rho^{-1}.
  \end{equation}
  
  Suppose by induction on $i$ that (F$_{i}$) holds for some $i\ge 0$. Note that $f_0=0$, so (F$_0$) holds. For $i\ge 0$ and $1\le d_0\le d$, let (P$_{i,d_0}$) be the statement
  \begin{equation}\tag{P$_{i,d_0}$}
    \left\|\frac{\partial^k \hat{z}_i}{\partial \hat{t}^k}\right\|_\infty \lesssim \rho^{-1} \text{ for $2\le k \le d_0$.}
  \end{equation}
  We will show that (F$_i$) implies (P$_{i,d}$) and use (P$_{i,d}$) to prove (H$_i$) and (F$_{i+1}$). We must be careful to ensure that the implicit constants in (P$_{i,d}$) and (H$_i$) are independent of $i$.
  
  Suppose that (P$_{i,d_0}$) holds for some $1\le d_0< d$; note that (P$_{i,1}$) is vacuous. In \cite[Lemma 3.10]{NY2}, it was calculated that
  \begin{equation}\label{eq:logarithmic integral}
    \frac{\partial z}{\partial t}=\exp \left(-\int_0^{s} \frac{\partial f_i}{\partial z} \big(R_{i,j}(\sigma, t)\big) \ud \sigma\right)
  \end{equation}
  and
  \begin{equation}\label{eq:logarithmic integral2}
    \frac{\partial^2 z}{\partial t^2} = -\frac{\partial z}{\partial t}\int_0^{s} \frac{\partial}{\partial t}\left[\frac{\partial f_i}{\partial z_i}\right] \ud \sigma,
  \end{equation}
  where the integrand is evaluated at $R_{i,j}(\sigma, t)$. Thus
  \begin{equation}\label{eq:logarithmic integral2 rescaled}
    \frac{\partial^2 \hat{z}_i}{\partial \hat{t}^2} = - r_i^2 \frac{\partial \hat{z}_i}{\partial \hat{t}}\int_0^{\hat{s}} r_i^{-2} \frac{\partial}{\partial \hat{t}}\left[\frac{A^{-1} r_i}{r_i^2}\frac{\partial \hat{f}_i}{\partial \hat{z}_i}\right] A r_i \ud \hat{\sigma} = - \frac{\partial \hat{z}_i}{\partial \hat{t}}\int_0^{\hat{s}}  \frac{\partial}{\partial \hat{t}}\left[\frac{\partial \hat{f}_i}{\partial \hat{z}_i}\right] \ud \hat{\sigma},
  \end{equation}
  where the integrand is evaluated at $R_{i,j}(A r_i \hat{\sigma}, t)$. By the product rule, for $k\ge 2$,
  \begin{align}\label{eq:log-intk}
    \frac{\partial^k \hat{z}_i}{\partial \hat{t}^k} = - 
    \sum_{j=1}^{k-1} \binom{k-2}{j-1} \frac{\partial^{j} \hat{z}_i}{\partial \hat{t}^{j}} \int_0^{\hat{s}} \frac{\partial^{k-j}}{\partial \hat{t}^{k-j}}\left[\frac{\partial \hat{f}_i}{\partial \hat{z}_i}\right]\ud \hat{\sigma}.
  \end{align}
  
  Since $\frac{\partial}{\partial \hat{t}} = \frac{\partial \hat{z}_i}{\partial \hat{t}} \frac{\partial}{\partial \hat{z}_i}$, an inductive argument (or the Fa\`{a} di Bruno formula) shows that there are constants $c_{\mathbf{a},n}$ such that
  $$\frac{\partial^n}{\partial \hat{t}^n} = \left(\frac{\partial \hat{z}_i}{\partial \hat{t}} \right)^n  \frac{\partial^n}{\partial \hat{z}_i^n} + \sum_{j=1}^{n-1} \sum_{\substack{\mathbf{a} \in \N^j, \\ \|a\|_1 = n}} c_{\mathbf{a},n} \left(\prod_{\ell=1}^{j} \frac{\partial^{a_\ell} \hat{z}_i}{\partial \hat{t}^{a_\ell}} \right) \frac{\partial^j}{\partial \hat{z}_i^j}.$$
  Suppose that $n\le d_0$. By (P$_{i,d_0}$) and \eqref{eq:z-t-bound}, $\frac{\partial^{m} \hat{z}_i}{\partial \hat{t}^{m}} \lesssim 1$ if $m=1$ and $\frac{\partial^{m} \hat{z}_i}{\partial \hat{t}^{m}} \lesssim \rho^{-1}$ if $2\le m \le d_0$. For each $\mathbf{a}$ in the sum, the coefficients of $\mathbf{a}$ are between $1$ and $d_0$, and not all of them are $1$. Therefore,
  $\prod_{\ell=1}^{j} \frac{\partial^{a_\ell} \hat{z}_i}{\partial \hat{t}^{a_\ell}}\lesssim_{d_0} \rho^{-1}$, and
  \begin{equation}\label{eq:pdt-pdz}
    \frac{\partial^n}{\partial \hat{t}^n} = \left(\frac{\partial \hat{z}_i}{\partial \hat{t}} \right)^n  \frac{\partial^n}{\partial \hat{z}_i^n} + \sum_{j=1}^{n-1} O_{d_0}(\rho^{-1}) \frac{\partial^j}{\partial \hat{z}_i^j}.
  \end{equation}
  We apply this to \eqref{eq:log-intk}.
  By (F$_i$), when $1\le n\le d_0$,
  \begin{equation}\label{eq:ptpsi}
    \left|\frac{\partial^{n}}{\partial \hat{t}^{n}}\left[\frac{\partial \hat{f}_i}{\partial \hat{z}_i}\right]\right| \le \left|\left(\frac{\partial \hat{z}_i}{\partial \hat{t}} \right)^n \frac{\partial^{n+1}\hat{f}_i}{\partial \hat{z}_i^{n+1}}\right| + \sum_{j=1}^{n-1} O_{d_0}(\rho^{-1}) \left|\frac{\partial^{j+1}\hat{f}_i}{\partial \hat{z}_i^{j+1}}\right| \lesssim_{d_0} \rho^{-1}.
  \end{equation}
  By (P$_{i,d_0}$), \eqref{eq:log-intk}, and \eqref{eq:ptpsi},
  $$\left|\frac{\partial^{d_0+1} \hat{z}_i}{\partial \hat{t}^{d_0+1}}\right| \lesssim_{d_0}  \sum_{j=1}^{d_0} \left|\frac{\partial^{j} \hat{z}_i}{\partial \hat{t}^{j}}\right| \int_0^{\hat{s}} \left|\frac{\partial^{d_0+1-j}}{\partial \hat{t}^{d_0+1-j}}\left[\frac{\partial \hat{f}_i}{\partial \hat{z}_i}\right]\right|\ud \hat{\sigma} \lesssim_{d_0} \sum_{j=1}^{d_0} |\hat{s}| \rho^{-1} \lesssim_{d_0} \rho^{-1}.$$
  That is, (F$_i$) and (P$_{i,d_0}$) imply (P$_{i,d_0+1}$). By induction, this implies (P$_{i,d}$). Furthermore, the implicit constant in (P$_{i,d_0+1}$) depends only on $d$ and the implicit constant in (P$_{i,d_0}$), so the implicit constant in (P$_{i,d}$) depends only on $d$.
  
  Consequently, \eqref{eq:pdt-pdz} holds for all $1\le n\le d$. Solving the resulting system of equations for $\frac{\partial^n}{\partial \hat{z}^n_i}$, we obtain
  $$\frac{\partial^n}{\partial \hat{z}^n_i} = \left(\frac{\partial \hat{t}}{\partial \hat{z}_i} \right)^n  \frac{\partial^n}{\partial \hat{t}^n} + \sum_{\ell=1}^{n-1} O_d(\rho^{-1}) \frac{\partial^\ell}{\partial \hat{t}^\ell}.$$
  In particular, for $m,k \in\{1,\dots, d\}$,
  \begin{multline*}
    \left|\hat{Z}^k_i \hat{S}^m \hat{\nu}_i\right| \le \left|\left(\frac{\partial \hat{t}}{\partial \hat{z}_i} \right)^k \hat{T}^k \hat{S}^m\hat{\nu}_i \right| + \sum_{\ell=1}^{k-1} O_d(\rho^{-1}) \left|\hat{T}^\ell \hat{S}^m \hat{\nu}_i\right| \\
    \lesssim_d  \left\|\frac{\partial \hat{t}}{\partial \hat{z}_i} \right\|_\infty^k \|\partial_{mk} \kappa\|_\infty + \sum_{\ell=1}^{k-1} \rho^{-1} \|\partial_{m\ell} \kappa\|_\infty\lesssim_d 1.
  \end{multline*}
  This proves (H$_i$), with implicit constant depending only on $d$.
  
  It remains to prove (F$_{i+1}$). Suppose that $\rho>2$.
  By \eqref{eq:first-deriv-f}, we have $\|\hat{Z}_{i+1} \hat{f}_{i+1} \|_\infty \le 2\rho^{-1}$.
  Since
  $$\hat{f}_{i+1} = A r_{i+1}^{-1} \sum_{j=0}^{i} \nu_j = A r_{i+1}^{-1} \sum_{j=0}^{i} A^{-1}r_j\hat{\nu}_j = \sum_{j=0}^{i} \rho^{i+1-j} \hat{\nu}_j = \sum_{m=1}^{i+1} \rho^{m} \hat{\nu}_{i+1-m},$$
  and $\hat{Z}_{i+1} = \rho^{-2m}\hat{Z}_{i+1-m}$, for $2\le k \le d$,
  $$\left\|\hat{Z}^k_{i+1} \hat{f}_{i+1} \right\|_\infty \le \sum_{m=1}^{i+1} \left\|\rho^{-2km}\hat{Z}^k_{i+1-m}[\rho^{m} \hat{\nu}_{i+1-m}]\right\|_\infty \lesssim_d \sum_{m=1}^{i+1} \rho^{(1-2k)m} \le 2 \rho^{1-2k}.$$
  That is, there is a $c>0$ depending only on $d$ such that $\|\hat{Z}^k_{i+1} \hat{f}_{i+1} \|_\infty\le c \rho^{-3}$. We take $\rho>\sqrt{c}$, so that $\|\hat{Z}^k_{i+1} \hat{f}_{i+1} \|_\infty \le c \rho^{-3} \le 2\rho^{-1}$ for all $2\le k \le d$. This proves (F$_{i+1}$). By induction, (F$_i$) and (H$_i$) hold for all $i$.
\subsection{Proof of Lemmas~\ref{lem:induct-eta} and \ref{lem:induct-f}}

First, we prove Lemma~\ref{lem:induct-eta} by rewriting words $D\in \{\hat{\partial}_i,\hat{Z}_i\}$ as sums of operators of the form $\hat{Z}_i^k \hat{\partial}_i^j$. 

\begin{lemma} \label{l:D-reorder}
  Let $d>0$. Suppose that $\|f_i\|_{W_{i,d}'}<A^{-1} r_i$. 

  For any $0\le l<d$, there is a $c_l>0$ such that any word $D\in \{\hat{\partial}_i, \hat{Z}_i\}^{l}$ can be written as
  \begin{equation}\label{eq:std-form-2}
    D = \sum_{j + k \leq l} g_{j,k}(D) \hat{Z}_i^k \hat{\partial}_i^j,
  \end{equation}
  where for all $j$ and $k$, $g_{j,k}(D)$ is a smooth function such that $\|g_{j,k}(D)\|_{W_{i,d-l}}\le c_l$. 
\end{lemma}
We call the right side of \eqref{eq:std-form-2} the \emph{standard form} of $D$ and we call the $g_{j,k}(D)$'s the \emph{coefficients} of $D$.

The following lemma will be helpful in proving Lemma~\ref{l:D-reorder}. Let $\|\cdot\|_{K\to K'}$ denote the operator norm with respect to the norms $K$ and $K'$.
\begin{lemma}\label{lem:op-norm}
  For any $d\ge 0$,
  $\|\hat{Z}_i\|_{W'_{i,d}\to W_{i,d-1}} \le 1$ and $\|\hat{\partial}_i\|_{W_{i,d}\to W_{i,d-1}} \le 1$.
  For $g, h\in W_{i,d}$, $\|gh\|_{W_{i,d}}\lesssim_d \|g\|_{W_{i,d}} \|h\|_{W_{i,d}}$.
\end{lemma}
\begin{proof}
  The operator bounds on $\hat{Z}_i$ and $\hat{\partial}_i$ follow from the definitions of $W_{i,d}$ and $W'_{i,d}$.  Let $D \in \{\hat{Z}_i,\hat{\partial}_i\}^l$ with $0\le l\le d$ and suppose that $D = D_1 \cdots D_l$ where $D_i \in \{\hat{Z}_i,\hat{\partial}_i\}$.  Given a subset $A \subseteq \{1,...,l\}$, we let
  \begin{align*}
    D_A = D_{i_1} \cdots D_{i_j}
  \end{align*}
  where $i_1 < ... < i_j$ are the elements of $A$. By the product rule,
  \begin{align*}
    \|D(gh)\|_\infty &= \left\|\sum_{A \subseteq \{1,...,l\}} D_A(g) D_{A^c}(h)\right\|_\infty 
    \leq 2^d \|g\|_{W_{i,d}} \|h\|_{W_{i,d}}.
  \end{align*}
\end{proof}

We prove Lemma~\ref{l:D-reorder} by induction on $l$.
\begin{proof}[Proof of Lemma~\ref{l:D-reorder}]
  If $l=0$, then $D=\id$, so we can take $g_{0,0}(D)=1$ and $c_l=1$. Let $0\le l <d$ and suppose that the lemma holds for words of length $l$. Let $D\in \{\hat{\partial}_i, \hat{Z}_i\}^{l+1}$. Then $D=\delta D_0$ for some $\delta\in \{\hat{\partial}_i, \hat{Z}_i\}$ and $D_0 \in \{\hat{\partial}_i, \hat{Z}_i\}^{l}$, and there are coefficients $g_{j,k}=g_{j,k}(D_0)\from \HH \to \R$ such that 
  \begin{align*}
    D_0 = \sum_{j + k \leq l} g_{j,k} \hat{Z}_i^k \hat{\partial}_i^j,
  \end{align*}
  where $\|g_{j,k}\|_{W_{i,d-l}} \le c_l$ for all $j$ and $k$.

  First, we consider the case that $\delta=\hat{Z}_i$. Then
  \begin{align}\label{eq:Z-case}
    D= \hat{Z}_i D_0 = 
    \sum_{j + k \leq l} \hat{Z}_ig_{j,k} \cdot  \hat{Z}_i^k \hat{\partial}_i^j + \sum_{j + k \leq l} g_{j,k} \hat{Z}_i^{k+1} \hat{\partial}_i^j.
  \end{align}
  This sum is in standard form, and by Lemma~\ref{lem:op-norm}, $\|\hat{Z}_i g_{j,k}\|_{W_{i,d-l-1}} \le \|g_{j,k}\|_{W_{i,d-l}}\le c_l$. Moreover, $\|g_{j,k}\|_{W_{i,d-l-1}} \le \|g_{j,k}\|_{W_{i,d-l}}\le c_l$. Thus, the lemma holds for words of length $l+1$ that start with $\hat{Z}_i$.

  Second, we consider the case that $D=\hat{\partial}_i D_0$.  We have
  \begin{align}\label{eq:piD}
    D=\hat{\partial}_i D_0 = \sum_{j + k \leq l} \hat{\partial}_i g_{j,k} \cdot \hat{Z}_i^k \hat{\partial}_i^j + \sum_{j + k \leq l} g_{j,k} \hat{\partial}_i \hat{Z}_i^k \hat{\partial}_i^j =: \rI + \rII.
  \end{align}
  Then $\rI$ is already in standard form, and its coefficients satisfy $\|\hat{\partial}_i g_{j,k}\|_{W_{i,d-l-1}} \le \|g_{j,k}\|_{W_{i,d-l}}\le c_l$.
  To write $\rII$ in standard form, we use the identity
  $$\hat{\partial}_i\hat{Z}_i - \hat{Z}_i \hat{\partial}_i = Ar_i^3 [\partial_i,Z] = Ar_i^3 [X + (y - f_i)Z, Z] = Ar_i^3 Zf_i\cdot Z = \hat{Z}_i \hat{f}_i\cdot \hat{Z}_i.$$
  Since $\|f_i\|_{W'_{i,d}}\le A^{-1} r_i$, 
  \begin{equation}\label{eq:hatZf}
    \|\hat{Z}_i \hat{f}_i\|_{W_{i,d-1}} \le \|\hat{f}_i\|_{W'_{i,d}} \le A r_i^{-1} \|f_i\|_{W'_{i,d}} \le 1.
  \end{equation}
  
  Suppose that $E=g_{j,k} \cdot \hat{\partial}_i \hat{Z}_i^k \hat{\partial}_i^j$ is a summand of $\rII$. If $k=0$, there is nothing to do. Otherwise, if $k>0$, then
  $$E = g_{j,k} \cdot \hat{Z}_i\hat{\partial}_i \hat{Z}_i^{k-1} \hat{\partial}_i^j + g_{j,k} \cdot \hat{Z}_i\hat{f}_i \cdot \hat{Z}_i^k \hat{\partial}_i^j.$$
  The first term is a multiple of a word of length at most $l+1$ that starts with $\hat{Z}_i$. By the argument above it can be written in standard form, and by Lemma~\ref{lem:op-norm}, the norms of its coefficients are bounded by a function of $c_l$. The second term is already in standard form, and by Lemma~\ref{lem:op-norm} and \eqref{eq:hatZf}, its coefficient $g_{j,k} \cdot \hat{Z}_i\hat{f}_i$ is bounded by a function of $c_l$.
  
  Thus, $\rII$ can be written as a sum of terms in standard form. The coefficients of each term are bounded by a function of $c_l$ and there are at most $(l+1)^2$ terms, so $D=\rI+\rII$ can be written in standard form, with coefficients bounded by some $c_{l+1}$ that depends only on $l$.
\end{proof}

Lemma~\ref{lem:induct-eta} follows directly.
\begin{proof}[Proof of Lemma~\ref{lem:induct-eta}]
  Let $D \in \{\hat{Z}_i,\hat{\partial}_i\}^l$ for $l \leq d$.  By Lemma \ref{l:D-reorder}, we can write $D$ in standard form as
  \begin{align*}
    D = \sum_{j + k \leq l} g_{j,k} \hat{Z}_i^k \hat{\partial}_i^j,
  \end{align*}
  where $\|g_{j,k}\|_\infty \lesssim_d 1$.
  Then, by Lemma~\ref{lem:standard-form-bounds},
  \begin{multline*}
    \|D \nu_i\|_\infty = \left\| \sum_{j + k \leq l} g_{j,k} \hat{Z}_i^k \hat{\partial}_i^j \nu_i \right\|_\infty \leq \sum_{j + k \leq l} \|g_{j,k}\|_\infty \|\hat{Z}_i^k \hat{\partial}_i^j \nu_i\|_\infty \\ \overset{\eqref{e:eta-sobolev-bound}}{\lesssim_d} \sum_{j + k \leq l} A^{-i} r_i\lesssim_d A^{-1}r_i.
  \end{multline*}
\end{proof}

Finally, we prove Lemma~\ref{lem:induct-f}.
\begin{proof}[Proof of Lemma~\ref{lem:induct-f}]
  Note that
  \begin{equation}\label{eq:z-rescale}
    \hat{Z}_{i+1} = r_{i+1}^2 Z = \rho^{-2} r_{i}^2 Z = \rho^{-2}\hat{Z}_i
  \end{equation}
  and
  \begin{equation}\label{eq:p-rescale}
    \hat{\partial}_{i+1} = A r_{i+1}\partial_{i+1} = \rho^{-1} A r_{i}(\partial_i - \nu_{i} Z) = \rho^{-1}(\hat{\partial}_{i} - \hat{\nu}_i \hat{Z}_i).
  \end{equation}

  Let $0\le l\le d$ and let $D\in \{\hat{\partial}_{i+1}, \hat{Z}_{i+1}\}^l$ be a word of length $l$ such that $D\not\in \{\id, \hat{\partial}_{i+1}\}$. Let $n = 2 \# \hat{Z}_{i+1}(D) + \# \hat{\partial}_{i+1}(D)$ and note that $n \ge 2$.
  We replace $\hat{Z}_{i+1}$ by $\rho^{-2} \hat{Z}_{i}$ and $\hat{\partial}_{i+1}$ by $\rho^{-1}(\hat{\partial}_{i} - \hat{\nu}_i \hat{Z}_i)$ and distribute to get an expression $D = \rho^{-n} \sum_{j=1}^m \pm D_j$, where $D_j\in \{\hat{\partial}_{i}, \hat{\nu}_i, \hat{Z}_i\}^*$ for each $j=1,\dots,m$ and $m\le 2^l$. Furthermore, $l\le \ell(D_j) \le 2l$ for all $j$.
  
  If $\ell(D)=1$, then $D=\hat{Z}_{i+1}$, so $m=1$ and $D_1=\hat{Z}_{i}$. Otherwise, $\ell(D)\ge 2$ and $\ell(D_j)\ge 2$. Since every $\hat{\nu}_i$ in $D_j$ is followed by $\hat{Z}_i$, if $D_j$ ends in $\hat{\partial}_i$, then the previous letter is either $\hat{\partial}_i$ or $\hat{Z}_i$. That is, we can write $D_j=D'_j E_j$, where $E_j=\hat{Z}_i$, $E_j=\hat{\partial}_i^2$, or $E_j=\hat{Z}_i \hat{\partial}_i$. 
  
  Since $E_j\not\in \{\id, \hat{\partial}_i\}$, we have $\|E_j\|_{W'_{i,d}\to W_{i,d-\ell(E_j)}}\le 1$. By Lemma~\ref{lem:op-norm}, for any $0\le k\le d$, we have $\|\hat{\partial}_{i}\|_{W_{i,k}\to W_{i,k-1}} \le 1$,  $\|\hat{Z}_{i}\|_{W_{i,k}\to W_{i,k-1}} \le 1$, and
  $\|\hat{\nu}_{i}\|_{W_{i,k}\to W_{i,k}} \lesssim_d \|\hat{\nu}_{i}\|_{W_{i,d}}.$
  Let $L \lesssim_d 1 + \|\hat{\nu}_{i}\|_{W_{i,d}}$ be such that each letter of $D_j$ has operator norm at most $L$. Then
  $$\|D_j\|_{W'_{i,d}\to L_\infty} \le \|E_j\|_{W'_{i,d}\to W_{i,d-\ell(E_j)}} \|D_j\|_{W_{i,d-\ell(E_j)} \to W_{i,d-l}} \lesssim_d L^d.$$
  Therefore,
  $$\|D g\|_\infty \le \rho^{-n} \sum_{j=1}^m \|D_j\|_{W'_{i,d}\to L_\infty} \|g\|_{W'_{i,d}} \lesssim \rho^{-2} L^d \|g\|_{W'_{i,d}}.$$
  Since this holds for all $D\in \{\hat{\partial}_{i+1}, \hat{Z}_{i+1}\}^*$ such that $\ell(D)\le d$ and $D\not\in\{\id,\hat{\partial}_{i+1}\}$,
  $$\|g\|_{W'_{i+1,d}} \lesssim \rho^{-2} L^d \|g\|_{W'_{i,d}} \lesssim_d (1 + \|\hat{\nu}_i\|_{W_{i,d}})^d \rho^{-2}\|g\|_{W'_{i,d}},$$
  as desired.
\end{proof}
\bibliographystyle{alpha}
\bibliography{riesztr}
\end{document}